\documentclass[12pt]{amsart}

\usepackage[numbers,sort]{natbib}
\usepackage[margin=1in,letterpaper,portrait]{geometry}
\usepackage{longtable}
\usepackage{amsmath}
\usepackage{amssymb}
\usepackage{amsthm}
\usepackage{mathrsfs}
\usepackage{verbatim}
\usepackage{fancyhdr}
\usepackage{url}
\usepackage{xcolor}
\usepackage{tikz}
\usepackage{enumerate}
\usepackage{cleveref}
\usetikzlibrary{matrix,arrows,shapes,decorations.markings}
\usepackage{tikz-cd}
\usepackage{rotating}
\usepackage{mathabx,epsfig}

\usetikzlibrary{plotmarks}
\usetikzlibrary{matrix, arrows}
\usepackage{pgfplots}
\pgfplotsset{compat=1.14}
\usepgfplotslibrary{fillbetween}

\usepackage{setspace}

\newcommand{\mb}{\mathbb}

\newcommand{\mc}{\mathcal}
\newcommand{\im}{{\rm im}}
\newcommand{\ol}{\overline}
\newcommand{\wt}{\widetilde}
\newcommand{\codim}{{\rm codim}}

\newcommand{\Spec}{{\rm Spec}}

\newcommand{\Hom}{{\rm Hom}}
\newcommand{\ms}{\mathscr}
\newcommand*{\sheafhom}{\mathscr{H}\kern -.5pt om}

\newcommand{\ev}{{\rm ev}}

\newcommand{\pt}{{\rm pt}}
\newcommand{\rk}{{\rm rk}}
\newcommand{\MM}{{\mathcal M}}
\newcommand{\NN}{{\mathcal N}}

\newcommand{\Vt}{{\operatorname{Vert}}}
\newcommand{\Gr}{{\operatorname{Gr}}}
\newcommand{\ord}{{\operatorname{ord}}}
\newcommand{\conv}{{\operatorname{conv}}}

\theoremstyle{plain}
\newtheorem{Thm}{Theorem}[section]
\newtheorem{Thm*}{Theorem}
\crefname{Thm}{Theorem}{Theorems}
\crefname{Thm*}{Theorem}{Theorems}
\newtheorem{Cor}[Thm]{Corollary}
\crefname{Cor}{Corollary}{Corollaries}

\crefname{Conj}{Conjecture}{Conjectures}
\newtheorem{Pro}[Thm]{Problem}
\newtheorem{Pro*}{Problem}
\crefname{Pro}{Problem}{Problems}

\crefname{Main}{Main Theorem}{Main Theorems}

\newtheorem{Lem}[Thm]{Lemma}
\crefname{Lem}{Lemma}{Lemmas}
\newtheorem{Claim}[Thm]{Claim}
\crefname{Claim}{Claim}{Claims}
\newtheorem{Prop}[Thm]{Proposition}
\crefname{Prop}{Proposition}{Propositions}
\newtheorem{Prop*}{Proposition}

\crefname{ToDo}{To Do}{To Dos}

\crefname{Fact}{Fact}{Facts}
\newtheorem*{Fact*}{Fact}

\theoremstyle{definition}
\newtheorem{Def}{Definition}[section]
\crefname{Def}{Definition}{Definitions}

\crefname{Exer}{Exercise}{Exercises}
\newtheorem{Rem}[Thm]{Remark}
\crefname{Rem}{Remark}{Remarks}
\newtheorem{Exam}{Example}
\crefname{Exam}{Example}{Examples}

\theoremstyle{remark}
\usepackage{twoopt}
\newcommandtwoopt{\op}[3][r][n] {\mathcal{O}_{#3}}
\newcommandtwoopt{\opc}[3][r][n] {[\mathcal{O}_{#3}]}
\newcommandtwoopt{\og}[3][r+1][n] {\mathcal{O}(#3)_{\textrm{G}(#1,#2)}}
\newcommandtwoopt{\ogc}[3][r+1][n] {[\mathcal{O}(#3)]_{\textrm{G}(#1,#2)}}
\newcommandtwoopt{\oa}[3][(r+1)][n] {\mathcal{O}(#3)_{\mb{A}^{#1\times #2}}}
\newcommandtwoopt{\oac}[3][(r+1)][n] {[\mathcal{O}(#3)]_{\mb{A}^{#1 \times #2}}}

\begin{document}

\title{Orbits in $(\mathbb{P}^r)^n$ and equivariant quantum cohomology}
\author{Mitchell Lee, Anand Patel, Hunter Spink, Dennis Tseng}
\date{\today}

\allowdisplaybreaks

\begin{abstract}
We compute the $GL_{r+1}$-equivariant Chow class of the $GL_{r+1}$-orbit closure of any point $(x_1, \ldots, x_n) \in (\mb{P}^r)^n$ in terms of the rank polytope of the matroid represented by $x_1, \ldots, x_n \in \mb{P}^r$. Using these classes and generalizations involving point configurations in higher dimensional projective spaces, we define for each $d\times n$ matrix $M$ an $n$-ary operation $[M]_\hbar$ on the small equivariant quantum cohomology ring of $\mb{P}^r$, which is the $n$-ary quantum product when $M$ is an invertible matrix. We prove that $M \mapsto [M]_\hbar$ is a valuative matroid polytope association.

Like the quantum product, these operations satisfy recursive properties encoding solutions to enumerative problems involving point configurations of given moduli in a relative setting. As an application, we compute the number of line sections with given moduli of a general degree $2r+1$ hypersurface in $\mb{P}^r$, generalizing the known case of quintic plane curves.
\end{abstract}

\maketitle

\section{Introduction}
For any point $X=(x_1,\ldots,x_n)\in(\mb{P}^r)^n,$ we can consider its orbit closure $\overline{GL_{r+1}\cdot X}$ under the diagonal $GL_{r+1}$-action on $(\mb{P}^r)^n$. The Chow class of this subvariety in $A^\bullet((\mb{P}^r)^n)$ was computed by Li \cite{Binglin}, and previously studied by many others \cite{AF93,AF00,Serdica,Kapranov}. The point $X=(x_1,\ldots,x_n)$ determines a \emph{rank function} which associates to a subset $S\subset \{1,\ldots,n\}$ one plus the dimension of the linear span of $\{x_s\}_{s \in S}$ in $\mb{P}^r$, known to combinatorialists as a \emph{representable matroid}. Li's result shows this combinatorial data determines the Chow class of the orbit closure. 

Given a rank $r+1$ vector bundle $V\to B$, one can form a relative orbit closure $(\ol{GL_{r+1}\cdot X})_V\subset \mb{P}(V)^n$ that restricts to $\overline{GL_{r+1}\cdot X}$ in each fiber of $\mb{P}(V)^n\to B$. One can clearly do this locally on a trivialization, and this patches together as the transition functions preserve the orbit closure. It turns out that the Chow class of the relative orbit closure in $A^{\bullet}(\mb{P}(V)^n)$ is a universal polynomial in the $n$ $\mathscr{O}(1)$-classes $H_1,\ldots,H_n$ and the $r+1$ Chern classes of $V$ pulled back from $B$. This polynomial is the $GL_{r+1}$-\emph{equivariant Chow class} of $\overline{GL_{r+1}\cdot X}$. That such a universal polynomial exists follows from the foundations of equivariant intersection theory \cite{EG98}.

The previous approach to studying $\overline{GL_{r+1}\cdot X}$ in the literature involves viewing $X$ as an $(r+1)\times n$ matrix, and taking the row span $\Lambda\in G(r+1,n)$, provided $X$ has full row rank. The data of the $GL_{r+1}$-equivariant Chow class of $\overline{GL_{r+1}\cdot X}$ is then in principle equivalent to the data of the $T$-equivalent Chow class of $\overline{T\cdot \Lambda}$, where $T$ is the $n$-dimensional torus acting on $G(r+1,n)$ (see \cite{Fink} and \Cref{liftingsection}). Classes of torus orbits in the Grassmannian have been widely studied \cite{Kapranov,SpeInvariant,FS12,berget:hal-01283166}, but it is unclear how to translate these formulas to our setting. 

In this paper, we will take a completely different approach to studying these orbit closures, which has connections to equivariant quantum cohomology. In particular, we compute the $GL_{r+1}$-equivariant classes of all $GL_{r+1}$-orbit closures in $(\mb{P}^r)^n$. There are two main halves to the argument. The first half is to use $GL_{r+1}$-equivariant degenerations to reduce the case of an arbitrary $X$ to the case where the matroid of $X$ is ``series-parallel''. In particular, we show the $GL_{r+1}$-equivariant class depends only on the matroid of $X$. The second half is to show the $GL_{r+1}$-equivariant classes behave well under parallel and series connection of matroids, allowing us to compute the $GL_{r+1}$-equivariant class when the matroid of $X$ is series-parallel. 

In fact, the $GL_{r+1}$-equivariant Chow classes associated to the series and parallel connections of two matroids are computed from the classes associated to the original matroids via the small equivariant quantum product $\star$ of $\mb{P}^r$ (\Cref{serparintro}). This is fundamental to the connection we describe between orbit closures and quantum cohomology. This paper is self-contained from the perspective of equivariant quantum cohomology of $\mb{P}^r$ (see \Cref{QCRD} and \Cref{KSI}). However we rely heavily on the combinatorics of (matroid) polytopes, the relevant facts of which are collected in \Cref{matroidpolytopedef} and \Cref{polyapp}. 


Our approach has several advantages. 
\begin{enumerate}
    \item Even though the case of $GL_{r+1}$-orbits is easiest to state, our method naturally extends to the $GL_{r+1}$-invariant subvarieties given as the image of matrix multiplication (more precisely the pushforward of the fundamental cycle of the source after resolution, see \Cref{SchubertDef}):
    \begin{align*}
        \mu_M: \mb{P}^{d\times (r+1)-1}\dashrightarrow (\mb{P}^r)^n,
    \end{align*}
    where $M$ is a $d\times n$ matrix, which takes the projectivization of a $d\times (r+1)$ matrix $A$, and outputs the projectivized columns of $AM$. These are $GL_{r+1}$-orbit closures when $d\leq r+1$ and can be studied via torus orbits in the Grassmannian $G(r+1,n)$ when $d=r+1$. See \Cref{KGO} for a situation where these cycles with $d>r+1$ naturally occur. Our formulas are identical in this more general setting (see \Cref{finkformulaintro}).
    \item Equivalent to the data of the $GL_{r+1}$-equivariant class $[\overline{GL_{r+1}\cdot X}]$ is the Kronecker dual
    \begin{align*}
        [\overline{GL_{r+1}\cdot X}]^{\dagger}: A_{GL_{r+1}}^{\bullet}((\mb{P}^r)^n)\to A_{GL_{r+1}}^{\bullet}(\pt)
    \end{align*}
   defined by $$\alpha\mapsto \int\alpha\cap [\overline{GL_{r+1}\cdot X}].$$ When $\alpha$ comes from a $GL_{r+1}$-invariant subvariety $Z$ of $(\mb{P}^r)^n$, we can always intersect the relative versions of $Z$ and $\overline{GL_{r+1}\cdot X}$ in $\mb{P}(V)^n$ and pushforward to $B$ for every rank $r+1$ vector bundle $V\to B$. This yields a universal polynomial in the Chern classes of $V$, encoded by the class $\int\alpha\cap \overline{GL_{r+1}\cdot X}\in A^\bullet_{GL_{r+1}}(\pt)$.
   
   Our method shows that the map $[\overline{GL_{r+1}\cdot X}]^{\dagger}$ is often more natural to describe than the $GL_{r+1}$-equivariant class itself, and can be computed in terms of certain polynomial expansions and truncations (see \Cref{extendedexample}). This is particularly pleasant in the case where the matroid of $X$ is a Schubert matroid (see \Cref{Schubertexample} for a representative example). Our expressions for $[\overline{GL_{r+1}\cdot X}]^{\dagger}$ are motivated by the small quantum cohomology of projective space, and we have a self-contained expository account of this connection in \Cref{KSI} in the special case of a uniform matroid. 
   \item 
   The nonequivariant Chow class of $\overline{GL_{r+1}\cdot X}$ computed by Li \cite{Binglin} is a multiplicity-free sum of monomials in the $n$ hyperplane classes $H_i$, where the monomials correspond to integral points inside a polytope, specifically a rescaling and translation of the rank polytope of the matroid of $X$ (see \Cref{noneqsection}). Chow classes of torus orbits in the Grassmannian seem to be more complicated when expanded in terms of Schubert varieties, even non-equivariantly.
   
   Our method naturally yields the equivariant Chow class of $\overline{GL_{r+1}\cdot X}$ as an equivariant version of a sum of integral points of a polytope. Specifically, an application of Brion's theorem \cite{BHS} allows us to encode Li's result via a unique $2n$-variable rational function $f_M(x_1,\ldots,x_n,y_1,\ldots,y_n)$ depending only on $M$ and not on $r$ which has the property that $f_M(z_1,\ldots,z_n,z_1^{r+1},\ldots,z_n^{r+1})$ is a polynomial, and this polynomial specialized to $z_i=H_i$ recovers the non-equivariant class. Denoting by $F(z)=z^{r+1}+c_1z^r+\cdots+c_{r+1}$, we show that $f_M(z_1,\ldots,z_n,F(z_1),\ldots,F(z_n))$ is a polynomial, which when specialized to $z_i=H_i$ recovers the equivariant class. See the beginning of \Cref{statement} up to \Cref{finkformulaintro} for a precise formulation.
   \item 
   Given a $d\times n$ matrix $M$, we can consider the equivariant Chow class associated to the $k\times n$ matrix we get by applying a generic projection $K^d\to K^k$ to all of the columns of $M$. These classes don't depend on the choices of generic projections, and we will package these classes over all $k\leq d$ into an object we will call $[M]_{\hbar}$, which satisfies a fundamental relation 
   $$[M\oplus N]_{\hbar}=[M]_{\hbar}\star[N]_{\hbar}.$$
   Here, the matrix direct sum $M\oplus N$ is the block matrix $\begin{pmatrix} M & 0 \\ 0 & N\end{pmatrix}$. The $\star$ operation can be computed cleanly via the small quantum product of projective space, yielding elegant formulas relating Kronecker dual classes associated to column projections of $M$ and $N$ to Kronecker dual classes associated to column projections of $M\oplus N$. For example, we can use the fundamental relation to compute $[I_n]_{\hbar}$ inductively, where $I_n$ is the $n\times n$ identity matrix. This yields all Kronecker dual classes associated to uniform matroids. With only slightly more work, the same argument yields the nice formula for Kronecker dual classes associated to Schubert matroids mentioned in the second item above.
   
   The object $[M]_\hbar$ will actually be an operator taking $n$ classes in $A^\bullet_{GL_{r+1}}(\mb{P}^r)$, and outputting a class in the equivariant quantum cohomology ring $QH^\bullet_{GL_{r+1}}(\mb{P}^r)$. This ring is a deformation of $A^\bullet_{GL_{r+1}}(\mb{P}^r)$ by a single parameter $\hbar$, which we describe in detail in \Cref{NotandConv}. The construction of $[M]_\hbar$ is tightly paralleled by the construction of small equivariant quantum cohomology, and in fact, by construction, $[I_n]_{\hbar}$ is precisely the $n$-ary quantum product
$$(\alpha_1,\ldots,\alpha_n)\mapsto \alpha_1 \star \ldots \star \alpha_n.$$
The relation $[M\oplus N]_{\hbar}=[M]_{\hbar}\star [N]_{\hbar}$ can then be seen as a generalization of the associativity relations of small equivariant quantum cohomology. See \Cref{Mhbardef} for a precise definition of $[M]_\hbar$ and \Cref{Mhbarintro} for a precise theorem.
\end{enumerate}
As a consequence of our approach, we will compute all $GL_{r+1}$-equivariant classes of $GL_{r+1}$-orbit closures in $(\mb{P}^r)^n$ (see \Cref{allorbits}). This solves the problem from \cite{Fink} of computing $GL_{r+1}\times T$-equivariant classes of $GL_{r+1}\times T$-orbit closures in $\mb{A}^{(r+1)\times n}$ (see \Cref{liftingsection}).

Finally, we mention that these Chow classes associated to matrices have applications to enumerative geometry. For example, when the image of matrix multiplication $\mu_M$ happens to be a divisor, one can construct (virtual) effective divisors on $\ol{\mathscr{M}_g}$ \cite[Example 1.3]{T19}. For a second example, fix a hypersurface $Y\subset \mb{P}^r$ and a general hypersurface $Z\subset \mb{P}^{r+a}$.
\begin{Pro*}
\label{slicing}
As we vary over $r$-planes $\Lambda\subset\mb{P}^{r+a}$, how many times is $\Lambda\cap Z\cong Y$? 
\end{Pro*}
\Cref{slicing} has been solved in the context of line slices of a quintic plane curve \cite{Laza} and planar slices of a quartic threefold \cite[Corollary 1.3]{LPT19} for many choices of $Y$. To our knowledge, \Cref{slicing} is still open for hyperplane slices of a cubic threefold, even for general $Y$. A further generalization of \Cref{slicing} is to restrict $\Lambda$ to vary inside a subvariety of the Grassmannian, for example in a Schubert variety. 

The formulas in this paper solve \Cref{slicing} in the case where $Y$ is any union of hyperplanes. We demonstrate this in the case of line slices of a degree $2r+1$ hypersurface in $\mb{P}^r$, deriving an explicit formula and generalizing the example of quintic plane curves in \cite{Laza} (see \Cref{APIntro,APDIntro}). 
\subsection{Relation to Previous Work}
The map $\mu_M$ has been studied under various guises, both classically \cite{AF93, Binglin2, Kapranov, Binglin, Serdica} and $GL_{r+1}$-equivariantly \cite{berget:hal-01283166,BergetPaper,Fink,Harm,FS12,SpeInvariant}.

There is an analogous problem, studied in \cite{KMS,KMY09}, of computing $T\times T$- equivariant Chow classes of $B\times B$-orbit closures in the space of $n\times n$ matrices, lifting the $T$-equivariant classes of Schubert cycles in $GL_n/B$. Because of a lifting lemma implied by this result deduced in \cite{Fink} and a lifting lemma from \cite{FR07} (see \Cref{liftingsection}), a special case of our results yield the $GL_{r+1} \times T$-equivariant classes of $GL_{r+1}\times T$-orbit closures in $\mb{A}^{(r+1)\times n}$. Since this paper has been disseminated, \cite{BF19} has computed the $GL_{r+1}\times T$-equivariant K theory classes of $GL_{r+1}\times T$-orbit closures in $\mb{A}^{(r+1)\times n}$, which implies this special case of our results.

Small equivariant quantum cohomology, pioneered by Givental and Kim \cite{Kim93}, yields in a relative setting operations encoding counts of rational curves which intersect  a collection of subvarieties with fixed moduli of intersection. The operation $[M]_{\hbar}$ we construct tightly parallels this classical construction, and has analogous recursive and enumerative properties.

We remark as a consequence of our results that $[M]_\hbar(\exp(d_1H_1+\ldots+d_nH_n))$ computes an equivariant generalization to the adjusted predegree polynomial of hyperplane configurations with multiplicities $d_i$ as defined in \cite{AF00, Serdica}, and explains their truncated exponential formulas.

The degenerations we consider induce subdivisions of matroid rank polytopes. Such subdivisions have been studied in related contexts \cite{Derksen,GelfandZelevinsky,Kapranov,SpeThesis,SpeInvariant}, and in particular this extends the connection of Gelfand, Zelevinsky and Kapranov \cite{GelfandZelevinsky} between regular subdivisions of moment polytopes and degenerations of toric varieties to our situation.

Of interest is that our sequence of degenerations yields an algorithm providing a converse to the statement that degenerations of orbits yield additive relations in the sense of Derksen and Fink \cite{Derksen}. This is one of the main technical challenges we overcome, relying heavily on the geometry of infinite polytopal regions, and appears to have been a gap in understanding the geometry of realizable matroid polytope subdivisions.

Finally, we note that the orbits of ordered point configurations have also appeared under the guise of hyperplane arrangements by taking duals \cite{ Alexeev, Hacking,Kapranov, Keel}.

 \subsection{Acknowledgements}
 We would like to thank Alex Fink for many helpful discussions on polytope subdivisions. We would also like to thank Paolo Aluffi and Andrew Berget for helpful discussions. Finally, we would like to thank Joe Harris, Davesh Maulik, Richard Rim\'anyi, and Jenia Tevelev for answering questions in the initial stages of this project, and the anonymous referee for greatly improving the exposition of the paper.

\section{Definitions}\label{NotandConv}
 In what follows, we fix an algebraically closed field $K$ of arbitrary characteristic. If $v$ is a nonzero element in a vector space $V$ we write $\wt{v}$ for the image of $v$ in $\mb{P}(V)$. We denote the projectivization of the vector space $\mb{A}^{k \times \ell}$ of $k \times \ell$ matrices by $\mb{P}^{k \times \ell-1}$.
 
To help the reader locate relevant symbols, we have included the tables \Cref{wordtable} and \Cref{symboltable} in \Cref{tableappendix}.

\subsection{Matrices, matroids and polytopes}
\label{matroidpolytopedef}
Let $d$ and $n$ be integers. For a polytope $P$, denote by $\Vt(P)$ the vertices of $P$. The \emph{hypersimplex} $\Delta_{d,n} \subset \mb{R}^n$ is defined by $$\Delta_{d,n}=\left\{(x_1, \ldots, x_n) \in [0, 1]^n\middle|\sum_{i=1}^n x_i = d\right\}.$$ When $d$ and $n$ are clear from context, we sometimes write the hypersimplex as $\Delta$ rather than $\Delta_{d, n}$. The hypersimplex is the convex hull of the vectors $e_A := \sum_{i \in A} e_i$, where $A$ ranges over all $d$-element subsets of $\{1, \ldots, n\}$. We will sometimes abuse notation and use the same symbol to denote both $e_A$ and $A$. 

Given a $d \times n$ matrix $M$, we define the \emph{rank function} $\rk_M:2^{\{1,\ldots,n\}} \to \mb{Z}_{\ge 0}$ by $\rk_M(A)=\rk(M^A)$, where $M^A$ is the submatrix of $M$ formed by the columns in $A$. When $M$ is clear from context, we sometimes write the rank function as $\rk$ rather than $\rk_M$. The \emph{matroid} of $M$ is the data of the rank function $\rk_M$, which is determined by the maximal independent subsets of the columns of $M$. The \emph{rank polytope} of $M$ is the set
\[P_M=\left\{(x_1, \ldots, x_n) \in \Delta_{d,n}\middle|\text{$\sum_{i \in A} x_i \leq \rk_M(A)$ for all $A \subset \{1, \ldots, n\}$}\right\}\] and the \emph{independence polytope} of $M$ is the set
\[I_M=\left\{(x_1, \ldots, x_n) \in \mb{R}_{\geq 0}^n\middle|\text{$\sum_{i \in A} x_i \leq \rk_M(A)$ for all $A \subset \{1, \ldots, n\}$}\right\}.\]
 The rank polytope of $M$ is the convex hull of $e_A$ where $A$ ranges over all $d$-element subsets of $\{1, \ldots, n\}$ such that $\rk_M(A) = d$, and the independence polytope of $M$ is the convex hull of $e_A$ where $A$ ranges over all subsets of $\{1, \ldots, n\}$ such that $\rk_M(A) = |A|$.
The rank polytope of $M$ is empty if $\rk(M) < d$ and determines the matroid of $M$ if $\rk(M)=d$. The rank polytope of a general $d \times n$ matrix is $\Delta_{d, n}$. 

As in \cite{Derksen}, we say that a function $f$ from $d \times n$ matrices to an abelian group $Z$ is \emph{valuative} (respectively \emph{additive}) if whenever $\sum a_i1_{P_{M_i}}=0$ (respectively $\sum a_i 1_{P_{M_i}}$ is supported on a positive codimension subset of $\{\sum_{i=1}^n x_i=d\} \subset \mb{R}^n$ ), we have $\sum a_i f(M_i)=0$. These notions have been intensely studied in a variety of contexts \cite{Harm,Derksen,SpeInvariant}.

The following proposition shows we can replace the condition $\sum a_i1_{P_{M_i}}=0$ with the condition $\sum a_i1_{I_{M_i}}=0$, which will be convenient at times.

\begin{Prop}\label{rankind}
For $d\times n$ matrices $M_i$ of rank $d$, $\sum a_i1_{P_{M_i}}=0$ if and only if $\sum a_i1_{I_{M_i}}=0$.
\end{Prop}
\begin{proof}
If $\sum a_i1_{I_{M_i}}=0$, then we can restrict to the hyperplane $\sum_{i=1}^n x_i=d$ to obtain $\sum a_i1_{I_{M_i}}=0$. Conversely, suppose $\sum a_i1_{P_{M_i}}=0$. We recall that if $\{P_i\}$ is a collection of possibly unbounded polytopal regions, and $\sum a_i 1_{P_i}=0$, then $\sum a_i 1_{R+P_i}=0$ for any ray $R$ (this can be checked on each ray parallel to $R$ separately). Applying this successively to $\sum a_i 1_{P_{M_i}}=0$ with each of the negative coordinate rays and then restricting to the positive orthant yields $\sum a_i 1_{I_{M_i}}=0$.
\end{proof}

Parallel and series connection of matroids play a crucial role in this paper, and we recall the definition for matrices from \cite[page 262, Figure 7.6(b)]{Oxley}. First, to simplify notation, we define the direct sum of two matrices.

\begin{Def}
For matrices $M_1,M_2$, define the \emph{direct sum} matrix $M_1\oplus M_2=\begin{pmatrix} M_1 & 0 \\ 0 & M_2\end{pmatrix}$.
\end{Def}

Let $M_1 \in K^{d_1\times n_1}$ and $M_2 \in K^{d_2 \times n_2}$. We now define the ``parallel connection matrix'' $P(M_1, M_2)$ and ``series connection matrix'' $S(M_1, M_2)$ (along the last column of $M_1$ and the first column of $M_2$). \Cref{serpardefn} depends on certain non-canonical isomorphisms, so it only defines $P(M_1, M_2)$ and $S(M_1, M_2)$ up to left $GL$-action.

\begin{Def}
\label{serpardefn}
Let $W_1 = K^{n_1}$ and $W_2 = K^{n_2}$. Let $x_{1, 1}, \ldots, x_{1, n_1} \in W_1$ be the columns of $M_1$ and let $x_{2, 1}, \ldots, x_{2, n_2} \in W_2$ be the columns of $M_2$. The \emph{parallel connection matrix} $P(M_1, M_2) \in K^{(d_1 + d_2 - 1) \times (n_1 + n_2 - 1)}$ is the matrix with columns $x_{1, 1}, \ldots, x_{1, n_1} = x_{2, 1}, \ldots, x_{2, n_2}$, considered as elements of $(W_1 \oplus W_2) / \langle x_{1, n_1} - x_{2, 1}\rangle \cong K^{n_1 + n_2 - 1}$. The \emph{series connection matrix} $S(M_1, M_2) \in K^{(d_1 + d_2) \times (n_1 + n_2 - 1)}$ is the matrix with columns $x_{1, 1}, \ldots, x_{1, n_1 - 1}, x_{1, n_1} + x_{2, 1}, x_{2, 2}, \ldots, x_{2, n_2}$, considered as elements of $W_1 \oplus W_2 \cong K^{n_1 + n_2}$.
\end{Def}

 If we consider a matrix as a collection of points in projective space, one point for each column, then series and parallel connection have intuitive geometric meanings. 
\begin{center}
    \includegraphics[scale=.4]{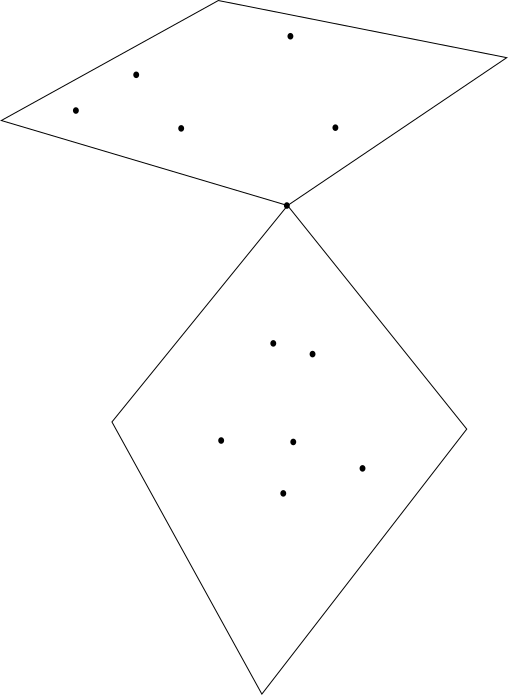}
    \includegraphics[scale=.4]{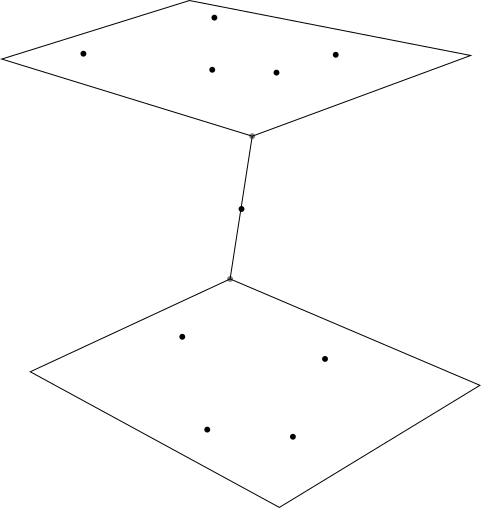}  
\end{center}
Parallel connection (pictured on the left) corresponds to identifying two of the points. Series connection (pictured on the right) corresponds to drawing a line through two of the points and replacing them with a third (distinct) point on the connecting line. 

\begin{Def}\label{serpardef}
A \emph{series-parallel} matrix is a matrix obtained from a general $1 \times 1$ matrix by permutation of columns and series and parallel connection with general $1 \times 2$ matrices.
\end{Def}

\begin{center}
        \includegraphics[scale=.4]{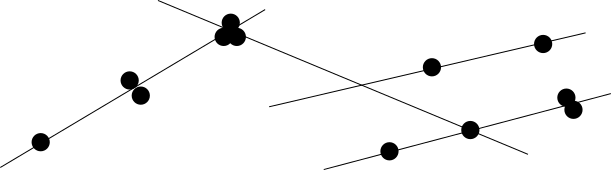}
\end{center}

Considering matrices as projective vector configurations, parallel connection with a general $1 \times 2$ matrix corresponds to duplicating a vector, and series connection with a general $1 \times 2$ matrix corresponds to removing a vector and sprouting off two general points on the line.

\begin{Def}
\label{asttau}
We denote by $\ast$ a non-zero $1\times 1$ matrix.

We denote by $\tau^{\le k}$ to be a general matrix with $k$ rows (where the number of columns will always be clear from context), a different general matrix every time it appears. 
\end{Def}

\begin{Def}\label{SchubertDef}
A \emph{Schubert matroid} $\operatorname{Sch}(r_1,\ldots,r_\ell,X_1,\ldots,X_\ell)$ with $\emptyset \subseteq X_1 \subseteq  \ldots \subseteq X_\ell=\{1,\ldots,n\}$ and $0 \le r_1 \le \ldots \le r_\ell = d$ is the matroid associated to a $d \times n$ matrix $M$ with the first $r_i$ entries of every column in $X_i\setminus X_{i-1}$ being general and the remaining entries being zero. Up to left $GL$-action and column permutation, such a matrix can be written as
$$\tau^{\le r_\ell}(\tau^{\le r_{\ell-1}}(\ldots\tau^{\le r_2}(\tau^{\le r_1}(\ast \oplus \ldots \oplus \ast)\oplus \ast \oplus \ldots \oplus \ast)\ldots)\oplus \ast \oplus \ldots \oplus \ast),$$
where the $\ast$'s in the $i$'th block from the left are associated to the indices in $X_i \setminus X_{i-1}$.
\end{Def}
In terms of point configurations, a Schubert matroid is given by fixing a flag, and picking general points in each member of the flag.
\begin{center}
        \includegraphics[scale=.6]{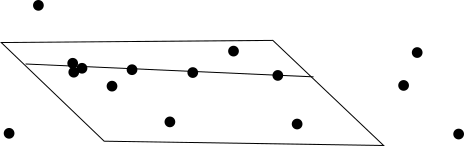}
\end{center}
\subsection{$GL_{r+1}$-equivariant Chow rings}
\label{Chowring}
The standard references for equivariant intersection theory are \cite{EG98,EG98b}, see also \cite{A12} for a self-contained exposition. Let $t_0,\ldots,t_r:(K^\times)^{r+1}\to K^\times$ be the standard characters of the maximal torus $(K^\times)^{r+1}$ of diagonal matrices in $GL_{r+1}$. The $GL_{r+1}$-equivariant Chow ring of a point is $A^\bullet_{GL_{r+1}}(\pt) \cong \mb{Z}[t_0,\ldots,t_r]^{S_{r+1}}$ \cite[Section 3.2]{EG98}. All tensor products $\otimes$ of equivariant Chow rings will be over $A^\bullet_{GL_{r+1}}(\pt)$.

We only consider varieties with algebraic cell decompositions, so $GL_{r+1}$-equivariant Chow rings in this paper embed into the corresponding $(K^\times)^{r+1}$-equivariant Chow rings via the inclusion $\mb{Z}[t_0,\ldots,t_n]^{S_{r+1}} \hookrightarrow \mb{Z}[t_0,\ldots,t_n]$. Thus we may treat the $t_i$ as separate variables.

The $GL_{r+1}$-equivariant Chow ring of projective space is $$A^\bullet_{GL_{r+1}}(\mb{P}^r) =  A^\bullet_{GL_{r+1}}(\pt)[H]/(F(H))$$ where $H=c_1(\ms{O}_{\mb{P}^r}(1))$ and the degree $r+1$ polynomial $F(H)=\prod_{i=0}^r (H+t_i)$ is the universal Leray relation \cite[Section 3.3]{EG98}. Also, \[A^\bullet_{GL_{r+1}}((\mb{P}^r)^n) = A^\bullet_{GL_{r+1}}(\mb{P}^r)^{\otimes n} = A^\bullet_{GL_{r+1}}(\pt)[H_1,\ldots,H_n]/(F(H_1),\ldots,F(H_n))\] and \[A^\bullet_{GL_{r+1}}(\mb{P}^{(r+1)\times d-1}) = A^\bullet_{GL_{r+1}}(\pt)[H]/(F(H)^d).\]
\begin{Def}
\label{reducedef}
For a one-variable polynomial $g$ with coefficients in $A^\bullet_{GL_{r+1}}(\pt)$, we define $\overline{g}$ to be the reduction of $g$ modulo $F$ (i.e. $\deg \overline{g}<\deg F$ and $g-\overline{g}$ is divisible by $F$).
\end{Def}
If $g$ is a polynomial with coefficients in $A^\bullet_{GL_{r+1}}(\pt)$, then the equivariant integration (pushforward) map $\int_{\mb{P}^r}:A^\bullet_{GL_{r+1}}(\mb{P}^r) \to  A^\bullet_{GL_{r+1}}(\pt)$ takes $g(H)$ to the coefficient of $H^r$ in $\bar{g}(H)$.

\begin{Def}
\label{dualdef}
If $X$ is a variety acted on by $GL_{r+1}$ and $\alpha \in A^\bullet_{GL_{r+1}}(X)$, then we define the Kronecker dual
$\alpha^{\dagger}$ as the function $A^\bullet_{GL_{r+1}}(X) \to A^\bullet_{GL_{r+1}}(\pt)$ given by
$$\alpha^{\dagger}:\beta \mapsto \int \alpha\cap\beta.$$
\end{Def}

\subsection{The quantum product $\star$ and $QH^\bullet_{GL_{r+1}}(\mb{P}^r)$}
\label{QCRD}
We define the \emph{small equivariant quantum cohomology ring of $\mb{P}^r$} (as an abstract ring, independent of characteristic) to be
$$QH^\bullet_{GL_{r+1}}(\mb{P}^r) = A^\bullet_{GL_{r+1}}(\pt)[z][\hbar]/(\hbar-F(z))\cong A^\bullet_{GL_{r+1}}(\pt)[z].$$  Specializing to $\hbar=0$ recovers $A^\bullet_{GL_{r+1}}(\mb{P}^r)$, so $QH^\bullet_{GL_{r+1}}(\mb{P}^r)$ is a commutative, associative deformation of $A^\bullet_{GL_{r+1}}(\mb{P}^r)$. We use $\star$ to denote the multiplication on $QH^\bullet_{GL_{r+1}}(\mb{P}^r)$. 


If $f$ is a polynomial in one variable $z$ with coefficients in $A^\bullet_{GL_{r+1}}(\pt)$, we denote by $[z^k]f \in A^\bullet_{GL_{r+1}}(\pt)$ the coefficient of $z^k$ in $f$. We denote by $[F(z)^k]f$ or $[\hbar^k]f$ the coefficient $a_k$ when $f$ is expressed as \[f(z)=a_0+a_1F(z)+a_2 F(z)^2 + \ldots=a_0+a_1\hbar+a_2 \hbar^2 + \ldots\] where each coefficient $a_i$ is a polynomial of degree at most $r$ in $z$.

\subsubsection{Quantum product $\star$ on $A^\bullet_{GL_{r+1}}(\mb{P}^r)[\hbar]$}
We will define an isomorphism of $A^\bullet_{GL_{r+1}}(\pt)$-modules between $A^\bullet_{GL_{r+1}}(\mb{P}^r)[\hbar]$ and $H^\bullet_{GL_{r+1}}(\mb{P}^r)$ and define $\star$ on $A^\bullet_{GL_{r+1}}(\mb{P}^r)[\hbar]$ via this isomorphism. 

There is an $A^\bullet_{GL_{r+1}}(\pt)$-linear lifting map 
\begin{align*}
    A^\bullet_{GL_{r+1}}(\mb{P}^r)&\to QH^\bullet_{GL_{r+1}}(\mb{P}^r) \qquad\text{where }    g(H) \mapsto \bar{g}(z),
\end{align*} 
given by choosing the representative that is reduced modulo $F$. This yields an isomorphism $A^\bullet_{GL_{r+1}}(\mb{P}^r)[\hbar] \cong QH^\bullet_{GL_{r+1}}(\mb{P}^r)$ of $A^\bullet_{GL_{r+1}}(\pt)$-modules given by extending $\hbar$-linearly. We use this isomorphism to define multiplication $\star$ on $A^\bullet_{GL_{r+1}}(\mb{P}^r)[\hbar]$. (Observe that $\star$ is not the usual multiplication operation on $A^\bullet_{GL_{r+1}}(\mb{P}^r)[\hbar]$ as a polynomial ring over $A^\bullet_{GL_{r+1}}(\mb{P}^r)$.)

The operation $\star$ on $A^\bullet_{GL_{r+1}}(\mb{P}^r)[\hbar]$ is $A^\bullet_{GL_{r+1}}(\pt)[\hbar]$-linear. If $f$ and $g$ are polynomials with coefficients in $A^\bullet_{GL_{r+1}}(\pt)$, then $f(H) \star g(H)=a(H)+ b(H) \hbar$, where $a$ and $b$ are the polynomials of degree at most $r$ with $\overline{f}(z) \overline{g}(z) = a(z) + b(z)F(z)$. These two properties characterize the $\star$-product. From an operational point of view, the usual multiplication on $A^\bullet_{GL_{r+1}}(\mb{P}^r)$ is polynomial multiplication mod $F$ and the $\star$-product is polynomial multiplication.

\subsection{Generalized matrix orbit closures}
\label{Setup}
Fix an integer $r$. For each $d \times n$ matrix $M$ with no zero columns, we consider the map $\mu_M:\mb{P}^{(r+1)\times d-1}\dashrightarrow (\mb{P}^r)^n$ induced by multiplication $A \mapsto AM$ of $(r+1)\times d$ matrices with $M$. More precisely, let $x_1, \ldots, x_n$ be the columns of $M$. Then $\mu_M(\wt A) = (\wt {Ax_1}, \ldots, \wt {Ax_n}) \in (\mb{P}^r)^d$ whenever $A \in \mb{A}^{(r+1) \times d}$ is a matrix such that $Ax_i \neq 0$ for all $i$. The map $\mu_M$ is clearly $GL_{r+1}$-equivariant. If $d \leq r+1$, then $\ol{\im(\mu_M)}$ is equal to the closure of the $GL_{r+1}$-orbit of $(\wt{x}_1, \ldots, \wt{x}_n) \in (\mb{P}^r)^n$, where $\wt{x_i}$ is considered in $\mb{P}^r$ via any linear inclusion $\mb{P}^{d-1} \subset \mb{P}^r$.
Let
\begin{center}
\begin{tikzcd}
B \arrow[d,"\pi_M"] \arrow[dr, "\wt{\mu}_M"] & \\
\mb{P}^{(r+1)\times d-1} \arrow[r, dashed, "\mu_M"] & (\mb{P}^r)^n
\end{tikzcd}
\end{center}
be any resolution of the indeterminacy of $\mu_M$. In proofs, we will always use the $GL_{r+1}$-equivariant resolution given by the wonderful compactification of \cite{LiLi} detailed in \Cref{LL}, which has the advantage of being given by a sequence of blowups along smooth centers. 

We define the \emph{generalized matrix orbit of $M$} to be the algebraic cycle
$$\op{M}=(\wt{\mu}_M)_*[B] = (\wt{\mu}_M)_*\pi_M^*[\mb{P}^{(r+1)\times d-1}] \in Z_{(r+1)d-1}((\mb{P}^r)^n).$$ This cycle is independent of the choice of resolution $\wt{\mu}_M$, and is equal to $\overline{\im(\mu_M)}$ or $0$ because if $\mu_M$ is generically finite then it is generically injective (if $\mu_M(A_1)=\mu_M(A_2)$, then this further equals $\mu_M(A_1+tA_2)$ for every $t$). We remark that this cycle is automatically $0$ if $d\ge n$, however the graph closure $\overline{\Gamma_M}$, which we will shortly define will be nontrivial even if $d\ge n$. As the morphisms $\wt{\mu}_M, \pi_M$ are $GL_{r+1}$-equivariant 
we may further define the \emph{equivariant generalized matrix orbit pushforward and pullback maps} \[(\mu_M)_*:A^\bullet_{GL_{r+1}}(\mb{P}^{(r+1)\times d-1}) \rightleftarrows A^\bullet_{GL_{r+1}}((\mb{P}^r)^n):\mu_M^*\] on equivariant Chow groups by $(\mu_M)_* =(\wt{\mu}_M)_* \circ \pi_M^*$ and $\mu_M^*=(\pi_M)_*\circ \wt{\mu}_M^*$. These are $A^\bullet_{GL_{r+1}}(\pt)$-linear maps, independent of the choice of resolution. The graph closure $\ol{\Gamma_M}$ of $\mu_M$ is the $GL_{r+1}$-equivariant cycle given by $$\ol{\Gamma_M}=\ol{\operatorname{im}(\operatorname{id} \times \mu_M)}\in Z_{(r+1)d-1}(\mb{P}^{(r+1)\times d-1}\times (\mb{P}^r)^n),$$ and resolves the indeterminacy locus of $\mu_M$ via its two projections. We have $$\int_{\mb{P}^{(r+1)\times d-1}}\alpha\cap\mu_M^*\beta=\int_{B}\pi_M^*\alpha\cap\wt{\mu}_M^*\beta=\int_{(\mb{P}^r)^n}\beta\cap(\mu_M)_*\alpha,$$ and $\int_{\ol{\Gamma_M}}\pi^*\alpha\cap\tilde{\mu}_M^*\beta=\int_{\mb{P}^{(r+1)\times d-1}\times (\mb{P}^r)^n}\alpha\cap\beta\cap[\ol{\Gamma_M}]$, so the equivariant formulas for $(\mu_M)_*$, $\mu^*_M$, and $[\ol{\Gamma_M}]$ all determine each other.

\subsection{Class of an arbitrary orbit closure}
\label{allorbits}
Observe that since we defined $\op{M}$ to be an $((r+1)d-1)$-dimensional algebraic cycle, it does not contain the data of $\ol{\im(\mu_M)}$ if $\dim (\ol{\im(\mu_M)}) < (r+1)d-1$. This in particular occurs if $M$ is not a ``connected'' matrix. Nevertheless, for $d \leq r+1$ we can still write $\ol{\im(\mu_M)}$ as the product of generalized matrix orbits as follows. First, by replacing $M$ with a subset of rows we may assume $\rk(M)=d$.

\begin{Def}
We say that a $d \times n$ matrix $M$ is \emph{connected} if $\rk(M) = d$ and there does not exist $\varnothing \subsetneq A \subsetneq \{1, \ldots, n\}$ such that $\rk_M(A) + \rk_M(\{1, \ldots, n\} \setminus A) = d$.
\end{Def}
\begin{Exam}
Any direct sum $M_1\oplus M_2$ is not connected. If $M$ is a general $d\times n$ matrix, then $M$ is connected if and only if $d<n$.
\end{Exam}

Observe that any $d \times n$ matrix $M$ of rank $d$ with no zero columns can be written as \[g (M_1 \oplus \ldots \oplus M_k) P\] where each $M_i$ is connected, $P$ is a permutation matrix, and $g \in GL_d$.
\begin{Prop}
\label{connectedfact}
Let $\wt{x}_1, \ldots, \wt{x}_n \in \mb{P}^{d-1}$ be points not all contained in any hyperplane and let $M$ be the matrix with columns $x_1, \ldots, x_n$. Suppose that $M = M_1 \oplus \cdots \oplus M_k$ where each $M_i$ is connected.

Then if $d \leq r+1$, the $GL_{r+1}$-orbit closure $\ol{\im(\mu_M)}$ of $(\wt{x}_1, \ldots, \wt{x}_n) \in (\mb{P}^r)^n$ is equal to $\op{M_1} \times \ldots \times \op{M_k}$ and has the expected dimension $(r+1)d - k$. Hence $[\ol{\im(\mu_M)}]=\prod_{i=1}^k \opc{M_i}$.
\end{Prop}
\begin{proof}
We will first prove the proposition for $k=1$. Viewing $\mb{P}^{d-1}$ as linearly embedded in $\mb{P}^r$, it suffices to show that any element stabilizing $(\wt{x}_1, \ldots, \wt{x}_n)$ stabilizes every point of $\mb{P}^{d-1}$. Indeed, given such a $g$, each $x_i$ is an eigenvector of $g$, so $K^{d} = \operatorname{Span} \{x_1, \ldots, x_n\}$ is a direct sum of eigenspaces of $g$. Since $M$ is connected, each $x_i$ has the same eigenvalue, so $g$ acts by a scalar on $K^d$, as desired.

For any $k$, let $d_i$ be the number of columns of $M_i$. Then, if $A_i$  is a $(r+1)\times d_i$ matrix for $1 \leq i \leq n$, we have the equality of block matrices
\[
\begin{pmatrix} 
\vert &  & \vert \\
A_1 & \cdots & A_k \\
\vert &  & \vert 
\end{pmatrix} 
\begin{pmatrix}M_1 & & \\ 
 & \ddots & \\
&  &  M_k
\end{pmatrix} 
= \begin{pmatrix}
\vert &  & \vert \\
A_1 M_1 & \cdots & A_k M_k \\
\vert &  & \vert
\end{pmatrix}.
\]
Hence $\ol{\im(\mu_M)} = \ol{\im(\mu_{M_1})} \times \ldots \times \ol{\im(\mu_{M_k})}$ and the result follows from the previous paragraph.
\end{proof}

By the above, to compute the equivariant class of any $GL_{r+1}$-orbit in $(\mb{P}^r)^n$ it suffices to compute the classes $\opc{M}$. From now on we work exclusively with $\op{M}$ instead of $\ol{\im(\mu_M)}$.

All of our theorems can be trivially extended to all matrices by defining $\op{M}$, $(\mu_M)_*$, $\mu_M^*$, and $\ol{\Gamma_M}$ to all be zero if $M$ has a zero column. However, we will not treat matrices with a zero column in our proofs.

\subsection{The operation $[M]_\hbar$}\label{Mhbardef}
In this last subsection, we define our operation $[M]_\hbar$ for a $d \times n$ matrix $M$. 
The operator $$[M]_\hbar:A^\bullet_{GL_{r+1}}(\mb{P}^r)^{\otimes n} \to QH^\bullet_{GL_{r+1}}(\mb{P}^r)$$ is defined by taking the formal sum of convolutions (partial integrals to the last factor of $\mb{P}^r$)
\begin{align*}[M]_\hbar(f_1(H_1),\ldots,f_n(H_n))=\sum_{k=0}^{d-1}\left(\int_{(\mb{P}^r)^{n+1}\to \mb{P}^r}f_1(H_1)\ldots f_n(H_n)\opc{\tau^{\le k+1}(M\oplus \ast)}\right)\hbar^k\\\in A^\bullet_{GL_{r+1}}(\mb{P}^r)[\hbar]\cong QH^\bullet_{GL_{r+1}}(\mb{P}^r).\end{align*}

See \Cref{KSI}, where we exposit the geometric construction of the small equivariant quantum cohomology of $\mb{P}^r$, which parallels the above.

\section{Statement of Results, extended example, and high level outline}
\label{statement}
In this section, we state our main results, give an extended example, and give a high level outline of our proofs.

Our first result computes the $GL_{r+1}$-equivariant Chow class of $\opc{M}$.

The non-equivariant class $\opc{M}$ computed in \cite{Binglin} (see \Cref{NE}) is an integer point transform of $(r+\frac{1}{n})(1,\ldots,1) - (r+1)P_M$. By Brion's~Theorem \cite{BHS}, it can be written as $f_M(H_1,\ldots,H_n,H_1^{r+1},\ldots,H_n^{r+1})$, where $f_M$ is a rational function that does not depend on $r$. The expression in \Cref{finkformulaintro} is then $f_M(H_1,\ldots,H_n,F(H_1),\ldots,F(H_n))$. We will see later that a similar procedure computes the $GL_{r+1}$-equivariant class $[\ol{\Gamma_M}]$ using the independence polytope $I_M$.

\begin{Thm*}[\Cref{finkformulas}] \label{finkformulaintro}
Let $M$ be a $d \times n$ matrix with nonzero columns $x_1,\ldots, x_n$. If $\rk(M) < d$, then $\op{M} = 0$. Otherwise, for each permutation $\sigma\in S_n$, let $B(\sigma)$ be the lexicographically first $d$-element subset $\{i_1,\ldots, i_d\} \subset \{1,\ldots,n\}$ with respect to the ordering $\sigma(1)\prec \cdots\prec\sigma(n)$ such that $x_{i_1},\ldots,x_{i_d}$ is a basis of $K^d$. Then for indeterminates $z_1,\ldots,z_n$, the expression $$\sum_{\sigma \in S_n} \left(\prod_{i \in \{1,\ldots,n\}\setminus B(\sigma)} F(z_i)\right)\frac{1}{(z_{\sigma(2)}-z_{\sigma(1)})\ldots (z_{\sigma(n)}-z_{\sigma(n-1)})}$$ is a polynomial of degree at most $r$ in each $z_i$ with coefficients in $A^{\bullet}_{GL_{r+1}}(\pt)$, and the equivariant Chow class $\opc{M}$ is given by evaluating this polynomial at $H_1,\ldots,H_n$.
\end{Thm*}

The above formula, and the one we describe later for $[\overline{\Gamma_M}]$, also provide formulas for $GL_{r+1}\times (K^\times)^n$-equivariant classes of the analogues of $\op{M}$ and $\ol{\Gamma_M}$ in $\mb{A}^{(r+1)\times n}$ (see \Cref{liftingsection}). In particular, this recovers the torus-equivariant classes of torus orbits in the Grassmannian. 

Because of the connections to sums over lattice points of rescaled rank and independence polytopes, the formulas for $\opc{M}$ and for $[\overline{\Gamma_M}]$ turn out to be additive and valuative respectively.

\begin{Thm*}[\Cref{realval}, \Cref{graphvaluative}]
\label{realvalintro}
For fixed $\mb{P}^r$, the maps $M \mapsto \opc{M}$  and $M \mapsto [\ol{\Gamma_M}]$ from $d \times n$ matrices to equivariant Chow classes
are additive and valuative associations respectively. In particular, $\opc{M}$ and $[\overline{\Gamma_M}]$ only depend on the matroid of $M$.
\end{Thm*}

The classes $\opc{M}$ also behave well under series and parallel connection.

\begin{Thm*}[\Cref{serpar}]\label{serparintro}
For $i \in \{1,2\}$ let $M_i$ be a $d_i \times n_i$ matrix. Taking series and parallel connections along the last column of $M_1$ and the first column of $M_2$, we have in $A^\bullet_{GL_{r+1}}(\mb{P}^r)^{\otimes n_1-1}\otimes QH^\bullet_{GL_{r+1}}(\mb{P}^r)\otimes A^\bullet_{GL_{r+1}}(\mb{P}^r)^{\otimes n_2-1}\cong A^\bullet_{GL_{r+1}}(\mb{P}^r)^{\otimes n_1+n_2-1}[\hbar]$ the equality $$\opc[r][n_1]{M_1} \star \opc[r][n_2]{M_2}=\opc[r][n_1+n_2-1]{P(M_1,M_2)}+\hbar\opc[r][n_1+n_2-1]{S(M_1,M_2)}$$ where the quantum product $\star$ is taken along the last factor of $A^\bullet_{GL_{r+1}} ((\mb{P}^r)^{n_1}) = A^\bullet_{GL_{r+1}} (\mb{P}^r)^{\otimes n_1}$ and the first factor of $A^\bullet_{GL_{r+1}} ((\mb{P}^r)^{n_2}) = A^\bullet_{GL_{r+1}} (\mb{P}^r)^{\otimes n_2}$.
\end{Thm*}

The next theorem develops the key properties of the operations $[M]_\hbar$.

\begin{Thm*}[\Cref{poincaresection}]\label{Mhbarintro}
Fix $r \ge 0$. Then the $A^\bullet_{GL_{r+1}}(\pt)$-linear operations $$[M]_{\hbar}:A^\bullet_{GL_{r+1}}(\mb{P}^r)^{\otimes n}\to QH^\bullet_{GL_{r+1}}(\mb{P}^r)$$ satisfy
$$[z^r][F(z)^k][M_{\hbar}]=\opc{\tau^{\le (k+1)}M}^{\dagger}.$$
The following properties uniquely characterize the $A^\bullet_{GL_{r+1}}(\pt)$-linear operations $[M]_{\hbar}$.
\begin{enumerate}
    \item Let $\ast$ denote the $1 \times 1$ identity matrix. Then $$[\ast]_\hbar(f(H))=\overline{f}(z).$$
    \item If $M_1$ and $M_2$ are matrices with $n_1$ and $n_2$ columns respectively, then $$[M_1 \oplus M_2]_\hbar = [M_1]_\hbar \star [M_2]_\hbar,$$ where $[M_1]_\hbar$ applies to the first $n_1$ tensor factors and $[M_2]_\hbar$ applies to the last $n_2$ tensor factors.
    \item If $M$ is a $d \times n$ matrix and $\tau^{\leq k}$ is a general $k \times d$ matrix, then $[\tau^{\leq k} M]_\hbar$ is the reduction of $[M]_\hbar$ modulo $\hbar^k=F(z)^k$.
    \item The association $M \mapsto [M]_{\hbar}$ is valuative. For fixed $n$ and possibly varying $d$, the operation $[M]_{\hbar}$ depends only on the matroid of $M$.
    \item The actions of the symmetric group $S_n$ on the columns of $M$ and on $A^\bullet_{GL_{r+1}}(\mb{P}^r)^{\otimes n}$ are compatible.
\end{enumerate}
\end{Thm*}

The next theorem relates the class $[\ol{\Gamma_M}]$ to the operation $[M]_\hbar$. Recall that $\mu_M^*:A^\bullet_{GL_{r+1}}((\mb{P}^r)^n)\to A^\bullet_{GL_{r+1}}(\mb{P}^{(r+1)\times d-1})$ was defined via convolution with $[\overline{\Gamma_M}]$.
\begin{Thm*}[\Cref{PFA}]\label{PFAintro}
Define 
\begin{align*}
    q&:QH^\bullet_{GL_{r+1}}(\mb{P}^r)\to A^\bullet_{GL_{r+1}}(\mb{P}^{(r+1)\times d-1})\text{ by }z \mapsto H\\
    \ell&:A^\bullet_{GL_{r+1}}(\mb{P}^{(r+1)\times d-1}) \to QH^\bullet_{GL_{r+1}}(\mb{P}^r)\text{ by }\ell(f(H))=\widetilde{f}(z),
\end{align*}
where $\widetilde{f}$ is the reduction of $f$ modulo $F(H)^d$ (analogously to \Cref{reducedef}). Then
$$\mu_M^*=q\circ [M]_{\hbar},\qquad [M]_{\hbar}=\ell \circ \mu_M^*.$$
\end{Thm*}

\begin{Rem}
From this, we can express the $GL_{r+1}$-equivariant class $[\overline{\Gamma_M}]$ in terms of the classes $\opc{\tau^{\le k+1}(M \oplus \ast)}$ used to define $[M]_{\hbar}$ (see \Cref{PFA}). This expression in turn may be seen as an application of Brion's formula applied to $(r+\frac{1}{n})(1,\ldots,1) - (r+1)I_M$ with $H$ used non-equivariantly as a grading variable to ensure the expression is homogenous of rank $rn$, see the proof of \cref{starunion}.
\end{Rem}

We give an application to enumerative geometry. Given a smooth quintic plane curve $X\subset\mb{P}^2$, there is a rational map
\begin{align*}
\phi_X: (\mb{P}^{2})^\vee \dashrightarrow \ol{\mc{M}_{0,5}}/S_5
\end{align*}
sending a line $\ell$ to the moduli of the intersection $\ell \cap X$ as a quintuple of points on $\ell\cong \mb{P}^1$. Cadman and Laza \cite{Laza} showed that if every line intersects $X$ in at least three points, then 
\begin{align*}
\deg (\phi_X) = 2\cdot (\text{\# bitangents of $X$})+4\cdot (\text{\# flexes of $X$}) = 2\cdot 120 + 4\cdot 45 = 420.
\end{align*}
We extend this to hypersurfaces $X\subset\mb{P}^r$ of degree $d=2r+1$ for all $r$ by computing $\deg(\phi_X)$ and separately expressing $\deg(\phi_X)$ in terms of the lines tri-incident to $X$. 

\begin{Thm*}[\Cref{AP}]
\label{APIntro}
Let $X\subset\mb{P}^r$ be a hypersurface of degree $d=2r+1$ such that every line meets $X$ in at least three points and is not contained in $X$. The degree of
\begin{align*}
\phi_X: \mb{G}(1,r) \dashrightarrow \ol{\mc{M}_{0,d}}/S_{d}
\end{align*}
is the difference between the coefficients of $u^{r-1}v^{r-1}$ and $u^{r-2}v^{r}$ in the expansion of 
\begin{align*}
d(d-1)(d-2)\prod_{k=2}^{d-2} (ku+(d-k)v).
\end{align*}
\end{Thm*}
In fact \Cref{AP} computes the class of a general fiber of $\phi_X$ when we drop the assumption $d=2r+1$. Finally, we can express $\deg(\phi_X)$ in terms of numbers of tri-incident lines. 
\begin{Thm*} [\Cref{APD}]
\label{APDIntro}
Let $X\subset\mb{P}^r$ be a hypersurface of degree $d=2r+1$ such that every line meets $X$ in at least three points and is not contained in $X$. Then 
\begin{align}
 \deg(\phi_X) = \sum_{a\ge b>1,a+b+1=d}{2n_{a,b,1}}+4n_{d-2,1,1},\nonumber\label{APT}
\end{align}
where $n_{a,b,c}$ is the number (counted with certain multiplicities) of lines that intersect $X$ at three points with multiplicities exactly $a,b,c$. 
\end{Thm*}
We will see later that \Cref{APDIntro} is induced by a relation between generalized matrix orbit classes arising from \Cref{realvalintro}.

\subsection{Extended Example}
\label{extendedexample}
In this subsection we present an extended example to demonstrate the key aspects of our results. 
\subsubsection{Setup} Let $p_1,\ldots,p_{6}\in \mb{P}^2$ be the points depicted below (or any other six points with the same collinearity relations).

\begin{center}
\begin{tikzpicture}
\draw (0, 0) -- (1, 1) -- (2, 0) -- (0, 0);
\filldraw[black] (1,1) circle (2pt) node[anchor=south] {1};
\filldraw[black] (0.5,0.5) circle (2pt) node[anchor=south east] {2};
\filldraw[black] (1.5,0.5) circle (2pt) node[anchor=south west] {3};
\filldraw[black] (0, 0) circle (2pt) node[anchor=north east] {4};
\filldraw[black] (1,0) circle (2pt) node[anchor=north] {5};
\filldraw[black] (2,0) circle (2pt) node[anchor=north west] {6};
\end{tikzpicture}
\end{center}
Let $A$ be any $3 \times 6$ matrix whose projectivized columns are $p_1, \ldots, p_{6}$. For any $r \geq 2$, the cycle $\opc{A} \in A^\bullet_{GL_{r+1}}((\mb{P}^r)^{6})$ is the closure of the $GL_{r+1}$-orbit of $(p_1, \ldots, p_{6}) \in (\mb{P}^r)^{6}$, where each point $p_i$ is considered as an element of $\mb{P}^r$ via any linear inclusion $\mb{P}^2 \hookrightarrow \mb{P}^{r}$ (\Cref{connectedfact}).

\subsubsection{Formula for $\opc{A}$}
The rank polytope of (the column matroid of) $A$ is the convex hull of the 17 points
\begin{align*}
\{e_i+e_j+e_k\mid p_i, p_j, p_k\text{ span }\mb{P}^2\} \subset \mb{R}^{6}.
\end{align*}
and \Cref{finkformulaintro} expresses the class $\opc{M}\in A^\bullet_{GL_{r+1}}((\mb{P}^r)^{6})$ as a sum of $6!$ rational functions. However, such a large sum is difficult to manipulate in practice. Instead, we may understand $\opc{A}$ by computing the Kronecker dual \[\opc{A}^{\dagger} : A^\bullet_{GL_{r+1}}(\mb{P}^r)^{\otimes 6} \to A^\bullet_{GL_{r+1}}(\pt),\] which we recall is defined by \[\opc{A}^\dagger(f_1(H), \ldots, f_6(H)) = \int_{(\mb{P}^r)^{6}} \opc{A} f_1(H_1) \ldots f_6(H_6)\] for any polynomials $f_1, \ldots, f_{6}$ with coefficients in $A^\bullet_{GL_{r+1}}(\pt)$.
\subsubsection{Degenerations}
In many cases (such as this one), we can compute Kronecker dual classes efficiently by breaking the problem into smaller pieces.

We can relate $\opc{A}$ with other classes through the following degeneration tree. Here $B,C,D,E$ are $3\times 6$ matrices whose projectivized columns are the depicted point configurations. 
\begin{center}
\begin{tikzpicture}

\draw (0.5,0) -- (3.5,0);
\draw (1.75,1.25) -- (3.25,-0.25);
\draw (0.75,-0.25) -- (2.25,1.25);
\draw[->] (3.5,1.5)--(3,1);
\draw[->] (6.5,1.5)--(7,1);
\filldraw[black] (2,0) circle (2pt) node[anchor=north] {5};
\filldraw[black] (3,0) circle (2pt) node[anchor=north] {6};
\filldraw[black] (1.5,0.5) circle (2pt) node[anchor=south east] {2};
\filldraw[black] (2,1) circle (2pt) node[anchor=south] {1};
\filldraw[black] (2.5,0.5) circle (2pt) node[anchor=south west] {3};
\filldraw[black] (1,0) circle (2pt) node[anchor=north] {4};
\node at (0,0) [rectangle,draw] {A};

\draw (3.5,2) -- (6.5,2);
\draw (4.75,3.25) -- (6.25,1.75);
\draw[->] (6.5,3.5)--(6,3);
\draw[->] (9.5,3.5)--(10,3);
\filldraw[black] (5,2) circle (2pt) node[anchor=north] {5};
\filldraw[black] (6,2) circle (2pt) node[anchor=north] {6};
\filldraw[black] (4.25,2.75) circle (2pt) node[anchor=west] {2};
\filldraw[black] (5,3) circle (2pt) node[anchor=south west] {1};
\filldraw[black] (5.5,2.5) circle (2pt) node[anchor=south west] {3};
\filldraw[black] (4,2) circle (2pt) node[anchor=north] {4};
\node at (3,2) [rectangle,draw] {B};

\draw (6.5,4) -- (9.5,4);
\filldraw[black] (8,4) circle (2pt) node[anchor=north] {5};
\filldraw[black] (9,4) circle (2pt) node[anchor=north] {6};
\filldraw[black] (7.25,4.75) circle (2pt) node[anchor=west] {2};
\filldraw[black] (8,5) circle (2pt) node[anchor=west] {1};
\filldraw[black] (8.75,4.75) circle (2pt) node[anchor=north] {3};
\filldraw[black] (7,4) circle (2pt) node[anchor=north] {4};
\node at (6,4) [rectangle,draw] {D};

\filldraw[black] (7.25,0.75) circle (2pt) node[anchor=west] {2};
\filldraw[black] (8,1) circle (2pt) node[anchor=west] {1};
\filldraw[black] (7,0) circle (2pt) node[anchor=west] {4};
\filldraw[black] (9,0) circle (2pt) node[anchor=west] {3, 5, 6};
\node at (6,0) [rectangle,draw] {C};

\filldraw[black] (11,3) circle (2pt) node[anchor=west] {1};
\filldraw[black] (11.75,2.75) circle (2pt) node[anchor=west] {3};
\filldraw[black] (12,2) circle (2pt) node[anchor=west] {6};
\filldraw[black] (10,2) circle (2pt) node[anchor=west] {2, 4, 5};
\node at (9,2) [rectangle,draw] {E};
\end{tikzpicture}
\end{center}
The rank polytopes of $A$, $B$, $C$, $D$, and $E$ are related by subdivisions. The hyperplane $\{x_1+x_2+x_{4}=2\}$ separates the rank polytope $P_B$ of $B$ into the rank polytopes $P_A$ and $P_C$ of $A$ and $C$ respectively. Similarly, the hyperplane $\{x_1+x_3+x_6=2\}$ separates $P_D$ into $P_B$ and $P_E$.

Accordingly, the cycle $\op{B}$ degenerates to a union $\op{A}\cup \op{C}$ and $\op{D}$ degenerates to a union $\op{B}\cup \op{E}$ (\Cref{degenerationfibers}), so we get 
\begin{align*}
\opc{B}=\opc{A}+\opc{C} \qquad \opc{D}=\opc{B}+\opc{E}.
\end{align*}
Not all subdivisions of rank polytopes yield degenerations of the cycles $\op{M}$, but they all yield relations between the generalized matrix orbit classes $\opc{M}$ (\Cref{realvalintro}).

Therefore, 
\[
\opc{A}=\opc{D}-\opc{C}-\opc{E}
\] so 
\[
\opc{A}^\dagger=\opc{D}^\dagger-\opc{C}^\dagger-\opc{E}^\dagger.
\]

\subsubsection{Kronecker duals to Schubert Matroids}
\label{Schubertexample}
To compute $\opc{A}^{\dagger}$, it remains to compute $\opc{C}^{\dagger}$, $\opc{D}^{\dagger}$, and $\opc{E}^{\dagger}$, for which we will use the operations 
\[[C]_\hbar, [D]_\hbar, [E]_\hbar : A^\bullet_{GL_{r+1}}(\mb{P}^r)^{\otimes 6} \to QH^\bullet_{GL_{r+1}}(\mb{P}^r) = A^\bullet_{GL_{r+1}}(\pt)[z][\hbar]/(\hbar-F(z)).\]
(\Cref{Mhbarintro}). 

To compute $[C]_\hbar, [D]_\hbar, [E]_\hbar$, we will use the fact that the column matroids of $C$, $D$, $E$ are Schubert matroids. For example, we may write \[D = \operatorname{Sch}(2,3,\{4,5,6\},\{1,2,3,4,5,6\}).\] Hence, by the definition of a Schubert matroid and \Cref{Mhbarintro} we have 
\begin{align*}
[D]_\hbar(f_1(H), \ldots, f_{6}(H)) &= f_1(z)f_2(z)f_3(z) (f_4(z) f_{5}(z) f_{6}(z)\text{ mod }F(z)^2)\text{ mod }F(z)^3
\end{align*}
for any polynomials $f_1, \ldots, f_{6}$ of degree at most $r$ with coefficients in $A^\bullet_{GL_{r+1}}(\pt)$. We have thus expressed $[D]_\hbar$ using polynomial multiplication and division with remainder.

By \Cref{Mhbarintro} again, we have
\begin{align*}
\opc{D}^{\dagger}(f_1(H), \ldots, f_{6}(H)) &= [z^r][F(z)^2][D]_\hbar(f_1(H), \ldots, f_{6}(H)) \\
&= [z^r][F(z)^2](f_1(z)f_2(z)f_3(z) (f_4(z) f_{5}(z) f_{6}(z)\text{ mod }F(z)^2)).
\end{align*}

We can perform a similar procedure for the matrices $C$ and $E$, yielding
\begin{align*}
\opc{C}^{\dagger}(f_1(H), \ldots, f_{6}(H)) &= [z^r][F(z)^2][C]_\hbar(f_1(H), \ldots, f_{6}(H)) \\
&= [z^r][F(z)^2](f_1(z)f_2(z)f_4(z) (f_3(z) f_{5}(z) f_{6}(z)\text{ mod }F(z))) \\
\opc{E}^{\dagger}(f_1(H), \ldots, f_{6}(H)) &= [z^r][F(z)^2][E]_\hbar(f_1(H), \ldots, f_{6}(H)) \\
&= [z^r][F(z)^2](f_1(z)f_3(z)f_6(z) (f_2(z) f_{4}(z) f_{5}(z)\text{ mod }F(z))).
\end{align*}
Hence $\opc{A}^\dagger$ is given by
\begin{align*}
\opc{A}^{\dagger}(f_1(H), \ldots, f_{6}(H)) = [z^r][F(z)^2]( &f_1(z)f_2(z)f_3(z) (f_4(z) f_{5}(z) f_{6}(z)\text{ mod }F(z)^2)\\
&-f_1(z)f_2(z)f_4(z) (f_3(z) f_{5}(z) f_{6}(z)\text{ mod }F(z)) \\
&-f_1(z)f_3(z)f_6(z) (f_2(z) f_{4}(z) f_{5}(z)\text{ mod }F(z)) ) .
\end{align*}
for all polynomials $f_1, \ldots, f_{6}$ of degree at most $r$ with coefficients in $A^\bullet_{GL_{r+1}}(\pt)$. We will see in \Cref{poincaredualsforallom} how to apply a similar procedure to compute $\opc{M}^\dagger$ for all $M$.

\subsection{High level outline}
We now give a high level outline of our proofs. A flow chart containing the logical structure of the proofs is given in \Cref{flow}. Input from the literature is shown in dotted boxes, while our contributions are in solid boxes. 

\tikzset{%
  block/.style    = {dashed,draw, very thick, rectangle, minimum height = 3em,
    minimum width = 3em,text width=5cm},
  importantblock/.style    = {draw, very thick, rectangle, minimum height = 3em,
    minimum width = 3em,text width=5cm},
  sum/.style      = {draw, circle, node distance = 2cm}, 
  input/.style    = {coordinate}, 
  output/.style   = {coordinate} 
}
\begin{figure}[h]
    \centering
\begin{tikzpicture}[auto, thick, node distance=2cm, >=triangle 45]
    \node at (0,0) [block] (add) {Additivity (nonequivariant) from \cite{Binglin}};
    \node [importantblock, below of=add] (deg) {Equivariant degeneration of orbits};
    \node at (7,0) [importantblock] (infinite) {Relations generated by degenerations (modified algorithm of \cite{Derksen})};
    \node [importantblock, below of=infinite] (adde) {Equivariant Additivity (\Cref{realvalintro})};
    \draw[->](add) -- node {}(deg);
    \draw[->](deg) -- node {}(adde);
    \draw[->](infinite) -- node {}(adde);
    \draw [color=gray,thick](-3,1) rectangle (11,-3);
    
    \node at (0,-6) [importantblock] (seriesparallel) {Classes under series and parallel (\Cref{serparintro})};
    \node at (0,-4) [block] (quantum) {Quantum product of $\mathbb{P}^r$};
    \node at (7,-6) [importantblock] (serparm) {Classes for series-parallel matrices};  
    \node at (7,-8) [importantblock] (bigformula) {Formula for class $\opc{M}$ (\Cref{finkformulaintro})}; 
    \draw[->](seriesparallel) --node{} (serparm);
    \draw[->](quantum) --node{} (seriesparallel);
    \draw[->](serparm) --node{} (bigformula);
    
    \node at (7,-4) [importantblock] (rat) {Rational Function Trick};
    \coordinate (rrat) at (10.5,-4);
    \draw[-] (rat)--node{} (rrat);
    
    \node at (0,-8) [block] (Li) {Wonderful compatification of \cite{LiLi}};
    \node at (0,-10) [importantblock] (Graph) {$[M]_\hbar$ same as $[\overline{\Gamma_M}]$ convolution
    (\Cref{PFAintro})};
    \node at (0,-12) [importantblock] (valuative) {Equivariant valuativeness (\Cref{realvalintro})};
    \draw[->] (Li)--node{} (Graph);
    \draw[->] (Graph)--node{} (valuative);

    \node at (7,-10) [importantblock] (Mhbar) {$[M_1 \oplus M_2]_\hbar = [M_1]_\hbar \star [M_2]_\hbar$ (\Cref{Mhbarintro})};  
     \node at (7,-12) [importantblock] (Mhbarv) {$[M]_\hbar $ valuativeness (\Cref{Mhbarintro})};
    \draw[->] (Mhbarv)--node{} (valuative);

    \coordinate (c1) at (3.5,-6);
    \coordinate (c2) at (3.5,-10);
    \draw[-] (c1) --node{} (c2);
    \draw[->] (c2)--node{} (Mhbar);

    \coordinate (radde) at (10.5,-2);
    \coordinate (rbigformula) at (10.5,-8);
    \draw[-] (radde)--node{} (rbigformula);
    \draw[-] (adde)--node{} (radde);
    \draw[->] (rbigformula)--node{} (bigformula);
    
    \coordinate (rMhbar) at (10.5,-10);    
    \draw[-] (radde)--node{} (rMhbar);
    \draw[->] (rMhbar)--node{} (Mhbar);  
    \coordinate (rMhbarv) at (10.5,-12);
    \draw[-] (rMhbar) --node{} (rMhbarv);
    \draw[->] (rMhbarv) --node{} (Mhbarv);
\end{tikzpicture}
\caption{Logical structure of proofs} \label{flow}
\end{figure}
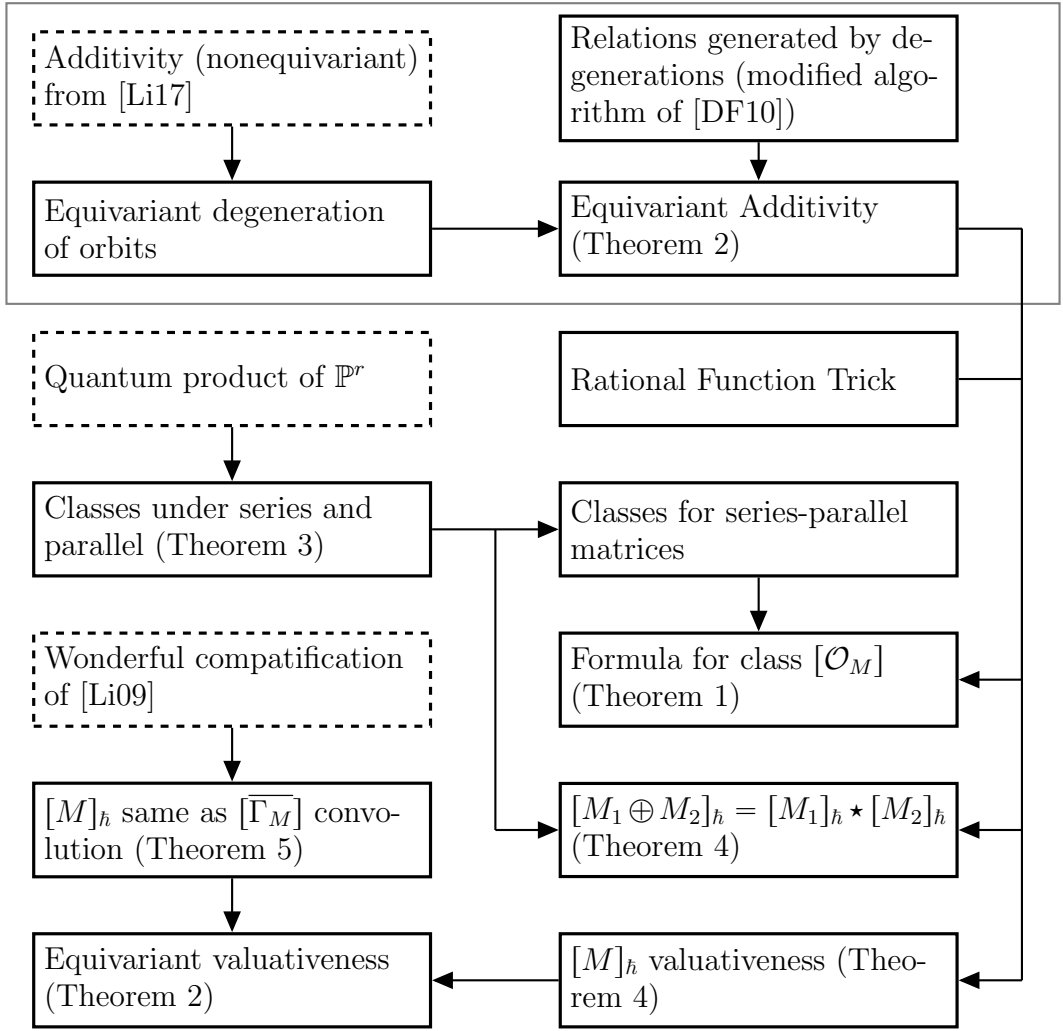

\subsubsection{Non-equivariant additivity}
In \Cref{NE}, we show the non-equivariant additivity of $M \mapsto \opc{M}$. Li's result \cite{Binglin} exhibits the non-equivariant class $\opc{M}$ as the sum of monomials in the $H_i$ variables whose exponent sequence lies inside the region $(r+\frac{1}{n})(1,\ldots,1)-(r+1)P_M$, so the result essentially follows by definition.

\subsubsection{Equivariant degenerations of orbits} 
\label{degsssection}
In \Cref{degenerationfibers}, we show that any flat limit of the form $\op{M(t)}$ as $t \to 0$ has the property that the special fiber is $\bigcup_{i=1}^k \opc{M_i}$ for some matrices $M_i$ such that the rank polytopes subdivide $P_{M_p}$ for $p$ general. In particular, we deduce in \Cref{siorbit} additivity relations of the form
$$\opc{M(p)}=\sum_{i=1}^k \opc{M_i}$$ arising from such degenerations. We do this by identifying enough $\op{M_i}$ in the limit so that both sides agree non-equivariantly, and then conclude by the fact that an effective cycle that is non-equivariantly zero is also equivariantly zero. Note at this point we do not know that $\opc{M}$ depends only on the matroid of $M$.

\subsubsection{Relations generated by degenerations}
Even though \Cref{degenerationfibers} provides many relations among the classes $\opc{M}$, it is not clear that these relations induced from degenerations imply additivity of $\opc{M}$. We will show that they in fact do.

 We prove in \Cref{matroidconelemma} that for $M$ having a particular type of matroid polytope, which we call a ``cone-like matroid polytope'' (see \Cref{conelike}), $\opc{M}$ depends only on the matroid of $M$. Given $M,M'$ with the same cone-like matroid polytope, we do this by creating a flat degeneration with two special fibers $\op{M}$ and $\op{M'}$.

Using degenerations to reduce to a certain class of matrices with ``positive cone-like matroid polytopes'', which satisfy no non-trivial additive relations, is a delicate combinatorial argument which modifies the algorithm from \cite[Appendix A]{Derksen}. The essential point is that we closely follow a sequence of infinite polytope subdivisions, which when restricted to $\Delta_{d,n}$ yield a sequence of finite polytope subdivisions. This is done by attaching a certain cone to each matrix in our sequence of degenerations, which governs future degenerations based on past ones. We defer the more detailed high level explanation to \Cref{outline}.

\subsubsection{Series and parallel connection}
We show that parallel and series connection are given by certain convolutions with the classes of the diagonal in $(\mb{P}^r)^3$ and the class of collinear triples of points in $(\mb{P}^r)^3$ respectively in \Cref{ParAndSer}. Then identical to the standard derivation of the equivariant $3$-point Gromov-witten invariants of $\mb{P}^r$, we are able to deduce \Cref{serparintro}.

\subsubsection{Series-parallel matroids}
To analyze series-parallel matrices, through their recursive definition it suffices to have a good understanding of how the $\opc{M}$ classes behave under series and parallel connections of the underlying matrices. 

Through a clever application of partial fraction decomposition combined with \Cref{serparintro}, we are able to show that for each series-parallel matroid there is a universal rational function in $2n$ variables such that for any $r$, if we substitute $H_1,\ldots,H_n,F(H_1),\ldots,F(H_n)$ into the variables then we formally recover the equivariant class. Note that as $r$ increases, the polynomial $F(H_i)$ changes with the addition of more $t_i$ variables.

\subsubsection{Formula for $\opc{M}$ (\Cref{finkformulaintro})}

Once we have additivity of the classes $M \mapsto \opc{M}$, to prove \Cref{finkformulaintro} it suffices to prove it for a special class of matrices whose matroid polytopes generate all others through additivity relations. \cite{Derksen} used the additive basis of Schubert matroids, which we will analyze in detail later, but we instead use series-parallel matroids, as every Schubert matroid polytope can be subdivided into series-parallel matroid polytopes.

From additivity we are able to deduce the existence of a rational function for any matrix $M$ determining $\opc{M}$ as above. We show in \Cref{formulaexists} that this rational function is uniquely determined by its values when substituting $H_1,\ldots,H_n,H_1^{r+1},\ldots,H_n^{r+1}$ (giving the nonequivariant class of $\op{M}$), so it must agree with the rational function produced by Brion's formula for the non-equivariant classes. \Cref{finkformulaintro} then follows by substituting $F(H_i)$ in place of $H_i^{r+1}$ in the formula, which is done in \Cref{formulasection}.

\subsubsection{Properties of $[M]_{\hbar}$ (\Cref{Mhbarintro})}
To prove \Cref{Mhbarintro}, we use an analogous non-equivariant for every $r$ to equivariant strategy as above, and show that non-equivariantly the data of $[M]_{\hbar}$ can be encoded via a sum of monomials inside $(r+\frac{1}{n})(1,1,\ldots,1)-(r+1)I_M$, by slicing the independence polytope according to the hyperplanes $\{\sum x_i=k\}$ with $k \in \mathbb{Z}$, and showing that each slice, homogenized with an extra variable $H_{n+1}$, appears as the rank polytope of $\tau^{\le k+1}(M\oplus \ast)$. Then all of the parts of \Cref{Mhbarintro} readily follows as shown in \Cref{poincaresection}. For example valuativity of $[M]_{\hbar}$ follows from the fact that valuativity can be interpreted in terms of independence polytopes, and $[M\oplus N]_{\hbar}=[M]_{\hbar}\star [N]_{\hbar}$ follows from the fact that $I_{M\oplus N}=I_M \times I_N$.

\subsubsection{Convolution with graph closure $\overline{\Gamma_M}$ (\Cref{PFAintro})}
Finally, \Cref{PFAintro} follows from carefully analyzing the wonderful compactification of \cite{LiLi}, where we have to relate the images of classes under $(\mu_M)_*$, and $(\mu_{\tau^{\le k+1}(M\oplus \ast)})_*$ under various pullbacks and pushforwards. Then valuativity of $[\overline{\Gamma_M}]$ from \Cref{realvalintro} follows from the corresponding valuativity of $[M]_{\hbar}$. This is done in \Cref{graphclosuresection}.

\section{Motivating example and the connection of $\mu_M$ with $QH_{GL_{r+1}}^\bullet(\mb{P}^r)$}
\label{KSI}

Here we make precise the connection between generalized matrix orbits and the small equivariant quantum cohomology of $\mb{P}^r$, assuming the part of \Cref{realvalintro} asserting that $\opc{M}$ depends only on the matroid of $M$. In this section only, we will assume $\operatorname{char} K = 0$.

The concrete focus of this section will be on the following problem.
\begin{Pro}
\label{ProS}
What is the class in $A^{\bullet}_{GL_{r+1}}((\mb{P}^r)^n)$ of the closure of the $GL_{r+1}$-orbit of $(p_1,\ldots,p_n)\in (\mb{P}^r)^n$, where $p_1, \ldots, p_n \in \mb{P}^r$ are general? \footnote{The expression for the class of this orbit closure computed in \cite[Theorem 5.1]{Fink} is incorrect because their expression for ``$[Y_{\pi(v)}]_T$'' does not satisfy the degree hypothesis of \cite[Theorem 3.5 (ii)]{Fink}.}
\end{Pro}
Let $M_{r+1}$ be an $(r+1) \times n$ matrix whose projectivized columns are these general $p_1, \ldots, p_n$. By \Cref{connectedfact} the $GL_{r+1}$-orbit closure of $(p_1,\ldots, p_n)$ is $\op{M}$ if $n > r+1$ and $(\mb{P}^r)^n$ if $n \le r+1$. 

In this section we will answer the following more refined problem, which yields orbit closures for $d+1 \le r+1$, generalized matrix orbits for $r+1<d+1<n$, and $0$ for $n \le d+1$.
\begin{Pro}
\label{ProR}
If $M_{d+1}$ is a general $(d+1)\times n$ matrix, what is $\opc{M_{d+1}}\in A^{\bullet}_{GL_{r+1}}((\mb{P}^r)^n)$?
\end{Pro}

We will assume $n\geq d+2$ for the remainder of \Cref{KSI}. 

\subsection{Specialization}
\label{SKS}
We first reduce \Cref{ProR} to a fact about the small equivariant quantum cohomology of $\mb{P}^r$. Let $p_1,\ldots,p_n$ be the projectivized columns of $M_{d+1}$.

We will show in \Cref{realval} that the class $\opc{M_{d+1}}$ depends only on the matroid of $p_1,\ldots,p_n$. Hence, rather than choosing $p_1, \ldots, p_n$ to be general, we may choose them in any way such that any $d+1$ of them are linearly independent. 

Therefore, we can choose $p_1,\ldots,p_n$ to be distinct points on a rational normal curve $C\subset \mb{P}^d$. In this way, we may view $C$ as an $n$-pointed rational curve $(\mb{P}^1,p_1,\ldots,p_n)$. We will now show that $\op{M_{d+1}}$ is the closure of \[\{(\psi(p_1), \ldots, \psi(p_n)) \mid \text{$\psi:\mb{P}^1 \to \mb{P}^r$ is a degree $d$ map}\}.\] Afterwards, we may break our problem into smaller parts by using the structure of $\overline{\MM_{0, n}}$, the coarse moduli space of $n$-pointed stable curves of genus $0$.

\subsection{Small Equivariant Quantum Cohomology of $\mb{P}^r$}
At this point it makes sense to recall small equivariant quantum cohomology \cite{Kim93}, specifically in the case of $\mb{P}^r$. We will use the \emph{Kontsevich space}, the coarse moduli space $\ol{\MM_{0,n}}(\mb{P}^r,d)$ whose points are $(\mc{C}, \psi)$ where $\mc{C}$ is a genus $0$ curve with $n$ marked points $p_1, \ldots, p_n$ and $\psi : \mc{C} \to \mb{P}^r$ is a degree $d$ stable map. There is a map $\pi : \ol{\mc{M}_{0,n}}(\mb{P}^r,d) \to \ol{\mc{M}_{0, n}}$ that forgets the morphism to $\mb{P}^r$ and stabilizes the domain curve and there are evaluation maps $\ev_1, \ldots, \ev_n : \ol{\mc{M}_{0,n}}(\mb{P}^r,d) \to \mb{P}^r$ that send $(\mc{C}, \psi)$ to $\psi(p_1), \ldots, \psi(p_n)$ respectively. We thus have the following diagram.
\begin{center}
\begin{tikzcd}
\ol{\MM_{0,n}}(\mb{P}^r,d) \arrow[r,"{\rm ev}"] \arrow[d,"\pi"] & (\mb{P}^r)^n\\
\ol{\MM_{0,n}}
\end{tikzcd}
\end{center}
Let $X \subset \overline{\MM_{0,n}}(\mb{P}^r,d)$ be a $GL_{r+1}$-equivariant cycle. The \emph{convolution operation} $\operatorname{conv}_{X,d}:A^\bullet_{GL_{r+1}}((\mb{P}^r)^{n-1})\to A^\bullet_{GL_{r+1}}(\mb{P}^r)$ is defined by \[\conv_{X,d}(x) = (\ev_n)_*([X] \cap(\ev_{1}, \ldots, \ev_{n-1})^* x).\] 
The $\star$-product of small equivariant quantum cohomology allows one to compute the operation $\conv_{\pi^{-1}(\mc{C}),d}$ for $\mc{C} \in \ol{\MM_{0, n}}$, which can be geometrically interpreted as follows. Given effective cycles $X_1, \ldots, X_{n-1}$ in the projectivization of the universal rank $r+1$ vector bundle, the convolution $\conv_{\pi^{-1}(\mc{C}),d}([X_1] \otimes \ldots \otimes [X_{n-1}])$ should be thought of as the Chow class of the cycle swept out by $\psi(p_n)$ as $\psi: \mc{C} \to \mb{P}^r$ ranges over all degree $d$ rational maps to fibers of the bundle such that $\psi(p_i) \in X_i$ for $1 \leq i \leq n-1$.

As $\pi$ is flat \cite[Remark 2.6.8]{Kock}, $\conv_{\pi^{-1}(\mathcal{C}),d}$ is independent of the choice of $\mathcal{C}$. More generally, the operation $\conv_{X,d}$, where $X$ is an irreducible component of $\pi^{-1}(\mc{C})$, can also be computed using small equivariant quantum cohomology, and if $\pi^{-1}(\mc{C})$ has irreducible components $X_1, \ldots, X_k$, then $\sum_{i=1}^k \conv_{X_i, d}=\conv_{\pi^{-1}(\mc{C}), d}$. The fact that the components of $\pi^{-1}(\mc{C})$ have multiplicity one can be deduced either from the deformation theory techniques of \cite[Section 5]{FP} or from the description in \cite[Section 4.1]{Conformal}. 

In \Cref{expandquantum}, we will see how for a fixed $\mathcal{C}$ the data of $\conv_{\pi^{-1}(\mc{C}),d}$ for varying $d$ is encoded in the $\star$-product on $QH^\bullet_{GL_{r+1}}(\mb{P}^r)$.

\subsection{Connecting the Kontsevich mapping space with $\mu_M$}
The operation $\conv_X$ is Kronecker dual (in the first $n-1$ factors of $\mb{P}^r$) to the class  \[[\ev_*X] \in A^\bullet_{GL_{r+1}}((\mb{P}^r)^{n}) = A^\bullet_{GL_{r+1}}((\mb{P}^r)^{n-1}) \otimes A^\bullet_{GL_{r+1}}(\mb{P}^r).\]

We now make a fundamental observation, apparently new in this level of generality. If $X$ is an irreducible component of a fiber $\pi^{-1}(\mc{C})$ of $\pi$, then $\ev_*X$ can be expressed as $\op{M}$ for an explicit $d \times n$ matrix $M$, which we show in \Cref{KGO}.
In the case $d=r+1$, this is reflected in the fact that $\ol{\MM_{0,n}}$ embeds into the Chow quotient $(\mb{P}^r)^n//_{Ch}SL_{r+1}$ \cite[Section 3.2]{Conformal}.

\begin{Lem}
\label{KGO}
Let $Z\subset \ol{\MM_{0,n}}(\mb{P}^r,d)$ be a component of a fiber $\pi^{-1}(\mathcal{C})$. Then $\ev_*Z=\op{M}$ for some $(d+1)\times n$ matrix $M$.
\end{Lem}

\begin{proof}
Suppose that $\mathcal{C}=[(C,p_1,\ldots,p_n)]$ and let $C_1, \ldots, C_k$ be the components of $C$. As one might expect \cite[Section 4.1]{Conformal}, there is an open locus $Z^\circ \subset Z$ and integers $d_1, \ldots, d_k$ such that $Z^\circ$ parameterizes maps $C \to \mb{P}^r$ that restrict to degree $d_i$ on each $C_i$.

Let $L$ be the line bundle on $C$ that restricts to $\ms{O}_{\mb{P}^1}(d_v)$ on each component $C_v\cong \mb{P}^1$ and let $V = H^0(C, L)$. Then $\dim(V) = d+1$ and we have a map $f: C\to \mb{P}(V^{\vee})$. 

Let $W = K^{r+1}$. Each rational map $C\to \mb{P}^r = \mb{P}(W^\vee)$ that restricts to degree $d_v$ on each component $C_v$ can be written uniquely as a composition $C \xrightarrow{f} \mb{P}(V^{\vee})\dashrightarrow \mb{P}(W^\vee)$, where the map $\mb{P}(V^{\vee})\dashrightarrow \mb{P}(W^\vee)$ is induced by a linear map $W \to V$. This is a regular map if and only if the linear system $W\to V = H^0(C, L)$ is basepoint-free.

Let $U\subset \mb{P}({\rm Hom}(W,V))$ be the open locus of basepoint-free linear systems $W\to V\cong H^0(C,L)$. Then $U$ is nonempty and by the above it embeds into $Z^\circ$. We have a diagram
\begin{center}
\begin{tikzcd}
U \arrow[d,hook] \arrow[r,hook] &Z \arrow[d,"\ev|_Z"]  \\
\mb{P}({\rm Hom}(W, V)) \arrow[r,dashed,"\mu_M"] & (\mb{P}^r)^n.
\end{tikzcd}
\end{center}
So $\mb{P}({\rm Hom}(W, V))$ and $Z$ share the open set $U$, and the bottom map is $\mu_M$ where $M$ is any matrix with projectivized columns $f(p_1), \ldots, f(p_n) \in \mb{P}(V^\vee)$. The result now follows.
\end{proof}

\subsection{Degenerating the domain curve}
\label{DTDC}
By the proof of \Cref{KGO}, we have $\op{M_{d+1}}=\ev_*\pi^{-1}(\mc{C})$ if $\mc{C} = [(\mb{P}^1, p_1, \ldots, p_n)]$, where $p_1, \ldots, p_n \in \mb{P}^1$ have the same moduli as in \Cref{SKS}.
Since $\pi$ is flat, the same equation holds for any $\mc{C}$, so we can choose $\mc{C} = [(C_0,p_1,\ldots,p_n)]$ to be a chain of $n-2$ rational curves. 

\begin{center}
\begin{tikzpicture}

\draw[gray, thick] (0,0) -- (3,1);
\draw[gray, thick] (2,1) -- (5,0);
\filldraw[black] (2/3,2/9) circle (1.5pt) node[anchor=south]{$p_1$};
\filldraw[black] (3/2,1/2) circle (1.5pt) node[anchor=south]{$p_2$};
\filldraw[black] (7/2,1/2) circle (1.5pt) node[anchor=south]{$p_3$};

\draw[gray, thick] (4,0) -- (5,1/3);
\filldraw[black] (11/2,1/2) circle (1pt) node{};
\draw[gray, thick] (8,1/3) -- (9,0);
\filldraw[black] (15/2,1/2) circle (1pt) node{};
\filldraw[black] (13/2,1/2) circle (1pt) node{};

\draw[gray, thick] (8,0) -- (11,1);
\filldraw[black] (19/2,1/2) circle (1.5pt) node[anchor=south]{$p_{n-2}$};
\draw[gray, thick] (10,1) -- (13,0);
\filldraw[black] (23/2,1/2) circle (1.5pt) node[anchor=south]{$p_{n-1}$};
\filldraw[black] (37/3,2/9) circle (1.5pt) node[anchor=south]{$p_n$};

\end{tikzpicture}
\end{center}

Let $\Gamma_0$ be the dual graph of $C_0$, which is a chain of $n-2$ vertices connected by edges. Then the fiber of $\pi$ over $(C_0,p_1,\ldots,p_n)$ has $\binom{n-3+d}{d}$ components, indexed by assignments of degrees to the $n-2$ vertices of $\Gamma_0$ that sum to $d$. However, most components push forward to zero under ${\rm ev}_{*}$ for dimension reasons, and only $\binom{n-2}{d}$ components, corresponding to assignments of degrees of 0's and 1's, survive. 

Therefore, the cycle $\op{M_{d+1}}$ degenerates to $\binom{n-2}{d}$ generalized matrix orbits by \Cref{KGO}.

\subsection{Expanding out a small quantum product} \label{expandquantum}
To make use of the observations in \Cref{DTDC}, we recall the standard geometric definition of the small quantum cohomology ring of projective space. The following is the usual derivation of the WDVV relations for $\mathbb{P}^r$, but we connect the computation to the $\opc{M_d}$ classes explicitly. Recall that we defined $QH^\bullet_{GL_{r+1}}(\mb{P}^r)$ to be $A^\bullet_{GL_{r+1}}(\mb{P}^r)[\hbar]$ with a multiplication $\star$ deforming the usual multiplication on $A^\bullet_{GL_{r+1}}(\mb{P}^r)$. In the remainder of this subsection, we will be working with the following equivalent geometric definition of $\star$.

\begin{Def}
\label{QCD}
Given $f,g\in A^\bullet_{GL_{r+1}}(\mb{P}^r)$, define $f\star g$ in $QH^\bullet_{GL_{r+1}}(\mb{P}^r)$ by
\begin{align*}
f\star g&=
\left(\tikz[baseline=15pt]{ \draw(1,0)--(0,1);\filldraw[black] (1/4,3/4) circle (1pt) node[anchor=south]{$f$}; \filldraw[black] (1/2,1/2) circle (1pt) node[anchor=south]{$g$} node[anchor=north east]{$\boxed 0$};\filldraw[black] (3/4,1/4) circle (1pt) node[anchor=south]{$$}}\right) +
\left(\tikz[baseline=15pt]{\draw(1,0)--(0,1);\filldraw[black] (1/4,3/4) circle (1pt) node[anchor=south]{$f$}; \filldraw[black] (1/2,1/2) circle (1pt) node[anchor=south]{$g$} node[anchor=north east]{\boxed 1};\filldraw[black] (3/4,1/4) circle (1pt) node[anchor=south]{$$}}\right) \hbar
\end{align*}
where $\tikz[baseline=15pt]{\draw(1,0)--(0,1);\filldraw[black] (1/4,3/4) circle (1pt) node[anchor=south]{$f$}; \filldraw[black] (1/2,1/2) circle (1pt) node[anchor=south]{$g$} node[anchor=north east]{\boxed i};\filldraw[black] (3/4,1/4) circle (1pt) node[anchor=south]{$$}}$ denotes $\conv_{\pi^{-1}(\pt),i}(f,g)$ for $\pt=\ol{\MM_{0,3}}$. In characteristic zero this agrees with the definition from \Cref{QCRD} which is the definition used exclusively outside this section.
\end{Def}

We now show geometrically that $\star$ is commutative and associative, and interpret $[\hbar^k](f_1\star \cdots \star f_n)$. Note that $[\hbar^0](f \star g)=fg$ since $\ol{\MM_{0,3}}(\mb{P}^r,0)=\mb{P}^r$, so $\star$ will thus yield a commutative associative deformation of the multiplication on $A^\bullet_{GL_{r+1}}(\mb{P}^r)$.

\begin{Rem}
Naively, if $f$ and $g$ are effective cycles, then $\tikz[baseline=15pt]{\draw(1,0)--(0,1);\filldraw[black] (1/4,3/4) circle (1pt) node[anchor=south]{$f$}; \filldraw[black] (1/2,1/2) circle (1pt) node[anchor=south]{$g$} node[anchor=north east]{\boxed i};\filldraw[black] (3/4,1/4) circle (1pt) node[anchor=south]{$$}}$ is the locus in $\mb{P}^r$ swept out by all degree $i$ curves passing through the two cycles. This expression vanishes for $i>1$ for dimension reasons. 
\end{Rem}

We can expand products under $\star$ by applying \Cref{QCD} repeatedly. For example,
\begin{align*}
(f\star g)\star h &=
\left(\left(\tikz[baseline=15pt]{ \draw(1,0)--(0,1);\filldraw[black] (1/4,3/4) circle (1pt) node[anchor=south]{$f$}; \filldraw[black] (1/2,1/2) circle (1pt) node[anchor=south]{$g$} node[anchor=north east]{\boxed 0};\filldraw[black] (3/4,1/4) circle (1pt) node[anchor=south]{$$}}\right) +
\left(\tikz[baseline=15pt]{\draw(1,0)--(0,1);\filldraw[black] (1/4,3/4) circle (1pt) node[anchor=south]{$f$}; \filldraw[black] (1/2,1/2) circle (1pt) node[anchor=south]{$g$} node[anchor=north east]{\boxed 1};\filldraw[black] (3/4,1/4) circle (1pt) node[anchor=south]{$$}}\right) \hbar\right)\star h \\
&=
\left(\tikz[baseline=15pt]{ \draw(1,0)--(0,1); \draw(1/2,0)--(3/2,1);\filldraw[black] (1/4,3/4) circle (1pt) node[anchor=south]{$f$}; \filldraw[black] (1/2,1/2) circle (1pt) node[anchor=south]{$g$} node[anchor=north east]{\boxed 0};\filldraw[black] (3/4,1/4) circle (1pt) node[anchor=south]{$$}; \filldraw[black] (1,1/2) circle (1pt) node[anchor=south]{$h$} node[anchor=north west]{\boxed 0}; \filldraw[black] (1+1/4,3/4) circle (1pt) node[anchor=south]{$$}}\right) + 
\left(
\tikz[baseline=15pt]{ \draw(1,0)--(0,1); \draw(1/2,0)--(3/2,1);\filldraw[black] (1/4,3/4) circle (1pt) node[anchor=south]{$f$}; \filldraw[black] (1/2,1/2) circle (1pt) node[anchor=south]{$g$} node[anchor=north east]{\boxed 1};\filldraw[black] (3/4,1/4) circle (1pt) node[anchor=south]{$$}; \filldraw[black] (1,1/2) circle (1pt) node[anchor=south]{$h$} node[anchor=north west]{\boxed 0}; \filldraw[black] (1+1/4,3/4) circle (1pt) node[anchor=south]{$$}}
+
\tikz[baseline=15pt]{ \draw(1,0)--(0,1); \draw(1/2,0)--(3/2,1);\filldraw[black] (1/4,3/4) circle (1pt) node[anchor=south]{$f$}; \filldraw[black] (1/2,1/2) circle (1pt) node[anchor=south]{$g$} node[anchor=north east]{\boxed 0};\filldraw[black] (3/4,1/4) circle (1pt) node[anchor=south]{$$}; \filldraw[black] (1,1/2) circle (1pt) node[anchor=south]{$h$} node[anchor=north west]{\boxed 1}; \filldraw[black] (1+1/4,3/4) circle (1pt) node[anchor=south]{$$}}
\right)\hbar
+
\left(
\tikz[baseline=15pt]{ \draw(1,0)--(0,1); \draw(1/2,0)--(3/2,1);\filldraw[black] (1/4,3/4) circle (1pt) node[anchor=south]{$f$}; \filldraw[black] (1/2,1/2) circle (1pt) node[anchor=south]{$g$} node[anchor=north east]{\boxed 1};\filldraw[black] (3/4,1/4) circle (1pt) node[anchor=south]{$$}; \filldraw[black] (1,1/2) circle (1pt) node[anchor=south]{$h$} node[anchor=north west]{\boxed 1}; \filldraw[black] (1+1/4,3/4) circle (1pt) node[anchor=south]{$$}}
\right)\hbar^2. 
\end{align*}
Here $\tikz[baseline=15pt]{ \draw(1,0)--(0,1); \draw(1/2,0)--(3/2,1);\filldraw[black] (1/4,3/4) circle (1pt) node[anchor=south]{$f$}; \filldraw[black] (1/2,1/2) circle (1pt) node[anchor=south]{$g$} node[anchor=north east]{\boxed i};\filldraw[black] (3/4,1/4) circle (1pt) node[anchor=south]{$$}; \filldraw[black] (1,1/2) circle (1pt) node[anchor=south]{$h$} node[anchor=north west]{\boxed j}; \filldraw[black] (1+1/4,3/4) circle (1pt) node[anchor=south]{$$}}$ denotes the convolution $\conv_{Z_{i,j},i+j}(f,g,h)$, where $Z_{i, j}\subset \ol{\MM_{0,4}}(\mb{P}^r,i+j)$ is defined as follows. Let $\mathcal{C} = [(C, p_1, p_2, p_3, p_4)]$, where $C = C_1 \cup C_2$ be the reducible curve in $\ol{\MM_{0,4}}$ with $p_1, p_2 \in C_1$ and $p_3, p_4 \in C_2$. Then $Z_{i, j}$ is the component of the fiber of $\pi^{-1}(\mathcal{C})$ generically parameterizing maps $C \to \mb{P}^r$ that restrict to degree $i$ on $C_1$ and degree $j$ on $C_2$.

Here, we are using the fact that the composition of convolution operations corresponds to the gluing morphism between Kontsevich spaces, so that 
\begin{align*}
\conv_{Z_{i, j}, i+j}(f,g,h)=\conv_{\pi^{-1}(\pt),j}(\conv_{\pi^{-1}(\pt),i}(f,g),h).
\end{align*}
There are no issues with multiplicities because the gluing morphism will always be an isomorphism 
in our case \cite[Lemma 12 (i)]{FP}.

Observe that, for any $d$, the coefficient of $\hbar^d$ in the expansion of $(f \star g) \star h$ is equal to \[\sum_{i+j=d} \tikz[baseline=15pt]{ \draw(1,0)--(0,1); \draw(1/2,0)--(3/2,1);\filldraw[black] (1/4,3/4) circle (1pt) node[anchor=south]{$f$}; \filldraw[black] (1/2,1/2) circle (1pt) node[anchor=south]{$g$} node[anchor=north east]{\boxed i};\filldraw[black] (3/4,1/4) circle (1pt) node[anchor=south]{$$}; \filldraw[black] (1,1/2) circle (1pt) node[anchor=south]{$h$} node[anchor=north west]{\boxed j}; \filldraw[black] (1+1/4,3/4) circle (1pt) node[anchor=south]{$$}} = \sum_{i+j=d} \conv_{Z_{i, j}, d}(f,g,h)=\conv_{\pi^{-1}(\mathcal{C}), d}(f,g,h).\] as the $Z_{i,j}$ for $i + j = d$ are exactly the components of the fiber $\pi^{-1}(\mathcal{C})$. By the flatness of $\pi$, the class of $\pi^{-1}(\mathcal{C})$ is unchanged if we redistribute the points $p_1,\ldots,p_4$, implying $(f \star g) \star h$ does not depend on the ordering of $f$, $g$, and $h$, so $\star$ is commutative and associative.

In turn, the discussion in \Cref{DTDC} shows that (taking $(\mb{P}^r)^4\to \mb{P}^r$ the last projection) \[\conv_{\pi^{-1}(\mathcal{C}), d}(f,g,h)=\int_{(\mb{P}^r)^4 \to \mb{P}^r}\opc{M_{d+1}} f(H_1) g(H_2) h(H_3),\] where $M_{d+1}$ is a general $(d+1) \times 4$ matrix. In general, the same reasoning shows $$f_1 \star \cdots \star f_{n-1}=\sum_{d=0}^n\hbar^d\conv_{\pi^{-1}(\mc{C}),d}(f_1,\ldots,f_{n-1})$$ thus deducing the following theorem, which we will use to find $\opc{M_{d+1}}$.
\begin{Thm}
\label{GOSP}
We have
\begin{align*}
f_1\star \cdots\star f_{n-1} &= \sum_{d=0}^{n-1} \left(\hbar^d\int_{(\mb{P}^r)^{n}\to \mb{P}^r}{\opc{M_{d+1}}f_1(H_1)\cdots f_{n-1}(H_{n-1})}\right).
\end{align*}
\end{Thm}

\begin{Rem}
This holds by 
\Cref{Mhbarintro} as follows. The equation is equivalently written as $$[\ast]_{\hbar}(f_1)\star [\ast]_{\hbar}(f_2) \star \ldots \star [\ast]_{\hbar}(f_{n-1})=[I_{n-1}]_{\hbar}(f_1,\ldots,f_{n-1}),$$ and the two sides are equal again by repeatedly applying \Cref{Mhbarintro}. By the simple presentation of $QH^\bullet_{GL_{r+1}}(\mb{P}^r)$ described in Section \ref{QCRD}, computing $f_1\star \cdots\star f_n$ reduces to polynomial multiplication and expansion in terms of powers of the polynomial $F(H)$.
\end{Rem}

\subsection{Computing $\opc{M}$ for a general $d \times n$ matrix}Recall the formulas from \Cref{QCRD}.
Note that $\int_{\mb{P}^r}[\hbar]f(H)\star g(H)=[z^r][F(z)]\overline{f}(z)\overline{g}(z)=0$ for degree reasons. Applying $\star f_n$ to the right of both sides of \cref{GOSP} and applying $\int_{\mb{P}^r}$, we deduce that \[\int_{(\mb{P}^r)^n}\opc{M_{d+1}}f_1(H_1) \cdots f_n(H_n)=[z^r][F(z)^{d}]\overline{f_1}(z)\cdots \overline{f_n}(z).\] 
Note if we assume \Cref{Mhbarintro}, then this also follows from
$$\opc{M_{d+1}}^{\dagger}=\opc{\tau^{\le d+1}I_{n}}^{\dagger}=[z^r][F(z)^{d}][I_{n}]_{\hbar}.$$
We will use this to construct the formula for $\opc{M_{d+1}}$.
\begin{Thm} \label{opcmd}
The equivariant Chow class $\opc{M_{d+1}}$ is given by
\begin{align*}
\opc{M_{d+1}}=[z^r][F(z)^{d}]\prod_{i=1}^n \frac{F(z)-F(H_i)}{z-H_i}.
\end{align*}
\end{Thm}

\begin{proof}
For any polynomial $f$ with coefficients in $A^\bullet_{GL_{r+1}}(\pt)$, we have
\begin{align*}
\int_{\mb{P}^r} \frac{F(z)-F(H)}{z-H}f(H)&=\int_{\mb{P}^r} \frac{F(z)-F(H)}{z-H}\overline{f}(H) \\
&= \int_{\mb{P}^r} \frac{F(z)-F(H)}{z-H} \overline{f}(z) +  \int_{\mb{P}^r} \frac{F(z)-F(H)}{z-H}(\overline{f}(H)-\overline{f}(z)) \\
&=\int_{\mb{P}^r} \frac{F(z)-F(H)}{z-H} \overline{f}(z) + \int_{\mb{P}^r} \frac{\overline{f}(z)-\overline{f}(H)}{z-H}(F(H)-F(z)) \\
&= \int_{\mb{P}^r}\frac{F(z)-F(H)}{z-H}\overline{f}(z) + 0 - 0 \\
&= \overline{f}(z);
\end{align*} that is, $\frac{F(z)-F(H)}{z-H}\in A^\bullet_{GL_{r+1}}(\mb{P}^r)\otimes QH^\bullet_{GL_{r+1}}(\mb{P}^r)$ is Kronecker dual to the lifting map $A^\bullet_{GL_{r+1}}(\mb{P}^r)\to QH^\bullet_{GL_{r+1}}(\mb{P}^r)$. 
Hence $$\int_{(\mb{P}^r)^n}\prod_{i=1}^n \frac{F(z)-F(H_i)}{z-H_i}f_i(H_i)=\prod_{i=1}^n \overline{f_i}(z).$$
Applying $[z^r][F(z)^{d}]$ to both sides of this equation yields the desired result.
\end{proof}
\begin{Cor} The expression $$\sum_{|A|=d+1}\left(\prod_{j \not \in A}F(z_j)\right)\sum_{i \in A}\prod_{j \ne i}\frac{(-1)^{n-d+1}}{z_i-z_j}$$ is a polynomial in $z_1,\ldots,z_r$, and $\opc{M_{d+1}}$ is obtained by evaluating this polynomial at $z_i=H_i$.
\end{Cor}
\begin{proof}
Apply partial fraction decomposition to the denominator of the expression in \Cref{opcmd}. This will also follow from the general expression in \Cref{finkformulas}. 
\end{proof}

\section{What is known non-equivariantly}\label{noneqsection}
In this section we recall from the literature the results on non-equivariant Chow classes associated to generalized matrix orbit maps that we will be extending equivariantly.

Non-equivariantly, there is a good understanding of generalized matrix orbits and more. It has been shown in \cite{Binglin} that we can express the non-equivariant Chow class of the image of $\mu:\mb{P}(W) \dashrightarrow \prod \mb{P}(W_i)$ induced by linear maps $\phi_i:W \to W_i$ in terms of the data of the \emph{polymatroid} of the subspaces $U_i=\phi_i^*(W_i^\vee) \subseteq W^\vee$, i.e. the data of $\dim(\sum_{i \in A} U_i)$ for all subsets $A \subseteq [n]$. 
To be more precise, consider a resolution of $\mu$ given by the diagram
\begin{center}
\begin{tikzcd}
B \arrow[d,"\pi"] \arrow[dr, "\tilde{\mu}"] & \\
\mb{P}(W) \arrow[r, dashed, "\mu"] & \prod \mb{P}(W_i).
\end{tikzcd}
\end{center}
Then analogously to generalized matrix orbits as in \Cref{Setup}, for a non-equivariant class $\alpha\in A^\bullet(\mb{P}(W))$, we define $\mu_*(\alpha):=\wt{\mu}_*\pi^*\alpha$, which is independent of $B$.

\begin{Thm}[{\cite[Theorem 1.1]{Binglin}}]\label{section5binglin}
Let $\phi_i:W \to W_i$ be nonzero linear maps, and let $U_i = \phi^*(W_i^\vee)\subseteq W_i^\vee $. Consider the rational map $$\mu:\mb{P}(W) \dashrightarrow \prod \mb{P}(W_i).$$ Then we have $$\mu_*(H^j)=\sum \prod_{i=1}^n H_i^{\dim(W_i)-1-e_i}$$ where the sum is over all tuples $(e_1,\ldots,e_n) \in \mb{Z}^n$ such that
\begin{align*}
&\sum_{i \in A}e_i < \dim \sum_{i \in A} U_i\qquad\text{for all nonempty $A\subseteq [n]$ and}\\
&\sum_{i=1}^n e_i=\dim(W)-1-j.
\end{align*}
\end{Thm}

Recall that given a $d \times n$ matrix $M$ with no zero columns, we defined the map $\mu_M:\mb{P}^{(r+1) \times d-1} \dashrightarrow (\mb{P}^r)^n$ by matrix multiplication $A \mapsto AM$ of $(r+1)\times d$ matrices with $M$. Letting $W=K^d$ and $V=K^n$, the columns of $M$ are elements of $W$, and hence induce evaluation functionals $W^\vee \to K$. By tensoring with $V$, we get induced maps $\phi_i:V \otimes W^\vee \to V$. The generalized matrix orbit map $\mu_M$ is then induced as above by the linear maps $\phi_i$.


From this description of $\phi_M$, the following follows immediately from \Cref{section5binglin}. Equivariant generalizations are proved in \Cref{formulasection,graphclosuresection}.

\begin{Thm}
\label{NE} Let $M$ be a $d\times n$ matrix. Then the non-equivariant class $\opc{M}$ is given by $$\opc{M}=\sum \prod_{i=1}^n H_i^{r-e_i}$$
where the sum is over all tuples $(e_1, \ldots, e_n)\in \mb{Z}^n$ such that
\begin{align*}
&\sum_{i \in A} e_i < (r+1)\rk(A)\qquad\text{for all nonempty $A\subset [n]$ and} \\
&\sum_{i=1}^n e_i = (r+1)d-1. \\
\intertext{Equivalently, fix $0<\epsilon_1,\ldots,\epsilon_n$ with $\sum_{i=1}^n \epsilon_i=1$. Then the sum is over all $(e_1, \ldots, e_n)\in \mb{Z}^n$ such that}
&(e_1+\epsilon_1, \ldots, e_n + \epsilon_n) \in (r+1)P_M. 
\end{align*}
\end{Thm}

\begin{Thm}\label{NEgraphclosure}
Let $M$ be a $d \times n$ matrix. Then the non-equivariant class of the graph closure $[\ol{\Gamma_M}]\in A^\bullet(\mb{P}^{(r+1)\times d-1}\times (\mb{P}^r)^n)$ of $\mu_M$ is given by $$[\ol{\Gamma_M}]=\sum H^{\sum e_i} \prod_{i=1}^n H_i^{r-e_i}$$
where the sum is over all tuples $(e_1, \ldots, e_n)\in \mb{Z}^n$ such that
\begin{align*}
&\sum_{i \in A} e_i < (r+1)\rk(A)\qquad\text{for all nonempty $A\subset [n]$.}\\
\intertext{Equivalently, fix $0<\epsilon_1,\ldots,\epsilon_n$ with $\sum_{i=1}^n \epsilon_i\leq 1$. Then the sum is over all $(e_1, \ldots, e_n)\in \mb{Z}^n$ such that}
&(e_1+\epsilon_1, \ldots, e_n + \epsilon_n) \in (r+1)I_M.
\end{align*}
\end{Thm}
We note the following immediate corollary of \cref{NE}.
\begin{Cor}
The non-equivariant class $\opc{M}$ depends only on the matroid of $M$.
\end{Cor}

\section{One-parameter degenerations of generalized matrix orbits}\label{degenerations}
In this section, describe how the cycles $\op{M}$ degenerate. As a consequence, we obtain relations between the equivariant Chow classes $\opc{M}$ which will be necessary for their later computation.

Degenerations of toric varieties are governed by polytopal geometry, and in particular ``regular subdivisions'' of their moment polytopes \cite{GelfandZelevinsky, GrossSeibert}. In particular, Kapranov \cite{Kapranov} studied the Chow quotient $\Gr(r+1,n)//(K^\times)^n$  and limits of torus orbits in the Grassmannian using polytope subdivisions via the Pl\"ucker embedding.  Using his isomorphism of Chow quotients $\Gr(r+1,n)//(K^\times)^n \cong (\mb{P}^r)^n//GL_{r+1}$, general orbits under degeneration split up into certain ``sibling orbits'' in the special fiber. Speyer investigated this combinatorially \cite{SpeThesis, TropicalLinearSpaces}, developing a notion of tropical linear space to study these degenerations.

We adopt the following conventions.
\begin{enumerate}
\item
Let $D$ be a smooth, affine, finite-type $K$-scheme of dimension 1 (we will only use the case that $D\subset \mb{A}^1$ is an open set).
\item
Let $0\in D$ be a closed point and let $t$ be a uniformizing parameter around 0. By replacing $D$ with a suitable open subset, we may assume that $t$ is regular and vanishes only at $0$.
\item Let $M_D: D\to \mb{A}^{d\times n}$ be a morphism from $D$ to the space of matrices that maps the generic point to a rank $d$ matrix with no identically zero column.
\item
Let $P$ be the rank polytope corresponding to the image of the generic point of $D$ under $M_D$. Equivalently, it is the rank polytope of $M_D$ evaluated at a general point.
\end{enumerate}

The generalized matrix orbits $\op{M_D(p)} \subset (\mb{P}^r)^n$ form a flat family over some open subset $U \subset D\setminus \{0\}$. We are interested in extending this family to a flat family of subvarieties of $(\mb{P}^r)^n$ over $U \cup \{0\}$. In other words, we are interested in the flat limit of the generalized matrix orbits $\op{M_D(p)}$ as $p \to 0$.

We will show that the components of this flat limit have multiplicity $1$ and are of the form $\op{M_1}, \ldots, \op{M_k}$ for certain $d\times n$ matrices $M_1, \ldots, M_k$. We can take $M_1 =M_D(0)$ if $\op{M_D(0)}$ is nonzero. \Cref{degenerationfibers} below, which is the main result of this section, describes the sibling orbits $\op{M_1}, \ldots, \op{M_k}$. It states that the matrices $M_i$ arise from a particular ``regular subdivision'' of $P$, and they are all what we call ``orb-limits'' of $M_D$. 

\begin{Def}
Given a convex polytope $P$, a \emph{subdivision} of $P$ is a decomposition $P=P_1 \cup \ldots \cup P_k$ into convex polytopes of the same dimension where each pairwise intersection $P_i \cap P_j$ is either empty or a common proper face of $P_i$ and $P_j$. The faces of the $P_i$ are called the \emph{faces} of the subdivision, and the polytopes $P_i$ are called the \emph{facets} of the subdivision.
\end{Def}


\begin{Def}
\label{regularsubdef}
Let $P$ be a (possibly unbounded) convex polytope, and let $h:P \to \mb{R}$ be a piecewise affine convex function. Denote by $P^h$ for the lower convex hull of $\{(p,h(p))\mid p \in P\}\subset P \times \mb{R}$. The facets of $P^h$ each have dimension $\dim(P)$, and the projections of these facets to $P$ form a subdivision of $P$, called the \emph{regular subdivision of $P$ with lifting function $h$}.

If $P$ is bounded, then given a function $h:\Vt(P)\to \mb{R}$, we define $P^h$ as the lower convex hull of $\{(v,h(v)) \mid v \in \Vt(P)\} \subset P \times \mb{R}$ and as above use the facets of $P^h$ to define the regular subdivison of $P$ with lifting function $h$. Extending $h$ to a piecewise affine convex function on all of $P$ by defining $h(p) \in \mb{R}$ as the unique value such that $(p,h(p)) \in P^h$ induces the same $P^h$ and regular subdivision.


\end{Def}
\begin{Rem}
Adding an affine function to $h$ or multiplying it by a positive scalar yields the same decomposition of $P$.
\end{Rem}
\begin{Def}
Define the height function $h_{M_D} : \Vt(P) \to \mb{Z}$ by \[h_{M_D}(v) = \operatorname{ord}_t \det(M_D^v)\] where for a matrix $M$ and a subset of columns $v$, we define $M^v$ to be the restriction of $M$ to the $v$ columns.
\end{Def}
\begin{Def}
\label{orblimitdef}
We say that $M_D$ \emph{simply orb-limits} to a matrix $M$ if there exist maps $\gamma_T: D\setminus\{0\}\to T$ and $\gamma_{GL}: D\setminus\{0\}\to GL_{d}$ such that $M=\lim_{p \to 0}\gamma_{GL}(p) M_D(p) \gamma_T(p) $. We say that $M_D$ \emph{orb-limits} to a matrix $M$ if $M_D$ simply orb-limits to $M$ after pulling back to a finite cover by a nonsingular curve (possibly everywhere ramified) of an open neighborhood of $0\in D$. 
\end{Def}

We note that if $M_D$ orb-limits or simply orb-limits to a matrix $M$, then it also does so after pulling back to a finite cover by a nonsingular curve (possibly everywhere ramified) of an open neighborhood of $0\in D$.

Recall that $\op{M}$ was defined in such a way that it is either zero or of dimension $(r+1)d-1$.

\begin{Thm}\label{degenerationfibers}
Let $D$ be a smooth affine $1$-dimensional $K$-scheme of finite type, let $0 \in D$ be a closed point, and let $t$ be a uniformizing parameter for $D$ around $0$. Let $M_D : D \to \mb{A}^{d \times n}$ be a morphism that maps the generic point of $D$ to a rank $d$ matrix with no identically zero column.

Consider the map $\psi: \mb{P}^{(r+1)\times d-1}\times D \dashrightarrow (\mb{P}^r)^n\times D$ given by $(\widetilde A,p) \mapsto (\widetilde{AM_D(p)},p)$. Let $U\subset \mb{P}^{(r+1)\times d-1}\times (D \setminus \{0\})$ be a dense open subset of the domain of definition of $\psi$, and let $\wt{\mc{O}}$ be the family $\overline{\psi(U)} \subset (\mb{P}^{r})^n\times D$. Then
\begin{enumerate}[(a)]
\item The family $\wt{\mc{O}}$ is locally flat over $0 \in D$.
\item The fiber of $\wt{\mc{O}}$ over a general point $p \in D$ set-theoretically contains $\op{M_D(p)}$, and it is equal (as an algebraic cycle) to $\op{M_D(p)}$ if it is $((r+1)d-1)$-dimensional.
\item Suppose that $M_D$ orb-limits to matrices $M_1, \ldots, M_k$ whose rank polytopes are the facets of the regular subdivision of $P$ induced by the lifting function $h_{M_D}$. (The existence of the $M_i$ is guaranteed by \cref{lpath}.) Then the fiber of $\wt{\mc{O}}$ over $0 \in D$ is equal to
\[\bigcup_{i=1}^{k}\op{M_i}\]
as an $((r+1)d-1)$-dimensional algebraic cycle.
\end{enumerate}
\end{Thm}

Because the fibers of a flat family of cycles in $(\mb{P}^r)^n$ have the same equivariant Chow class, \Cref{degenerationfibers} has the following immediate corollary.

\begin{Cor}\label{siorbit}
Suppose that $M_D$ orb-limits to matrices $M_1, \ldots, M_k$ whose rank polytopes are the facets of the regular subdivision of $P$ induced by the lifting function $h_{M_D}$. (The existence of the $M_i$ is guaranteed by \cref{lpath}.) Let $p\in D$ be a general point. Then we have the equality of equivariant Chow classes
\begin{align*}
\opc{M_D(p)}=\sum_{i=1}^{k}\opc{M_i}.
\end{align*}
\end{Cor}

The remainder of this section is dedicated to the proof of \Cref{degenerationfibers}.

\subsection{Non-equivariant additivity}
We now prove that the non-equivariant class $\opc{M}$ is additive (See \Cref{matroidpolytopedef}), which is needed in the proof of \Cref{degenerationfibers}. In the next section, we will use \Cref{degenerationfibers} to show that the equivariant class $\opc{M}$ is also additive.

\begin{Def}\label{additivedef}
Let $d, n \geq 0$ be integers. Let $T(d,n)$ to be the abelian group generated by the indicator functions $1_{P_M}$ of rank polytopes of $d \times n$ matrices, modulo functions whose support has codimension at least $1$ in $\Delta$. For a $d \times n$ matrix $M$, we define $\langle M \rangle$ to be the class of $1_{P_M}$ in $T(d,n)$.

Let $Z$ be an abelian group. Note that a map of sets $\phi : K^{d \times n} \to Z$ is additive if there is an abelian group homomorphism $\tilde\phi : T(d,n) \to Z$ such that $\tilde\phi(\langle M \rangle) = \phi(M)$ for all matrices $M \in K^{d \times n}$. Observe that if such a map exists, it is unique. 
\end{Def}
\begin{Rem}
If $\phi$ is additive, then $\phi(M)$ depends only on the matroid of $M$, and $\phi(M) = 0$ if $P_M$ has positive codimension in $\Delta$. This in particular occurs if $\rk(M)<d$.
\end{Rem}
\begin{Rem}
If $P_M = P_{M_1} \cup \ldots \cup P_{M_k}$ is a subdivision, then $\langle M\rangle = \sum_{i=1}^k \langle M_i\rangle.$
\end{Rem}

\begin{Thm} \label{realvalnoneq} For each $r$, the association $M \mapsto \opc{M}$ of $d \times n$ matrices to non-equivariant Chow classes is additive.
\end{Thm}
\begin{proof}
Suppose $M_1,\ldots,M_k$ are $d \times n$ matrices with $\sum_{i=1}^k a_i\langle M_i \rangle=0$. By \cref{NE}, the coefficient of $\prod_{i=1}^{n} H_i^{r-e_i}$ in $\sum_{i=1}^k a_i \opc{M_i} $ is given by
\begin{align*}
\sum_{i=1}^k a_i 1_{(r+1) P_{M_i}} (e_1 + \epsilon_1, \ldots, e_n + \epsilon_n)
\end{align*}
for any $0<\epsilon_1,\ldots,\epsilon_n$ with $\sum \epsilon_i=1$.
Since $\sum_{i=1}^k a_i \langle M_i \rangle = 0$, the sum $\sum_{i=1}^k a_i 1_{(r+1) P_{M_i}}$ is supported on a subset of $(r+1)\Delta$ with codimension at least $1$ in $(r+1)\Delta$. Hence we can find $\epsilon_1, \ldots, \epsilon_n$ with $\sum_{i=1}^k a_i 1_{(r+1) P_{M_i}} (e_1 + \epsilon_1, \ldots, e_n + \epsilon_n) = 0$. Therefore $\sum_{i=1}^k a_i \opc{M_i} = 0$ as no monomial can appear in its expansion.
\end{proof}

\begin{Cor}\label{positivecodim}
The equivariant Chow class $\opc{M}$ is zero if $P_M$ has positive codimension in $\Delta$. Also, if $\opc{M}=0$ then $\opc{M'}=0$ for all $M'$ with the same matroid as $M$.  
\end{Cor}
\begin{proof}
For the first part, by \cref{realvalnoneq} we have the non-equivariant class $\opc{M}=0$. Since $\op{M}$ is effective, it follows that $\op{M} = 0$, so the equivariant class $\opc{M} = 0$.

For the second part, if $\opc{M}=0$, then $\opc{M'}=0$ non-equivariantly, so $\opc{M'}=0$ equivariantly by the same argument. 
\end{proof}
\begin{Rem}
The matrix $M$ is connected if and only if $P_M$ is full-dimensional in $\Delta$.
\end{Rem}

\subsection{Regular subdivisions arising from degenerations}\label{dsection}
The following lemma connects degenerations of generalized matrix orbits and regular subdivisions of rank polytopes.

\begin{Lem}
\label{lpath}
Let $F$ be a nonempty face of the regular subdivision of $P$ with lifting function $h_{M_D}$. Then $M_D$ orb-limits to some matrix $M$ with $P_M=F$.
\end{Lem}

\begin{proof}
As in the definition of a regular subdivision, extend $h_{M_D}$ to a function on all of $P$, and let $P^h \subset P\times \mb{R}$ be the lower convex hull of $\{(v,h_{M_D}(v))\mid v\in \Vt(P)\}$. Because $F$ is the projection of a face of $P^h$, there exists an affine function $L_{i,F}:\mb{R}^n\to \mb{R}$ defined over $\mathbb{Q}$ such that $L_F(p) \le h_{M_D}(p)$ for all $p \in P$ and $F = \{p \in P \mid L_{F}(p)=h_{M_D}(p)\}$. There exists a positive integer $N$ such that $L_F(x_1,\ldots,x_n)=\sum_{i=1}^n a_{i,F}x_i+c_F$ for constants $c_F,a_{i,F}\in \frac{1}{N}\mb{Z}$.

Let $\wt{D}\to D$ be a finite cover with $\wt{D}$ nonsingular and a choice of point $\wt{0}$ mapping to $0$ such that the uniformizing parameter $s$ around $\wt{0}$ has $\operatorname{ord}_{s}(t)=N$. Then the height function of the composition $M_{\wt{D}}: \tilde{D}\to D\to \mb{A}^{d\times n}$ is $Nh_{M_D}$. Therefore, by pulling back to this cover, we may assume $c_F,a_{i,F}\in \mb{Z}$ for all $1 \le i \le n$. 

Define $\gamma_T: D \setminus \{0\} \to T\cong  (K^{\times})^n$ by $(t^{-a_{1,F}},\ldots,t^{-a_{n,F}})$.  Let $\pi_{{\rm Gr}}: \mb{A}^{d\times n}\dashrightarrow \Gr(d,n)$ be the map sending a matrix $M$ to the span of the rows of $M$. Consider the morphism $M_D \gamma_T : D\setminus\{0\} \to \mb{A}^{d\times n}$ given by pointwise multiplication. The composite $\pi_{{\rm Gr}}\circ (M_D\gamma_T): D\setminus\{0\}\to \Gr(d,n)$ extends uniquely to $D$ because $\Gr(d,n)$ is projective.

The map $\pi$ restricted to rank $d$ matrices is a principal $GL_d$-bundle with Zariski-local trivialization. Hence there exists an open neighborhood $U\subset \Gr(d,n)$ of $(\pi_{{\rm Gr}}\circ (M_D\gamma_T))(0)$ such that there exists a trivialization $\pi_{{\rm Gr}}^{-1}(U)\cong U\times GL_d$. Let $\pi_{GL} : \pi_{{\rm Gr}}^{-1}(U) \to GL_d$ be the coordinate projection of this trivialization, which is $GL_d$-equivariant. We may assume that the image of $M_D\gamma_T: D \setminus \{0\}\to \mb{A}^{d\times n}$ lies in $\pi_{{\rm Gr}}^{-1}(U)$ by replacing $D$ with a suitable open subset. This yields a morphism
\begin{align*}
\pi_{GL} \circ (M_D \gamma_T) : D \setminus \{0\}\to GL_d. \label{GLmap}
\end{align*}
Let $\gamma_{GL}$ be the composition of $\pi_{GL} \circ (M_D \gamma_T)$ with the inverse map $GL_d\to GL_d$. 

Let $M'_D : D \setminus \{0\} \to \mb{A}^{d \times n}$ be the pointwise product $\gamma_{GL} M_D \gamma_T$. Then the image of $M'_D$ is contained in $\pi_{{\rm Gr}}^{-1}(U)$. We will show that $M'_D$ extends uniquely to a map $D \to \pi_{{\rm Gr}}^{-1}(U)$, which we also call $M'_D$ by abuse of notation.
\begin{center}
\begin{tikzcd}
D \setminus \{0\}
\arrow[drr, bend left]
\arrow[ddr, bend right]
\arrow[dr, "M'_D"] \\
 & \pi_{{\rm Gr}}^{-1}(U) \arrow[r, "\pi_{GL}"] \arrow[d, "\pi_{\rm Gr}"]  & GL_d\\
 & U \\ 
\end{tikzcd}
\end{center}
To show this, it suffices to show that the compositions $\pi_{{\rm Gr}} \circ M'_D : D \setminus \{0\} \to U$ and $\pi_{GL} \circ M'_D : D \setminus \{0\} \to GL_d$ extend uniquely to $D$.
Observe that $\pi_{{\rm Gr}} \circ M'_D = \pi_{{\rm Gr}}\circ (M_D\gamma_T)$, which we have already extended uniquely to $U$. Also, the map $\pi_{GL} \circ M'_D : D \setminus \{0\} \to GL_d$ is given by 
\begin{align*}
\pi_{GL} \circ (\gamma_{GL} M_D \gamma_T) 
= \gamma_{GL} (\pi_{GL} \circ (M_D \gamma_T) ) 
= (\pi_{GL} \circ (M_D \gamma_T) )^{-1} (\pi_{GL} \circ (M_D \gamma_T) ) 
= I,
\end{align*}
 the constant map sending each point of the curve $D \setminus \{0\}$ to the identity. This manifestly extends uniquely to $D$, so there is an extension $M'_D: D \to \pi_{{\rm Gr}}^{-1}(U)$.

Define \[M = M'_D(0) = \lim_{p \to 0} \gamma_{GL}(p)M_D(p)\gamma_{T}(p),\] which is an orb-limit of $M_D$. Then $M \in \pi_{{\rm Gr}}^{-1}(U)$, so it has full rank. For any vertex $v$ of $P$ we have
\[
\det((M'_D)^v) = \det \gamma_{GL} \cdot \det M_D \cdot t^{- \sum_{i=1}^n a_{i,F} v_i}
\]
so \[\operatorname{ord}_t \det((M'_D)^v) = \operatorname{ord}_t(\det \gamma_{GL}) + h_{M_D}(v) - \sum_{i=1}^n a_{i,F} v_i.\] Observe that since $M'_D(0)$ exists and has full rank, the minimum value of $\operatorname{ord}_t \det((M'_D)^v)$ over all vertices $v \in \Vt(P)$ is $0$. Hence $\operatorname{ord}_t(\det \gamma_{GL}) = -c_F$. Now $\det(M^v) \neq 0$ if and only if $\operatorname{ord}_t\det((M'_D)^v) = 0$, and this happens exactly when $h_{M_D}(v) = L_F(v)$, i.e. $v \in F$. Hence, the rank polytope of $M$ is $F$, as desired.
\end{proof}

Hence, we have some information about the special fibers of degenerations of generalized matrix orbits. Now we want to understand the general fibers. The following technical lemma is necessary to compare images of general fibers of a family with general fibers of the image of a family.

\begin{Lem}\label{simage}
Let $S$ be a Noetherian integral $K$-scheme. Suppose that $X,Y$ are finite type schemes over $S$ such that $X\to S$ has integral general fiber. Let $\phi: X\to Y$ be a morphism of $S$-schemes that is generically finite onto its image. Then for $s\in S$ general, the scheme-theoretic image of $\phi|_{s}$ agrees with the fiber of the scheme-theoretic image of $\phi$ over $s$. 
\end{Lem}
\begin{proof}
Without loss of generality, we can restrict $S$ so the fibers of $X\to S$ are equidimensional \cite[Tag 05F7]{stacks-project}. The scheme-theoretic image of $\phi$ is $\overline{\phi(X)}$ with the induced reduced subscheme structure \cite[Tag 056B]{stacks-project}. It is clear that the scheme-theoretic image of $\phi|_{s}$ is contained in the fiber of $\overline{\phi(X)}$  over $s$. To show the other inclusion, we claim the geometric generic fiber of $\overline{\phi(X)}$ over $S$ is integral \cite[IV 12.1.1 (x)]{EGA}. 

Let $\Spec(\overline{L})\to S$ be the geometric generic point. Scheme-theoretic image of a quasicompact morphism commutes with flat base change \cite[Tag 081I]{stacks-project}, so the scheme-theoretic image of $\phi$ pulled back to $Y\times_S \Spec(\overline{L})$ agrees with the scheme-theoretic image of the induced map $X\times_S\Spec(\overline{L})\to Y\times_S \Spec(\overline{L})$, which is integral.  

Therefore, the general fiber of $\overline{\phi(X)}$ over $S$ is integral \cite[Tag 0551]{stacks-project}. Therefore, we can restrict $S$ so that every fiber of $\overline{\phi(X)}$ over $S$ is integral. Then, by dimension reasons, this shows that the scheme-theoretic image of $\phi|_{s}$ is equal to the fiber of $\overline{\phi(X)}$  over $s$.
\end{proof}

We are now ready to prove \Cref{degenerationfibers}.
\begin{proof}[Proof of \Cref{degenerationfibers}]
Part (a) follows directly from $\psi(U) \subset (\mb{P}^r)^n \times (D \setminus \{0\})$.

Next, we prove (b). Given a general $\wt A \in \mb{P}^{(r+1)\times d-1}$, we have $(\wt{AM_D(p)}, p) \in \wt{\mc{O}}$. It follows that the fiber of $\wt{\mc{O}}$ over $p \in D$ set-theoretically contains $\op{M_D(p)}$.
If the fiber of $\wt{\mc{O}}$ over $p$ is $((r+1)d-1)$-dimensional, then the map $\psi$ is generically finite onto its image. We claim that $\psi$ is generically injective. Indeed, assume that $\psi(\wt{A}, p') = \psi(\wt{B}, p')$ is a general element of $\im(\psi)$ for $(r+1) \times d$ matrices $A, B$ and $p' \in D$. Then this is further equal to $\psi(\wt{A+ \lambda B},p')$ for general $\lambda$, contradicting generic finiteness if $\wt{A} \neq \wt{B}$.
Hence, by \cref{simage}, the fiber of $\wt{\mc{O}}$ over $p$ is the closure of the image of $\psi|_{(\mb{P}^{(r+1)\times d-1} \times \{p\}) \cap U}$, which is $\op{M_D(p)}$.

Finally, we prove (c). We claim that the fiber $\wt{\mc{O}}_0$ contains each $\op{M_i}$. By replacing $D$ with a suitable finite cover by a nonsingular curve and $M_D$ with an extension to $D$ of $\gamma_{GL} M_D \gamma_T$ for suitable $\gamma_{GL} : D \setminus \{0\} \to GL_{r+1}$ and $\gamma_T : D \setminus \{0\} \to T$, we may assume that $M = M_D(0)$. Given a general $\tilde A \in \mb{P}^{(r+1)\times d-1}$, we have $(\wt{AM_D(p)}, p) \in \wt{\mc{O}}$ for general $p \in D$. Taking the limit as $p \to 0$ yields $(\wt{AM_D(0)}, 0) \in \wt{\mc{O}}$. Since $\tilde A$ is general, it follows that the fiber of $\wt{\mc{O}}$ over $0 \in D$ set-theoretically contains $\op{M_i}$ as desired.

Now, we claim that the nonzero $\op{M_i}$ are distinct. Assume to the contrary that $\op{M_i} = \op{M_j} \neq 0$ for some $1 \leq i< j \leq k$. Then the non-equivariant classes satisfy $\opc{M_i} = \opc{M_j} \neq 0$. By \cref{NE}, there exists $(e_1,\ldots, e_n)\in \mb{Z}^n$ such that for any $0 < \epsilon_1, \ldots, \epsilon_n$ with $\sum_{i=1}^n \epsilon_i = 1$ we have $(e_1 + \epsilon_1, \ldots, e_n + \epsilon_n)$ is contained in $(r+1) (P_{M_i} \cap P_{M_j})$. But $(r+1) (P_{M_i} \cap P_{M_j})$ has codimension at least $1$ in the hyperplane $\{\sum x_i= (r+1)d\}$, so we cannot have the containment for all choices of $\epsilon_i$.

Since $\wt{\mc{O}}_0$ contains each $\op{M_i}$, the $((r+1)d-1)$-dimensional algebraic cycle \[Z = \wt{\mc{O}}_0 - \bigcup_{i=1}^n \op{M_i}\] is effective. The nonzero $\op{M_i}$ are distinct and irreducible of dimension $(r+1)d - 1$ and $\wt{\mc{O}}$ is locally flat at $0 \in D$, so we have the equality of non-equivariant Chow classes
\[
[Z] = [\wt{\mc{O}}_0] - \sum_{i=1}^n \opc{M_i} 
= [\wt{\mc{O}}_p] - \sum_{i=1}^n \opc{M_i} 
= \opc{M_D(p)} - \sum_{i=1}^n \opc{M_i} 
= 0
\]
for $p$ general, where the third equality follows from (b) and the fourth equality follows from \cref{realvalnoneq}. Hence $Z = 0$, as desired.
\end{proof}
\section{Additivity of $\opc{M}$}
\label{AMOC}
In this section, we prove part of \Cref{realvalintro}, by showing that the association $M \mapsto \opc{M}$ is additive (see \Cref{matroidpolytopedef} and \Cref{additivedef}) using the algorithm implicit in \cite[Appendix A]{Derksen} and results from \Cref{degenerations} and \Cref{polyapp}. This will be used to reduce the computation of $\opc{M}$ to simpler generalized matrix orbit classes in later sections.

\begin{Thm} \label{realval}
For each $r$, the association $M \mapsto \opc{M}$ of $d \times n$ matrices to equivariant Chow classes is additive.
\end{Thm}

Additivity of equivariant Chow classes of torus orbits in the Grassmannian is known using \cite[Proposition 1.9]{KMS} to specialize the result in equivariant K-theory \cite[Example 3.5]{FS12} to equivariant Chow classes \cite[Section 2.2]{Fink}. However, this is insufficient for our purposes.

Even the fact that $\opc{M}$ depends only on the matroid of $M$ is far from obvious. The realization spaces of matroids should be expected to be poorly behaved in general by Mnev universality \cite[Theorem 1.3]{KM98} \cite{LV13}, and there are explicit examples of disconnected realization spaces of line arrangements \cite{Line1, Line5, Line3, Line4}.

\subsection{Outline of the proof of \Cref{realval}}
\label{outline}
Let $S(d,n)$ be the abelian group generated by indicator functions of all (possibly unbounded) convex polytopal regions in $\{ \sum x_i=d\}$ modulo functions whose support has positive codimension. We will be working exclusively inside $S(d,n)$ from now on, treating in particular $\langle M \rangle \in S(d,n)$ (recall \Cref{additivedef}). For any (possibly unbounded) convex polytopal region $P$, we will also write $\langle P \rangle$ for the class of $1_P$ inside $S(d,n)$.

In \cite[Appendix~A]{Derksen}, Derksen and Fink show using an explicit algorithm that if $\sum_{i=1}^k a_i \langle M_i \rangle=0$ then the expression $\sum_{i=1}^k a_i \langle P_{M_i} \rangle$ can be reduced to $0$ by using replacements $\langle P'\rangle\mapsto 0$ when $P'$ is of positive codimension inside $\{\sum x_i=d\}$, and replacements $\langle P \rangle\mapsto\langle P+\mb{R}_{\ge 0}v \rangle - \sum_{i=1}^\ell \langle P_i \rangle$ with $v=e_i-e_j$ with $i>j$, where $P+\mb{R}_{\ge 0}v = P \cup \bigcup_{i=1}^\ell P_i$  are certain subdivisions which we will call ``shadow facet subdivisions'' (see \Cref{finksub}). We will show that each polytope appearing in the algorithm is of the form $P_M+C$ where $(M,C)$ is an ``admissible pair'', which we define below.

As in \Cref{polyapp}, we define a \emph{matroidal cone} to be a cone $C$ generated by rays of the form $\mb{R}_{\ge 0}(e_i-e_j)$.

\begin{Def}
Let $M$ be a $d \times n$ matrix and let $C$ be a matroidal cone. We define the \emph{associated polytope} of $(M,C)$ to be the polytope $P_M+C \subset \{\sum x_i=d\}$. We say that $(M, C)$ is an \emph{admissible pair} if $P_{M} + C$ is full-dimensional in $\{\sum x_i = d\}$ and $(P_M + C) \cap \Delta = P_M$.
\end{Def}

The criterion for $(M,C)$ to be admissible implies that if $e_i-e_j \in C$, then the $i$th column of $M$ is morally speaking more general than the $j$th column of $M$.

We will augment Derksen and Fink's algorithm, replacing each polytope $P$ with an admissible pair $(M,C)$ with associated polytope $P$. This replacement will be done in such a way that if $P+\mb{R}_{\ge 0}v  =  P  \cup \bigcup_{i=1}^\ell  P_i $ is a subdivision used in the algorithm, then $\opc{M'}=\opc{M}+\sum_{i=1}^\ell \opc{M_i}$, where $(M',C')$ is the admissible pair associated to $P+\mb{R}_{\ge 0}v$ and $(M_i,C_i)$ is the admissible pair associated to $P_i$. We will accomplish this using explicit degenerations over $\mb{A}^1_t$ (see \Cref{degenerationfibers}) which are linear in $t$ and add $t$ times general multiples of certain columns of $M$ to other columns. 

As it turns out, applying the rewriting procedure of Derksen and Fink without combining or cancelling any codimension zero terms allows us to rewrite any $\langle M \rangle$ as a linear combination $\sum \epsilon_i \langle P_i \rangle$, where each $P_i$ is a ``strictly positive matroidal cone'' with apex in $\Delta$ and $\epsilon_i = \pm 1$.

\begin{Def}
\label{positivematroidalconedef}
We define a \emph{positive matroidal cone} $C$ to be a matroidal cone $C$ whose generating vectors $e_i-e_j$ have $i>j$. We define a \emph{strictly positive matroidal cone} $C$ to be an $(n-1)$-dimensional positive matroidal cone generated by $n-1$ vectors $e_j-e_{\alpha_j}$ with $\alpha_j<j$ and $j \in \{2,\ldots,n\}$.
\end{Def}

The key fact about translates of strictly positive matroidal cones is that there are no relations between them in $S(d,n)$, so if we rewrite the expression $\sum a_i\langle M_i \rangle$ as $\sum \epsilon_i \langle P_i \rangle$ where each $P_i$ is a strictly positive matroidal cone with apex in $\Delta$, then the sum of the coefficients in front of a given $\langle P \rangle$ is zero. We then deduce the theorem from the fact for any matroidal cone $P$ with apex in $\Delta$, the space of all matrices $M$ such that $P_M = P \cap \Delta$ is connected, so $\opc{M} = \opc{M'}$ whenever $P_M = P_{M'} = P \cap \Delta$.

\subsection{Invariance of $\opc{M}$ when $P_M=(v+C)\cap \Delta$}
In this subsection, we prove that if $v \in \Vt(\Delta)$ and $C$ is a matroidal cone, then all matrices $M$ with $P_M=(v+C)\cap \Delta$ have the same $\opc{M}$. We also prove \Cref{addcone}, which serves as one of the key combinatorial interpretations of all future degenerations.

\begin{Def}
\label{conelike}
We define a polytope $P \subset \Delta$ to be a \emph{cone-like matroid polytope} if there is a matroidal cone $C$ and vertex $v \in \Delta$ for which $P=(v+C)\cap \Delta$. 
\end{Def}

\begin{Lem} \label{matroidconelemma}
Let $M, M'$ be $d \times n$ matrices. Assume $\op{M}\neq 0$. Suppose that $P_{M} = P_{M'} = (v + C) \cap \Delta$ is a cone-like matroid polytope. Then 
\begin{enumerate}[(a)]
	\item There exists an open subset $U \subset \mb{A}^1_t$ containing $0$ and $1$, a path $M_U : U \to \mb{A}^{d \times n}$ satisfying $M_U(0) = M$ and $M_U(1) = M'$, and a flat degeneration $\wt{\mc{O}} \subset (\mathbb{P}^{d-1})^n \times U$ such that $\wt{\mc{O}}_{p} = \op{M_D(p)}$ as $((r+1)d-1)$-dimensional cycles for all $p \in U$.
	\item $\opc{M}=\opc{M'}$.
\end{enumerate}
\end{Lem}

In order to prove \cref{matroidconelemma}, we will need to understand the structure of cone-like matroid polytopes. We will do this in a much greater generality than will be needed at this moment.

\begin{Prop}\label{addcone}
Let $M$ be a  $d \times n$ rank $d$ matrix and let $C$ be a matroidal cone. Define the $n \times n$ matrix $B$ by
\[
B_{i,k}= \begin{cases}
\lambda_{i,k} & \mbox{if $e_k-e_i \in C$ and $k \ne i$}\\
0 & \mbox{otherwise}
\end{cases}
\]
where the $\lambda_{i,k}$ are general elements of $K$. Let $N = M(1+tB)$.
\begin{enumerate}[(a)]
\item For any $z \in \Vt(\Delta)$, the order $\ord_t \det(N^z)$ is the minimum value of $|z \setminus w|$ over all pairs $(w, \phi)$ such that $w \in \Vt(P_M)$ and $\phi: z \setminus w \to w \setminus z$ is a bijection with $e_k-e_{\phi(k)} \in C$ for all $k \in z \setminus w$. In particular, it is $\infty$ if and only if no such pair exists.
\item For general $t$, the rank polytope of $N$ is $(P_A+C)\cap \Delta$.
\end{enumerate}
\end{Prop}
\begin{proof}
We first prove (a). For any $w,z \in \Vt(\Delta)$, let $(1+tB)^z_w$ denote the submatrix of $1 + tB$ formed by the columns in $z$ and the rows in $w$. By the Cauchy-Binet formula, $\det(N^z) = \det(M (1+tB)^z)$ is given by
\begin{align*}
 \sum_{w \in \Vt(\Delta)} \det(M^w) \det ((1+tB)^z_w) 
&= \sum_{w \in \Vt(P_A)} \det(M^w) \sum_{\substack{\phi : z \to w \\ \text{bijection}}} \pm \prod_{k \in z} (1+tB)_{\phi(k), k} \\
&= \sum_{\substack{w \in \Vt(P_M) \\ \phi : z \to w \\ \text{bijection}}} \pm  \det(M^w) \prod_{\substack{k \in z \\ \phi(k) \neq k}} (tB_{\phi(k),k}).
\end{align*}
Because each $\lambda_{i, k}$ is general, all nonzero terms of this sum are linearly independent. Hence $\ord_t \det(N^z)$ is the minimum value of $|\{k \in z \mid \phi(k) \neq k\}|$ over all pairs $(w, \phi)$, where $w\in \Vt(P_A)$ and $\phi:z \to w$ is a bijection such that $e_k-e_{\phi(k)} \in C$ for all $k \in z$. In particular, it is $\infty$ if and only if no such pair exists.

If $\ord_t \det(N^z) = \infty$, then (a) clearly holds. Otherwise, we claim that there exists a pair $(w,\phi)$ attaining the minimum value of $|\{k \in z \mid \phi(k) \neq k\}|$ such that $\phi$ fixes $z \cap w$. Indeed, suppose $(w, \phi')$ minimizes $|\{k \in z \mid \phi(k) \neq k\}|$. Define $\phi$ by $$\phi(i)=\begin{cases}i & \text{if $i \in z \cap w$, and}\\(\phi')^{k_i}(i) & \text{otherwise}\end{cases}$$ where $k_i$ is the smallest number such that $(\phi')^{k_i}(i) \in w \setminus z$. It is clear that $\phi$ is a bijection fixing $z\cap w$, and $e_i-e_{\phi(i)} \in C$ as $C$ is closed under addition. This shows part (a), and part (b) now follows from \cref{horribleequivalence}.
\end{proof}


\begin{proof}[Proof of \cref{matroidconelemma}]
As $\op{M}\ne 0$, by \Cref{positivecodim}, all generalized matrix orbits with matroid $(v+C)\cap \Delta$ are nonzero.

First, we prove (a). Let $C'$ be the nonnegative span of $\{e_j - e_i \in C \mid j \not \in v, i \in v\}$. By \Cref{horribleequivalence}, we have $(v+C) \cap \Delta = (v+C') \cap \Delta$. Hence, by replacing $C$ with $C'$, we may assume that if $e_j - e_i \in C$ for $i \neq j$, then $j \not \in v$ and $i \in v$. 
We can multiply $M$ and $M'$ on the left by elements of $GL_d$ and assume that $M^v=(M')^v=I$. We can also assume $v=e_1+\cdots+e_{d}$. 

Now, we consider the locus $Z_M\subset \mb{A}^{d\times n}$ of $d\times n$ matrices whose underlying matroid has polytope $P_M$ and whose restriction to the leftmost $d\times d$ submatrix is the identity. We claim that, for all $N\in Z_M$, if $e_j-e_i\not \in C$ then $N_{ij}=0$. Indeed, if $N_{ij}\neq 0$, then as $M^v=I$, replacing $i$ in the set $\{1,\ldots, d\}$ with $j$ preserves linear independence, so 
\begin{align*}
v+(e_j-e_i)=(e_1+\cdots+e_d)+(e_j-e_i)\in P_M \subset v+C,
\end{align*}
and hence $e_j-e_i\in C$. Therefore, the affine subspace $L_M\subset \mb{A}^{d\times n}$ consisting of matrices $N$ where $N_{ii}=1$ and $N_{ij}=0$ if $e_j-e_i\notin C$ contains $Z_M$. 

We have that $Z_M$ is a locally closed set in $L_M$ because it can be defined by the vanishing and non-vanishing of minors. The inclusion $Z_M\subset L_M$ is dense by \Cref{addcone}(b) applied to the $d \times n$ matrix $A$ with $A_{ii}=1$ and $A_{ij}=0$ if $i \ne j$, so $Z_M$ is open.

In \Cref{degenerationfibers}, choose $D = \mb{A}^1_t$ and $M_D(t) = tM + (1-t)M'$. By choosing $U \subset \mb{P}^{(r+1)\times d - 1} \times (M_D^{-1}(Z_M)\setminus \{0, 1\})$, we obtain a family $\wt{\mc{O}}$ that is flat over both $0 \in D$ and $1 \in D$. Part (a) of the lemma follows, and part (b) follows immediately from part (a).
\end{proof}

\subsection{Regular subdivisions arising from unbounded polytope decompositions} In this subsection, we will describe particular regular subdivisions $P = \bigcup_{i=1}^k P_i$ of possibly unbounded polytopes. Each such regular subdivision yields a decomposition $P \cap \Delta = \bigcup_{i=1}^k (P_i \cap \Delta)$ of bounded polytopes, which becomes a regular subdivision after removing lower-dimensional terms from the right-hand side. We will then use these regular subdivisions to construct the remaining degenerations $M_D$ that will be used in the proof of \Cref{realval}.

\begin{Lem}
\label{finksub}
Given a convex (possibly unbounded) full-dimensional polytope $P$ in $\mb{R}^n$ and a nonzero vector $v$, we have a regular subdivision of $P+\mb{R}_{\ge 0}v$ induced by the piecewise affine function $h: z \mapsto \inf_\lambda \{\lambda | z-\lambda v \in P\}$. The subdivision consists of $P$ and all polytopes $F+\mb{R}_{\ge 0}v$, where $F$ ranges over facets of $P$ such that $(F+\mb{R}_{\ge 0}v)\cap P=F$.
\end{Lem}

\begin{Def}
We call the facets of $P$ appearing in \Cref{finksub} \emph{shadow facets} as in \cite[Appendix A]{Derksen}. They are the upper facets with respect to the $v$-direction. We further call the subdivision the \emph{shadow facet subdivision}.
\end{Def}

\begin{proof}[Proof of \Cref{finksub}]
Define the linear map $\phi : \mb{R}^n \times \mb{R} \to \mb{R}^n \times \mb{R}$ by $\phi(x, \lambda) = (x + \lambda v, \lambda)$ and define $Q = \phi(P \times \mb{R}_{\geq 0})$. Note that $(y,\lambda) \in Q$ if and only if $\lambda \ge 0$ and $y-\lambda v \in P$. Then, for any $x \in P + \mb{R}_{\geq 0}v$, the minimum of $\{\lambda \mid (x, \lambda) \in Q\}$ is exactly $h(x)$. Hence, the lower convex hull of $Q$ is $P^h$, so the facets of the regular subdivision of $P + \mb{R}_{\geq 0}v$ induced by $h$ are the projections of the lower facets of $Q$ to $P + \mb{R}_{\geq 0}v$.

Because $\phi$ is a linear embedding, the facets of $Q$ are $\phi(P \times \{0\})$ and $\phi(F \times \mb{R}_{\geq 0})$ for the facets $F$ of $P$. The facet $\phi(P \times \{0\})$ is a lower facet and $\phi(F \times \mb{R}_{\geq 0})$ is a lower facet if and only if $F$ is a shadow facet. Projecting these facets to $P + \mb{R}_{\geq 0}v$ yields the shadow facet subdivision.
\end{proof}

We will now use \Cref{addcone} to construct the remaining degenerations needed in the proof of \Cref{realval}.

\begin{Thm} \label{explsub}
Let $(M,C)$ be an admissible pair and let $v = e_i - e_j$ with $i \neq j$. Then there exists a regular map $M_D : \mb{A}^1_t \to \mb{A}^{d \times n}$ with $M_D(0)=M$ such that for general $s \in \mb{A}^1_t$, there exist admissible pairs $(M_1, C_1), \ldots, (M_k, C_k)$, where each $M_i$ is an orb-limit of $M_D$ and
\begin{enumerate}[(a)]
\item The pair $(M_D(s), C + \mb{R}_{\geq 0} v)$ is admissible and $P_{M_D(s)} + (C + \mb{R}_{\geq 0} v) = (P_{M} + C) + \mb{R}_{\geq 0} v$.
\item We have a decomposition \[P_{M_D(s)} + (C + \mb{R}_{\geq 0} v) = (P_{M} + C) \cup \bigcup_{i=1}^k (P_{M_i} + C_i)\] which is the shadow facet subdivision of $(P_{M} + C) + \mb{R}_{\geq 0} v$.
\item We have a decomposition \[P_{M_D(s)} = P_{M} \cup \bigcup_{i=1}^k P_{M_i}\] which becomes the regular subdivision of $P_{M_D(s)}$ with lifting function $h_{M_D}$ associated to the degeneration $M_D$ after removing all terms of positive codimension in $P_{M_D(s)}$.
\item We have \[\opc{M_D(s)} = \opc{M} + \sum_{i=1}^k \opc{M_i}.\]
\end{enumerate}

\end{Thm}

\begin{proof}

Define the $n \times n$ matrix $B$ by
\[
B_{i,k}= \begin{cases}
\lambda_{i,k} & \mbox{if $e_k-e_i \in C+\mb{R}_{\ge 0}v$ and $k \ne i$}\\
0 & \mbox{otherwise}
\end{cases}
\]
where the $\lambda_{i,k}$ are general elements of $K$. Define $M_D$ by $M_D(t) = M(1+tB)$.

By \Cref{addcone}, we have $P_{M_D(s)}=(P_M+(C+\mb{R}_{\ge 0}v))\cap \Delta$. Thus $P_{M_D(s)}+(C+\mb{R}_{\ge 0}v)=(P_M+C)+\mb{R}_{\ge 0}v$, which proves (a).

Furthermore by \Cref{addcone} and \Cref{lambdalambdaprime} we have that $h_{M_D}$, when restricted to $\Delta$, coincides with the height function of the shadow facet subdivision of $(P_M+C)+\mb{R}_{\ge 0}v$ as given in \cref{finksub}. Hence, letting $F_1,\ldots,F_k$ be the shadow facets of $P_M+C$, we have that $(F_i+\mb{R}_{\ge 0}v)\cap \Delta$ are faces of the regular subdivision of $P_{M_D(t)}$ with lifting function $h_{M_D}$, and all facets of the regular subdivision apart from $P_M$ arise in this way. Each shadow facet $F_i$ is the Minkowski sum of a face of $P_M$ with a face of $C$, so $F_i \cap \Delta \ne \emptyset$. Therefore the faces $(F_i+\mb{R}_{\ge 0}v)\cap \Delta$ are nonempty. By \cref{lpath}, for each $i$, the degeneration $M_D$ orb-limits to a matrix $M_i$ such that $P_{M_i} = (F_i+\mb{R}_{\ge 0}v)\cap \Delta$. By definition of the $M_i$, part (c) follows. Define $C_i = C_{F_i} + \mb{R}_{\ge 0} v$, where $C_{F_i}$ is the cone generated by the unbounded edges of $F_i$.

We will now show that $P_{M_i}+C_i=F_i+\mb{R}_{\ge 0}v$, which by definition of $M_i$ implies that $(M_i,C_i)$ is admissible. Clearly $P_{M_i}+C_i\subset F_i+\mb{R}_{\ge 0}v$, so it remains to show that $F_i+\mb{R}_{\ge 0}v \subset P_{M_i}+C_i$. As $F_i$ is the Minkowski sum of a face of $P_M$ with a face of $C$, we have $F_i=(F_i \cap P_M)+C_{F_i} \subset P_{M_i}+C_{F_i}$ from which we conclude $F_i+\mb{R}_{\ge 0}v \subset P_{M_i}+C_i$. Along with (a), this in particular implies (b).

Finally, (d) follows from (c) and \cref{siorbit}.

\end{proof}

\subsection{The proof of \cref{realval}}
In this subsection, we prove \cref{realval}. The final piece of data we need is the algorithm implicit in \cite[Appendix A]{Derksen}, which will be woven into the fabric of the proof.

\begin{proof}
We proceed exactly as in \Cref{outline}. \Cref{positivecodim} allows us to safely discard any region of positive codimension in $\{\sum x_i=d\}$, and  starting with replacing $P_M$ with $(M,0)$, \Cref{explsub} allows us to augment each step of the algorithm as described. Finally, \cref{matroidconelemma} allows us to combine terms at the end when we reduce to strictly positive matroidal cones.

For the interested reader, we supply a sketch of the algorithm to reduce to translates of strictly positive matroidal cones and a proof that there are no relations in $S(d,n)$ between translates of strictly positive matroidal cones.

Define a \emph{positive matroidal polytope} to be a polytope all of whose edges are in directions $e_i - e_j$ and all of whose unbounded edges are in directions $e_i - e_j$ with $i > j$. For example, any rank polytope is a positive matroidal polytope. The algorithm proceeds in such a way that all polytopes described below are positive matroidal polytopes. Take constants $1 \ll \lambda_n \ll \lambda_{n-1} \ll \cdots \ll \lambda_1$. Given a polytope $P$ which appears in the algorithm, take $k \ge -1$ minimal such that $P^k=P\cap \bigcap_{i=n-k}^n \{x_i=\lambda_i\}$ is not a positive matroidal cone. If no such $k$ exists, then the direction vectors connecting the apexes of the cones $P^\ell$ for $\ell=-1,0,1,\ldots,n-1$ realize $P$ as a strictly positive matroidal cone. Otherwise, we claim that $P^k$ is a positive matroidal polytope. Indeed let $A$ be the set of indices $i$ for which $e_{n-k}-e_i$ is a direction vector contained inside the positive matroidal cone $P^{k-1}$ and let $C$ be the positive matroidal subcone of $P^{k-1}$ generated by those $e_i-e_j\in P^{k-1}$ with $i,j < n-k$. Then for $v$ the apex of $P^{k-1}$, we have $P^k=v+(\lambda_{n-k}-v_{n-k})\conv_{i \in A}(e_{n-k}-e_i)+C$. Hence since the edge directions of a Minkowski sum of polytopes are a subset of the edge directions of those polytopes (see \Cref{polyapp}), the claim follows.

Since $P^k$ is a positive matroidal polytope that is not a positive matroidal cone, it is not a cone. Choose a bounded edge of $P^k$, which by the above discussion is parallel to $e_i-e_j$ for some $i>j$. In the shadow facet subdivision of $P^k + \mb{R}_{\ge 0}v$ with $v=e_i-e_j$, both the polytope $P^k+\mb{R}_{\ge 0}v$ and the $P^k_i+\mb{R}_{\ge 0}v$ for shadow facets $P^k_i$ have strictly fewer vertices than $P^k$, as their vertices are all subsets of $\Vt(P)$ containing at most one of the two vertices of the chosen edge. Furthermore, this shadow facet decomposition is induced by the shadow facet decomposition of $P+ \mb{R}_{\ge 0}v$ after intersecting with $\bigcap_{i=n-k}^n\{x_i=\lambda_i\}$, and the polytopes $P+\mb{R}_{\ge 0}v$ and the $P_i$ remain cones up until this intersection. Hence as the number of vertices of the polytopes intersected with $\bigcap_{i=n-k}^n\{x_i=\lambda_i\}$ strictly decreases in each application, we will eventually reduce down to polytopes whose intersections with $\bigcap_{i=n-k}^n\{x_i=\lambda_i\}$ are positive matroidal cones. Since we can now decrement $k$, we eventually reduce all polytopes to strictly positive matroidal cones.

Now, we argue why there are no relations between translates of strictly positive matroidal cones. Indeed, suppose we have a combination $\sum a_i \langle P_i \rangle=0$. Then as the cones are positive matroidal, if one of these cones intersects a sufficiently small neighborhood of the lexicographically first vertex $v$, then it has this vertex as an apex. By considering this small neighborhood of $v$, we deduce that if we rectrict our sum to just those cones with apex $v$ we still get zero. By induction on the number of cones, we may assume that all cones have apex $v$. Take $\lambda \gg 0$, and consider the intersections with $x_n=\lambda$. These intersections remain strictly positive matroidal cones, and $\sum a_i\langle P_i \cap \{x_n=\lambda\}\rangle=0$ in $S(d-\lambda,n-1)\times \{\lambda\}$. Hence, we are done by induction on dimension.
\end{proof}

\section{Parallel and Series Connection, Existence of Rational Functions}
\label{ParAndSer}

In this section, we show that the generalized matrix orbit classes of the series and parallel connection of two matrices (\Cref{serpardef}) can be computed from the generalized matrix orbit classes of the matrices themselves (\Cref{serpar}). Combining this with the additivity result of the previous section, we will prove that the generalized matrix orbit class of a matrix takes a very particular form (\Cref{formulaexists}). In \Cref{formulasection}, we will combine \Cref{formulaexists} with the non-equivariant results of \Cref{noneqsection} to compute $\opc{M}$ for any matrix $M$.

%
%

The next theorem, \Cref{serparintro} from the introduction, relates the classes
\begin{align*}
\opc[r][n_1]{M_1} &\in A^\bullet_{GL_{r+1}} ((\mb{P}^r)^{n_1}) = A^\bullet_{GL_{r+1}} ((\mb{P}^r)^{n_1-1}) \otimes A^\bullet_{GL_{r+1}} (\mb{P}^r) \\
\opc[r][n_2]{M_2} &\in A^\bullet_{GL_{r+1}} ((\mb{P}^r)^{n_2}) = A^\bullet_{GL_{r+1}} (\mb{P}^r) \otimes A^\bullet_{GL_{r+1}} ((\mb{P}^r)^{n_2-1}) \\
\opc[r][n_1 + n_2 - 1]{P(M_1, M_2)} &\in A^\bullet_{GL_{r+1}} ((\mb{P}^r)^{n_1 + n_2 - 1}) = A^\bullet_{GL_{r+1}} ((\mb{P}^r)^{n_1-1}) \otimes A^\bullet_{GL_{r+1}} (\mb{P}^r) \otimes A^\bullet_{GL_{r+1}} ((\mb{P}^r)^{n_2-1}) \\
\opc[r][n_1 + n_2 - 1]{S(M_1, M_2)} &\in A^\bullet_{GL_{r+1}} ((\mb{P}^r)^{n_1 + n_2 - 1}) = A^\bullet_{GL_{r+1}} ((\mb{P}^r)^{n_1-1}) \otimes A^\bullet_{GL_{r+1}} (\mb{P}^r) \otimes A^\bullet_{GL_{r+1}} ((\mb{P}^r)^{n_2-1})
\end{align*}

using the quantum product $\star$ from \Cref{QCRD}.

\begin{Thm}\label{serpar} For $i \in \{1,2\}$ let $M_i$ be a $d_i \times n_i$ matrix. We have the equality of equivariant Chow classes $$\opc[r][n_1]{M_1} \star \opc[r][n_2]{M_2}=\opc[r][n_1+n_2-1]{P(M_1,M_2)}+\hbar\opc[r][n_1+n_2-1]{S(M_1,M_2)}$$ where the quantum product $\star$ is taken along the last factor of $A^\bullet_{GL_{r+1}} ((\mb{P}^r)^{n_1}) = A^\bullet_{GL_{r+1}} ((\mb{P}^r)^{n_1-1}) \otimes A^\bullet_{GL_{r+1}} (\mb{P}^r)$ and the first factor of $A^\bullet_{GL_{r+1}} ((\mb{P}^r)^{n_2}) = A^\bullet_{GL_{r+1}} (\mb{P}^r) \otimes A^\bullet_{GL_{r+1}} ((\mb{P}^r)^{n_2-1})$.
\end{Thm}

\begin{Thm}\label{formulaexists}
For every matrix $M$, there is a rational function $Q_{M}(x_1,\ldots,x_n,y_1,\ldots,y_n)$ independent of $r$ such that $Q_{M}(z_1,\ldots,z_n,F(z_1),\ldots,F(z_n))$ is a polynomial of degree at most $r$ in the indeterminates $z_i$ that yields $\opc{M}$ after substituting $z_i = H_i$.
\end{Thm}

\subsection{Parallel Connection and Series Connection}
The goal of this subsection is to prove \cref{serpar}. We start by proving the parallel connection formula.

\begin{Lem}\label{parlem}
Let $p_1 : (\mb{P}^r)^{n_1 + n_2 - 1} \to (\mb{P}^r)^{n_1}$ be the projection to the first $n_1$ factors and let $p_2 : (\mb{P}^r)^{n_1 + n_2 - 1} \to (\mb{P}^r)^{n_2}$ be the projection to the last $n_2$ factors. Then we have $$p_1^*\opc[r][n_1]{M_1} \cap p_2^*\opc[r][n_2]{M_2}=\opc[r][n_1+n_2-1]{P(M_1,M_2)}.$$ (Here we write the intersection product as $\cap$ to distinguish it from the quantum product $\star$.)
\end{Lem}
\begin{proof}

Let $W_1 = K^{n_1}$ and $W_2 = K^{n_2}$. Let $x_{1, 1}, \ldots, x_{1, n_1} \in W_1$ be the columns of $M_1$ and let $x_{2, 1}, \ldots, x_{2, n_2} \in W_2$ be the columns of $M_2$. Define $W = (W_1 \oplus W_2) / \langle x_{1, n_1} - x_{2, 1}\rangle = K^{n_1 + n_2-1}$. The columns of $P(M_1, M_2)$ are the vectors $x_{1, 1}, \ldots, x_{1, n_1} = x_{2, 1}, \ldots, x_{2, n_2}$ considered as elements of $W$.

Let $V = K^{r+1}$. For $i \in \{1,2\}$, we can think of the orbit map of $M_i$ as the map $\mu_{M_i} : \mb{P}(V \otimes W_i^\vee) \dashrightarrow \mb{P}(V)^n$ defined by the formula $\mu_{M_i}(\wt A) = (\wt{Ax_{i,1}}, \ldots, \wt{Ax_{i,n_i}})$ whenever $A \in \Hom(W, V) = V \otimes W_i^\vee$ is a matrix such that $A x_{i,j} \neq 0$ for $1 \leq j \leq n_i$.

For $i \in \{1, 2\}$, let $\wt \mu_{M_i} : B_i \to \mb{P}(V)^{n_i}$ be a regular map resolving the orbit map $\mu_{M_i}$.
\begin{center}
\begin{tikzcd}
B_1 \arrow[d, "\pi_{M_1}"] \arrow[dr, "\wt{\mu}_{M_1}"]& & B_2 \arrow[d, "\pi_{M_2}"] \arrow[dr, "\wt{\mu}_{M_2}"]&\\
\mb{P}(V\otimes W_1^{\vee}) \arrow[dashed,r,"\mu_{M_1}"]& \mb{P}(V)^{n_1} & \mb{P}(V\otimes W_2^{\vee}) \arrow[dashed,r,"\mu_{M_2}"]& \mb{P}(V)^{n_2}
\end{tikzcd}
\end{center}

Let $\pi_1: \mb{P}(V)^{n_1}\to \mb{P}(V)$ be the projection to the last factor and let $\pi_2: \mb{P}(V)^{n_2}\to \mb{P}(V)$ be the projection to the first factor. Define $B = B_1 \times_{\mb{P}(V)} B_2$, where the fiber product is taken over the maps $\phi_i : \pi_i \circ \wt{\mu}_{M_i} : B_i \to \mb{P}(V)$. Then the fiber product of the morphisms $\wt{\mu}_{M_1}$ and $\wt{\mu}_{M_2}$ is a morphism $\wt \mu: B \to \mb{P}(V)^{n_1} \times_{\mb{P}(V)} \mb{P}(V)^{n_2} = \mb{P}(V)^{n_1 + n_2 - 1}$. We claim that $\wt{\mu}$ resolves the orbit map $\mu_{P(M_1, M_2)}$ up to birationality on the domain; that is, there exists a birational map $\psi$ that makes the below diagram commute.
\begin{equation}\label{parallelresolution}
\begin{tikzcd}[column sep=1.5in]
B \arrow[dashed, d, "\psi"] \arrow[dr, "\wt\mu"]&\\
\mb{P}(V\otimes W^{\vee}) \arrow[dashed,r, "\mu_{P(M_1, M_2)}"]& \mb{P}(V)^{n_1+n_2-1}
\end{tikzcd}
\end{equation}
For $i \in \{1, 2\}$, let \[U_i = \{\wt{A} \mid A \in V \otimes W_i^\vee, \text{$A x_{i,j}\neq 0$ for $1 \leq j \leq n_i$} \} \subset \mb{P}(V\otimes W_i^{\vee})\] be the locus of definition of $\mu_{M_i}$ and let \[U = \{\wt{A} \mid A \in V \otimes W^\vee, \text{$A x_{i,j}\neq 0$ for $i \in \{1, 2\}$ and $1 \leq j \leq n_i$} \} \subset \mb{P}(V\otimes W^{\vee})\] be the locus of definition of $\mu_{P(M_1, M_2)}$. To prove the existence of $\psi$, we will define $\psi^{-1}$ on $U$ and show that it maps $U$ birationally to $B$.

Observe that the inclusion $W_i \to W$ induces a projection $\mb{P}(V\otimes W^\vee) \dashrightarrow \mb{P}(V \otimes W_i^\vee)$ that restricts to a regular map $q_i : U \to U_i$. These fit into a fiber diagram
\begin{center}
\begin{tikzcd}
U \arrow[d, "q_2"] \arrow[r, "q_1"]& U_1\arrow[d, "\pi_1 \circ \mu_{M_1}"] \\
U_2\arrow[r, "\pi_2 \circ \mu_{M_2}"] &\mb{P}(V)
\end{tikzcd}
\end{center}
where the maps $\pi_i \circ \mu_{M_i} : U_i \to \mb{P}(V)$ are given by $(\pi_i\circ \mu_{M_i})(\wt{A}) = \wt{A x_{1, n_1}} = \wt{A x_{2, 1}}$ for $A \in V \otimes W_i^\vee$. (Here we consider $x_{1, n_1} = x_{2, 1}$ as an element of the intersection $W_1 \cap W_2 \subset W$.)

Hence we can define $\psi^{-1}:U \to B$ as the fiber product of the maps $\pi_{M_i}^{-1}|_{U_i} : U_i \to B_i$. This maps $U$ birationally to $B$ because each $\pi_{M_i}^{-1}|_{U_i}$ maps $U_i$ birationally to $B_i$. The fact that the birational inverse $\psi$ of $\psi^{-1}$ makes the diagram \eqref{parallelresolution} commute follows from the definitions.

By further resolving $\psi$, we immediately see that $\wt \mu_*[B] = \opc[r][n_1+n_2-1]{P(M_1, M_2)}$.

To finish, consider the diagram below where each square is a fiber diagram. By \Cref{RMM}, we may assume that $B_1$ and $B_2$ are smooth and carry actions of $GL(V)$.
\begin{center}
\begin{tikzcd}
B \arrow[r,hook, "\wt{\Delta}"] \arrow[d, "\wt \mu"] \arrow[dd, swap, bend right=100, "\phi"] &B_1 \times B_2\arrow[d,"\wt{\mu}_{M_1} \times \wt{\mu}_{M_2}"] \arrow[dd, bend left=100, "\phi_1 \times \phi_2"]\\
\mb{P}(V)^{n_1+n_2-1} \arrow[d] \arrow[r,hook] & \mb{P}(V)^{n_1+n_2}\arrow[d]\\
\mb{P}(V) \arrow[r,hook, "\Delta"] & \mb{P}(V)^2
\end{tikzcd}
\end{center}
The map $\phi_1\times \phi_2$ is flat because its fibers are isomorphic, so in particular equidimensional, and the source and target are smooth \cite[Tag 00R4]{stacks-project}. The map $\wt{\Delta}$ is the pullback of the regular embedding $\Delta$ along $\phi_1 \times \phi_2$, so it is also a regular embedding.
By \cite[Theorem~6.2]{Fulton}, the pushforward of the fundamental class of $B$ along $\wt\mu$, i.e. $\opc[r][n_1+n_2-1]{P(M_1,M_2)}$, is equal to the pullback along $\Delta$ of the pushforward of the fundamental class of $B_1 \times B_2$ to $\mb{P}(V)^{n_1+n_2}$, which is the exterior product $\opc[r][n_1]{M_1}\times \opc[r][n_2]{M_2}$ restricted to $\mb{P}(V)^{n_1+n_2-1}$ and is equal to $p_1^*\opc[r][n_1]{M_1} \cap p_2^*\opc[r][n_2]{M_2}$.
\end{proof}

We need two lemmas to prove the series connection formula.
\begin{Lem}\label{integrateaway}
Given a $d \times n$ matrix $M$, if $M'$ is formed from $M$ by deleting the $i$th column, then $$\int_{(\mb{P}^r)^n \xrightarrow{\pi} (\mb{P}^r)^{n-1}} \opc{M}=\opc[r][n-1]{M'},$$ where $\pi:(\mb{P}^r)^n \to (\mb{P}^r)^{n-1}$ is the projection away from the $i$th factor.
\end{Lem}
\begin{proof}
Given a resolution of the indeterminacy locus $B$ for $\mu_M$, the composition $B \to \mb{P}(V)^n \to \mb{P}(V)^{n-1}$ resolves the indeterminacy locus for $\mu_{M'}$. The result follows.
\end{proof}

\begin{Lem}
\label{SCL}
Let $L$ be a general $2 \times 3$ matrix (corresponding to 3 points on a line in $\mb{P}^1$). Let $\pi_1,\pi_2,\pi_3: (\mb{P}^r)^3\to \mb{P}^r$ be the three projections. Then, for any $f,g \in A^\bullet_{GL_{r+1}}(\mb{P}^r)$, we have
\begin{align*}
\int_{(\mb{P}^r)^3\xrightarrow{\pi_3}\mb{P}^r}{(\pi_1^{*}f) \cap (\pi_2^{*}g) \cap \opc[r][3]{L}}=[\hbar](f \star g),
\end{align*}

\end{Lem}

\begin{proof}
We work with $T$-equivariant cohomology, where $T = (K^\times)^{r+1}$ is the standard torus (see \Cref{NotandConv}). First, we note that both sides of the stated equation are $A_{T}^{\bullet}(\pt)$-linear with respect to both $f$ and $g$. Therefore, it suffices to prove the theorem for $f=\prod_{i\in I}(H+t_i)$ and $g=\prod_{j\in J}(H+t_j)$ where $I,J\subsetneq \{0,\ldots,r\}$ and either $|I|+|J|\le r$ or $I \cup J=\{0,\ldots,r\}$. If $|I|+|J|\leq r$, then both sides of the equation are zero, the left for dimension reasons, and the right because $\overline{f}(z)\overline{g}(z)$ has degree at most $r$.

Suppose now that $I \cup J = \{0, \ldots, r\}$. Observe that $f$ and $g$ are the classes of the $T$-invariant linear subspaces $\Lambda_I = \{x_i=0\mid i\in I\}$ and $\Lambda_J=\{x_j=0\mid j\in J\}$ respectively. Because $\pi_3$ is proper, the integral \[\int_{(\mb{P}^r)^3\xrightarrow{\pi_3}\mb{P}^r}{(\pi_1^{*}f) \cap (\pi_2^{*}g) \cap \opc[r][3]{L}}\] is the class of the subvariety of $\mb{P}^r$ swept out by all lines connecting a point in $\Lambda_I$ to a point in $\Lambda_J$, which is precisely $\Lambda_{I \cap J}$. The class of $\Lambda_{I \cap J}$ is $\prod_{k\in I \cap J}(H+t_k)$. We also have $\overline{f}(z)\overline{g}(z)=F(z)\prod_{k \in I \cap J} (z+t_k)$, so $[\hbar](f \star g)=\prod_{k\in I \cap J}(H+t_k)$ as desired.
\end{proof}

\begin{proof}[Proof of \cref{serpar}]
By \cref{parlem}, it suffices to show $[\hbar^1]\opc{M_1} \star \opc{M_2}=\opc{S(M_1,M_2)}$. Let $L$ be a general $2\times 3$ matrix, and note that
up to left $GL$-action, the series connection $S(M_1, M_2)$ is the parallel connection $P(P(M_1, L), M_2)$ with the $n_1$th and $(n_1 + 2)$nd columns removed. By \cref{integrateaway}, we thus have \[\opc{S(M_1,M_2)}=\int_{(\mb{P}^r)^{n_1 + n_2 + 1} \xrightarrow{\pi} (\mb{P}^r)^{n_1 + n_2 - 1}}\opc{P(P(M_1,L),M_2)},\] where $\pi$ is the projection away from the $n_1$st and $(n_1 + 2)$nd factors. If $\alpha_i$ are classes pulled back from $(\mb{P}^r)^{n_i-1}$, then
\begin{align*}
&\int_{(\mb{P}^r)^{n_1+n_2-1}\to \mb{P}^r}\alpha_1 \cap \alpha_2 \cap \int_{(\mb{P}^r)^{n_1 + n_2 + 1} \xrightarrow{\pi} (\mb{P}^r)^{n_1 + n_2 - 1}}  \opc{P(P(M_1,L),M_2)} \\
=&\int_{(\mb{P}^r)^3 \to \mb{P}^r}\left(\int_{(\mb{P}^r)^{n_1}\to \mb{P}^r}\alpha_1\cap\opc{M_1}\right) \cap \left(\int_{(\mb{P}^r)^{n_2}\to \mb{P}^r}\alpha_2\cap\opc{M_2}\right) \cap \opc{L}\\
=&[\hbar]\left(\int_{(\mb{P}^r)^{n_1}\to \mb{P}^r}\alpha_1 \cap \opc{M_1}\right)\star\left(\int_{(\mb{P}^r)^{n_2}\to \mb{P}^r}\alpha_2\cap\opc{M_2}\right)\\
=&[\hbar]\int_{(\mb{P}^r)^{n_1+n_2-1}\to \mb{P}^r} (\alpha_1 \cap \opc{M_1}) \star (\alpha_2 \cap \opc{M_2})\\
=&\int_{(\mb{P}^r)^{n_1+n_2-1}\to \mb{P}^r} \alpha_1 \cap \alpha_2 \cap [\hbar](\opc{M_1}\star \opc{M_2})
\end{align*}
which shows that $[\hbar](\opc{M_1}\star\opc{M_2})=\opc{S(M,N)}$ as desired.
\end{proof}

\subsection{Proof of \Cref{formulaexists}}

\begin{Lem}\label{diagclass}
Let $M$ be a general $1\times 2$ matrix, corresponding to two identical points in $\mb{P}^0$. Then $\op[r][2]{M}$ is the diagonal cycle in $\mb{P}^r \times \mb{P}^r$, and its class is given by $\frac{F(H_1)-F(H_2)}{H_1-H_2}$.
\end{Lem}
\begin{proof}
Clearly $\op{M}$ is the image of the diagonal map $\mb{P}^r \to \mb{P}^r \times \mb{P}^r$. The equivariant class of the diagonal is the unique class which satisfies $$\int_{\mb{P}^r \times \mb{P}^r}f(H_1)g(H_2)\opc{M}= \int_{\mb{P}^r}f(H)g(H).$$ This is indeed satisfied by $\frac{F(H_1)-F(H_2)}{H_1-H_2}$ since $$\int_{\mb{P}^r \times \mb{P}^r}f(H_1)g(H_2)\frac{F(H_1)-F(H_2)}{H_1-H_2}=\int_{\mb{P}^r \times \mb{P}^r}f(H_1)g(H_1)\frac{F(H_1)-F(H_2)}{H_1-H_2}=\int_{\mb{P}^r}f(H_1)g(H_1),$$
where the first equality came from the fact that the difference of the two expressions is $f(H_1)\frac{g(H_1)-g(H_2)}{H_1-H_2}(F(H_1)-F(H_2))$, and the second from the fact that $\frac{F(H_1)-F(H_2)}{H_1-H_2}$ is of degree $r$ in $H_2$ with leading coefficient $1$.
\end{proof}

\begin{Lem}\label{serparform}
If $M \in K^{d \times (n+1)}$ is a series-parallel matrix with distinguished column $c$, then letting $z$ denote the hyperplane variable associated to $c$ and $H_1,\ldots,H_n$ the remaining hyperplane variables, there are rational functions $Q_{M,i}(x_1,\ldots, x_n,y_1,\ldots, y_n)$ not depending on $r$ or $F$ such that the polynomial expression for $\opc{M}$ of degree at most $r$ in each $H_i$ and $z$ can be written as (formally simplifying to a polynomial before evaluating at $H_i$ and $z$)
$$\sum Q_{M,i}(H_1,\ldots,H_n,F(H_1),\ldots,F(H_n))\frac{F(z)-F(H_i)}{z-H_i}.$$
\end{Lem}

\begin{proof}
By \Cref{diagclass}, the result is true for $M$ a general $1 \times 2$ matrix. Suppose now $M$ is an arbitrary series-parallel matrix.

There exist series-parallel matrices $M_1$ and $M_2$ such that (after permutation and scaling of the columns) either $M=P(M_1,M_2)$ or $M=S(M_1,M_2)$ with $c$ the special column for parallel or series connection. Indeed, parallel connection handles the case that the special point is repeated or is on the intersection of two curves, and series connection handles the remaining case. Hence, we suppose by induction that we know the result for $M_1$ and $M_2$. Let $H_{1,1},\ldots,H_{1,n_1}, z_1$ and $z_2, H_{2,1},\ldots,H_{2,n_1}$ be the hyperplane variables associated to the columns of $M_1$ and $M_2$ respectively. For succinctness, we will suppress the inputs of $Q$ in what follows. We have
\begin{align*}
\opc{M_1}\star \opc{M_2}=\sum \sum Q_{M_1,i}Q_{M_2,j}\frac{F(z)-F(H_{1,i})}{z-H_{1,i}}\frac{F(z)-F(H_{2,j})}{z-H_{2,j}},
\end{align*}
and the right hand side is of degree at most $r$ in each of the $H$ variables.
Using partial fraction decomposition, replace each $$\frac{F(z)-F(H_{1,i})}{z-H_{1,i}}\frac{F(z)-F(H_{2,j})}{z-H_{2,j}}$$ with
\begin{align*}
&\frac{F(z)-F(H_{1,i})}{z-H_{1,i}}\frac{1}{H_{1,i}-H_{2,j}}(F(z)-F(H_{2,j}))+\frac{F(z)-F(H_{2,j})}{z-H_{2,j}}\frac{1}{H_{2,j}-H_{1,i}}(F(z)-F(H_{1,i}))\\
=&\left(\frac{-F(H_{2,j})}{H_{1,i}-H_{2,j}}\frac{F(z)-F(H_{1,i})}{z-H_{1,i}}+\frac{-F(H_{1,i})}{H_{2,j}-H_{1,i}}\frac{F(z)-F(H_{2,j})}{z-H_{2,j}}\right)\\
&+F(z)\left(\frac{1}{H_{1,i}-H_{2,j}}\frac{F(z)-F(H_{1,i})}{z-H_{1,i}}+\frac{1}{H_{2,j}-H_{1,i}}\frac{F(z)-F(H_{2,j})}{z-H_{2,j}}\right).\\
\end{align*}
By \cref{serpar}, the sum of the terms with $F(z)$ in front will be the class of the series connection, and the sum of the remaining terms will be the class of the parallel connection. Furthermore, the obtained expressions for $\opc{P(M_1,M_2)}$ and $\opc{S(M_1,M_2)}$ are clearly reduced (i.e. of degree at most $r$) in the $z$-variable. Because the $\opc{M_1}\star \opc{M_2}$ was reduced in the $H$ variables, $\opc{P(M_1,M_2)}=\opc{M_1}\star \opc{M_2} \mod F(z)$ is clearly still reduced in the $H$ variables, as is $\opc{S(M_1,M_2)}=(\opc{M_1}\star \opc{M_2}-\opc{P(M_1,M_2)})/F(z)$.
\end{proof}

Before we prove \cref{formulaexists}, we will need to appeal to a result of \cite{Derksen} on Schubert matroids, so we recall the definition.

\begin{Thm}[\cite{Derksen}]\label{SchubertFinkThm}
For a matroid $\MM$, we have the following expression for $1_{P_\MM}$ as a combination of indicator functions of Schubert matroid rank polytopes.
$$\sum_{\ell=1}^n \sum_{\emptyset \subsetneq X_1 \subsetneq \ldots \subsetneq X_\ell=\{1,\ldots,n\}}(-1)^{n-\ell}1_{\operatorname{Sch}(\rk_{\MM}(X_1),\ldots,\rk_{\MM}(X_\ell),X_1,\ldots,X_\ell)}.$$
\end{Thm}

\begin{proof}[Proof of \cref{formulaexists}]

By \cref{serparform}, we know the result for series-parallel matrices. The indicator functions of rank polytopes can be written as a combination of Schubert rank polytopes by \cref{SchubertFinkThm}, so by \cref{realval} it suffices to show that if the rank polytope of $\operatorname{Sch}(r_1,\ldots,r_\ell,X_1,\ldots,X_\ell)$ has codimension $0$ in $\Delta$, then it has a subdivision into series-parallel rank polytopes. This follows from the following observation. If $|X_\ell\setminus X_{\ell-1}| =1$ then the rank polytope is not of full dimension. Otherwise, let $a, b \in X_\ell\setminus X_{\ell-1}$ be distinct. As $\sum_{X_\ell}x_i=r_\ell$ is constant, there is a subdivision of the rank polytope into the regions where $x_a + x_b \le 1$ and where $\sum_{X_\ell \setminus \{a, b\}}x_i \le r_\ell-1$. The first region is the rank polytope of the parallel connection along $a$ of $\operatorname{Sch}(r_1,\ldots,r_{\ell},X_1,\ldots,X_{\ell-1},X_\ell\setminus\{b\})$ with $\operatorname{Sch}(1,\{a,b\})$ (i.e. $a$ and $b$ merge), and the second region is the rank polytope of the series connection along $a$ of $\operatorname{Sch}(r_1,\ldots,r_{\ell-1},r_\ell-1,X_1,\ldots,X_{\ell-1},X_\ell\setminus \{b\})$ with $\operatorname{Sch}(1,\{a,b\})$ (i.e. all points other than $a,b$ get projected down to rank $r_l-1$), where by abuse of notation we call the newly created points in the series and parallel connection $a$ as well. It is known that series and parallel connection distribute over matroid rank polytope subdivision (from e.g. the explicit description of these operations on polytopes in \cite[Proof of Theorem 7.3]{FS12}), so repeatedly applying this we may reduce the polytope down to series-parallel rank polytopes.
\end{proof}
\section{Formulas for $\opc{M}$}\label{formulasection}
In this section, we will derive the formulas for the equivariant Chow classes of $\opc{M}$ using Brion's theorem on rational convex polytopes. In the special case $d=r+1$, by \Cref{liftingsection} the formula in \Cref{finkformulas} (\Cref{finkformulaintro} from the introduction) induces the localized classes at the torus-fixed points of the Grassmannian computed in \cite[Lemma 5.2]{Fink}.

For each permutation $\sigma=(\sigma(1),\ldots,\sigma(n))$ of $\{1,\ldots,n\}$, define the total ordering $\prec_\sigma$ on $\{1, \ldots, n\}$ by $\sigma(1)\prec_\sigma \cdots\prec_\sigma\sigma(n)$. This total ordering induces a (lexicographic) total ordering on the $d$-element subsets of $\{1, \ldots, n\}$.

\begin{Thm}\label{finkformulas}
Let $M$ be a $d \times n$ matrix with columns $x_1,\ldots, x_n \ne 0$. If $\rk(M) < d$, then $\op{M} = 0$. Otherwise, for each permutation $\sigma$ of $\{1,\ldots,n\}$, let $B(\sigma)$ be the lexicographically first $d$-element subset of $\{1,\ldots,n\}$ with respect to the ordering $\prec_\sigma$ such that $e_{B(\sigma)} \in \Vt(P_M)$. Then for indeterminates $z_1,\ldots,z_n$, the expression
\begin{equation}\label{finkformularationalfunction}
\sum_{\sigma \in S_n} \left(\prod_{i \in \{1,\ldots,n\}\setminus B(\sigma)} F(z_i)\right)\frac{1}{(z_{\sigma(2)}-z_{\sigma(1)})\ldots (z_{\sigma(n)}-z_{\sigma(n-1)})}
\end{equation} is a polynomial of degree at most $r$ in each $z_i$ with coefficients in $A^{\bullet}_{GL_{r+1}}(\pt)$, and the equivariant Chow class $\opc{M}$ is given by evaluating this polynomial at $H_1,\ldots,H_r$, or informally
\begin{equation}\label{finkformulaseqn}
\opc{M}=\sum_{\sigma \in S_n} \left(\prod_{i \in \{1,\ldots,n\}\setminus B(\sigma)} F(H_i)\right)\frac{1}{(H_{\sigma(2)}-H_{\sigma(1)})\cdots (H_{\sigma(n)}-H_{\sigma(n-1)})}.
\end{equation}
\end{Thm}
\begin{Rem}
The expression for $\opc{M}$ must be evaluated by first writing \eqref{finkformularationalfunction} as a polynomial \emph{before} plugging in the $H_i$. This is because the denominators appearing in \eqref{finkformulaseqn} are zero-divisors (e.g. $H_i-H_j$ is killed by $\frac{F(H_i)-F(H_j)}{H_i-H_j}$).
\end{Rem}

Before proving \Cref{finkformulas}, we introduce notation and recall a classical theorem of Gale.

\begin{Def}
For any permutation $\sigma \in S_n$, let $C_\sigma$ be the cone generated by all vectors of the form $e_{i}-e_{j}$ with $i \succ_\sigma j$.
\end{Def}

\begin{Lem}\label{galemaximality}
For any permutation $\sigma \in S_n$ and any $d \times n$ matrix $M$, define $B(\sigma)$ as in the statement of \Cref{finkformulas}. Then we have $P_M + C_\sigma = e_{B(\sigma)} + C_\sigma$.
\end{Lem}
\begin{proof}
This is equivalent to the statement that if $B'$ is a $d$-element subset of $\{1, \ldots, n\}$ such that $e_{B'} \in \Vt(P_M)$, then there is a bijection $f : B(\sigma) \to B'$ such that $f(i) \succeq_\sigma i$ for all $i \in B(\sigma)$. This is the Gale~Maximality~Principle; see \cite[Theorem 1.1]{Borovik} or \cite{Gale}.
\end{proof}

Let $Z$ be the abelian group generated by indicator functions of possibly unbounded rational convex polytopes in $\mb{R}^n$, and let $\phi : Z \to \mb{Q}(z_1, \ldots, z_n)$ be the Brion homomorphism, which sends the indicator function of a (possibly unbounded) convex polytope $P$ not containing a line to the integer point transform of $P$ \cite[Theorem 2.4]{BHS}. If $P$ is a convex polytope containing a line, then $\phi(P) = 0$ \cite[Lemma 2.5]{BHS}. We sometimes denote $\phi(1_P)$ by $\phi(P)$.

\begin{proof}[Proof of \cref{finkformulas}]
If $\rk(M) < d$, then $\op{M} = 0$, so we may assume $\rk(M) = d$. First, we will show for all $r$ that substituting $F(H) = H^{r+1}$ into the right-hand side of \eqref{finkformulaseqn} yields the expression for the non-equivariant class $\opc{M} \in A^\bullet((\mb{P}^r)^n)$ with degree at most $r$ in each $H_i$.

Let $\alpha = (r+\frac{1}{n})(1,\ldots,1)$ and let $P'_M = \alpha -(r+1)P_M$. By \Cref{NE}, the non-equivariant class $\opc{M}$ is the sum of $\prod_{i=1}^n H_i^{e_i}$ over all $(e_i)\in P'_M$, and this sum has degree at most $r$ in each $H_i$. Hence
$\opc{M} = \phi(P'_M) (H_1,\ldots, H_n)$ non-equivariantly, so we wish to show that 
\begin{equation}\label{noneqformula}
\phi(P'_M) = \sum_{\sigma \in S_n} \left(\prod_{i \in \{1,\ldots,n\}\setminus B(\sigma)} z_i^{r+1}\right)\frac{1}{(z_{\sigma(2)}-z_{\sigma(1)})\cdots (z_{\sigma(n)}-z_{\sigma(n-1)})}.
\end{equation}
We will do so using the following lemma.

\begin{Lem}\label{addallcones}
Let $P \subset \mb{R}^n$ be a rational convex polytope. Then \[\phi(P) = \sum_{\sigma \in S_n} \phi(P + C_\sigma).\]
\end{Lem}
\begin{proof}
Let $\mc{P}$ be the convex hull of the vectors \[w_\sigma = -(\sigma^{-1}(1), \ldots, \sigma^{-1}(n)) + ((n+1)/2, \ldots, (n+1)/2)\] for all $\sigma \in S_n$. (This is the permutohedron \cite[Example 2.6]{permutahedron} translated so that it contains $0$.) The tangent cone to $\mc{P}$ at $w_\sigma$ is $w_\sigma + C_\sigma$.

For a sufficiently small rational number $\epsilon > 0$, the polytopes $P$ and $P_\epsilon := P + \epsilon \mc{P}$ have the same lattice points, so $\phi(P) = \phi(P_\epsilon)$. Now, by the description of vertices of Minkowski sums in \Cref{polyapp} and Brion's~Theorem \cite[p.3]{BHS}, we have $\phi(P) = \phi(P_\epsilon)$ can be written as
\[
  \sum_{u \in \Vt(P_\epsilon)} \phi(T_{u, P_\epsilon}) 
= \sum_{\sigma \in S_n} \sum_{v \in \Vt(P)}  \phi(T_{v, P} + \epsilon w_\sigma + C_\sigma) 
= \sum_{\sigma \in S_n} \phi(P + \epsilon w_\sigma + C_\sigma).\]
In each equality we used the vanishing of $\phi$ on cones containing lines.

Since $- w_\sigma \in C_\sigma$, the polytopes $P + \epsilon w_\sigma + C_\sigma$ and $P + C_\sigma$ have the same lattice points for sufficiently small $\epsilon$. Hence this sum is equal to $\sum_{\sigma \in S_n} \phi(P + C_\sigma)$ as desired.
\end{proof}

Continuing the proof of \Cref{finkformulas}, applying \Cref{addallcones} with $P = P'_M$ we obtain
\begin{align*}
\phi(P'_M) = \sum_{\sigma \in S_n} \phi(P'_M + C_\sigma)
=\sum_{\sigma \in S_n} \phi(\alpha-(r+1)P_M + C_\sigma) 
= \sum_{\sigma \in S_n} \phi(\alpha-(r+1)(P_M + C_\sigma)) 
\end{align*}
where the last equality comes from replacing $\sigma$ with its reverse $\sigma^r = (\sigma(n), \ldots, \sigma(1))$. By \Cref{galemaximality}, the term of this sum is equal to

\begin{align*}
\phi(\alpha-(r+1)(e_{B(\sigma)} + C_\sigma)) &= \left(\prod_{i \in \{1,\ldots,n\}\setminus B(\sigma)} z_i^{r+1}\right) (z_1 \cdots z_n)^{-1} \phi((1/n,\ldots,1/n) -  C_\sigma)\\
&= \left(\prod_{i \in \{1,\ldots,n\}\setminus B(\sigma)} z_i^{r+1}\right)(z_1 \cdots z_n)^{-1} z_{\sigma(1)} \frac{1}{\left(1 - \frac{z_{\sigma(1)}}{z_{\sigma(2)}}\right) \cdots \left(1 - \frac{z_{\sigma(n-1)}}{z_{\sigma(n)}}\right)} \\
&= \left(\prod_{i \in \{1,\ldots,n\}\setminus B(\sigma)} z_i^{r+1}\right)\frac{1}{(z_{\sigma(2)}-z_{\sigma(1)})\cdots (z_{\sigma(n)}-z_{\sigma(n-1)})}
\end{align*}
where we used the fact that the polytopes $(1/n,\ldots,1/n) -  C_\sigma$ and $e_{\sigma(1)} - C_\sigma$ have the same lattice points, from which \eqref{noneqformula} follows.

Let $f_M$ be the rational function
\begin{align*}
f_M(x_1,\ldots,x_n,y_1,\ldots,y_n) &= \sum_{\sigma \in S_n} \left(\prod_{i \in \{1,\ldots,n\}\setminus B(\sigma)} y_i\right)\frac{1}{(x_{\sigma(2)}-x_{\sigma(1)})\ldots (x_{\sigma(n)}-x_{\sigma(n-1)})},
\end{align*}
and let $Q_{M}(x_1,\ldots,x_n,y_1,\ldots,y_n)$ be the rational function from \Cref{formulaexists}. Then, by the above and \cref{formulaexists}, we have for all $r$ that the expressions
\begin{align*}
\text{$f_M(z_1,\ldots,z_n,z_1^{r+1},\ldots,z_{n}^{r+1})$ and $Q_{M}(z_1,\ldots,z_n,z_1^{r+1},\ldots,z_n^{r+1})$}
\end{align*}
are both polynomials of degree at most $r$ in each $z_i$ that yield the non-equivariant class $\opc{M} \in A^\bullet((\mb{P}^r)^n)$ after substituting $z_i = H_i$. Therefore, they are equal.

Since the set \[\bigcup_{r=0}^\infty \{(z_1,\ldots,z_n,z_1^{r+1},\ldots,z_{n}^{r+1}) \mid z_1, \ldots, z_n \in K\}\] is Zariski dense in $\mb{A}^{2n}$ (any polynomial of total degree $k$ vanishing on the $k$th term of this union is zero), this implies that $f_M = Q_M$, and the theorem follows from \Cref{formulaexists}.
\end{proof}
\section{Kronecker duals of $\opc{M}$ and constructing $[M]_\hbar$}\label{poincaresection}
In this section, we describe how to the Kronecker duals $\opc{M}^\dagger : A^\bullet_{GL_{r+1}}(\mb{P}^r)^{\otimes n} \to A^\bullet_{GL_{r+1}}(\pt)$, which we recall (\Cref{dualdef}) are defined by \[\opc{M}^\dagger(f_1(H_1), \ldots, f_n(H_n)) = \int_{(\mb{P}^r)^n} \opc{M} f_1(H_1) \ldots f_n(H_n),\] are computed via the $[M]_{\hbar}$ operator from \Cref{Mhbardef}. This operation specializes to the $n$-ary $\star$-product when $M$ is an invertible $n \times n$ matrix, and satisfies many other nice properties. Recalling the definitions of $\ast$ and $\tau^{\le k}$ from \cref{asttau}, we defined \begin{align*}[M]_\hbar(f_1(H_1),\ldots,f_n(H_n))=\sum_{k=0}^{d-1}\left(\int_{(\mb{P}^r)^{n+1}\to \mb{P}^r}f_1(H_1)\ldots f_n(H_n)\opc{\tau^{\le k+1}(M\oplus \ast)}\right)\hbar^k\\\in A^\bullet_{GL_{r+1}}(\mb{P}^r)[\hbar]\cong QH^\bullet_{GL_{r+1}}(\mb{P}^r),\end{align*} where $(\mb{P}^r)^{n+1}\to \mb{P}^r$ is projection to the last factor. We will show in this section all of the properties of $[M]_{\hbar}$ claimed in \Cref{Mhbarintro}.


\begin{Def}
\label{Mhbardaggerdef}
We define the Kronecker dual of $[M]_{\hbar}$ to be
$$[M]_\hbar^\dagger = \sum_{k=0}^{d-1} \opc{\tau^{\le k+1}(M \oplus \ast)}\hbar^k \in A^\bullet_{GL_{r+1}}(\mb{P}^r)^{\otimes n}\otimes QH^\bullet_{GL_{r+1}}(\mb{P}^r),$$ where we apply the isomorphism $A^\bullet_{GL_{r+1}}(\mb{P}^r)[\hbar]\cong QH^\bullet_{GL_{r+1}}(\mb{P}^r)$ to the last factor of $\mb{P}^r$.
\end{Def}
\begin{Rem}\label{truncgen}
For $k \ge \rk(M)$, we have $\opc{\tau^{\le k+1}(M \oplus \ast)}=0$ by \Cref{positivecodim}. When $k < d$, we have $\tau^{\le k+1}(M \oplus \ast)=\tau^{\le k+1}\tau^{\le d}(M \oplus \ast)$, and we may take $\tau^{\le d}(M \oplus \ast)$ to be the $d \times (n+1)$ matrix formed by adding a general extra column to $M$.
\end{Rem}

Later we will see how $[M]_\hbar$ contains the information of $[\ol{\Gamma_M}]$ and $\mu_M^*$, but for now it will serve only as a combinatorial tool to aid in computing Kronecker dual classes.

\subsection{Properties of $[M]_\hbar$}
We start with a lemma collecting the easy properties of $[M]_\hbar$.
\begin{Lem}\label{mhbareasy}
We have the following properties of $[M]_\hbar$.
\begin{enumerate}[(a)]
\item $[\ast]_\hbar(f(H))=\bar{f}(z)$ for any $f(H) \in A^\bullet_{GL_{r+1}}(\mb{P}^r)$.
\item The reduction of $[M]_\hbar$ modulo $\hbar^k$ is $[\tau^{\leq k} M]_\hbar$.
\item Let $M'$ be the matrix formed by deleting the last column of $M$. Then $[M]_\hbar(\alpha \otimes 1)=[M']_{\hbar}(\alpha)$ for all $\alpha \in A^\bullet_{GL_{r+1}}(\mb{P}^r)^{\otimes(n-1)}$.
\item Identifying $QH^\bullet_{GL_{r+1}}(\mb{P}^r) \cong A^\bullet_{GL_{r+1}}(\mb{P}^r)[\hbar]$, the composite \[\left(\int_{\mb{P}^r}\right) \circ [M]_\hbar : A^\bullet_{GL_{r+1}}(\mb{P}^r)^{\otimes n} \to A^\bullet_{GL_{r+1}}(\pt)[\hbar]\] is equal to $\sum_{k=1}^{d-1} \opc{\tau^{\le k+1}M}^\dagger\hbar^k$.
\end{enumerate}
\end{Lem}
\begin{proof}
Everything follows from the corresponding properties of $\opc{M}$.
\end{proof}

The association from $d \times n$ matrices to operations $M \mapsto [M]_\hbar$ is not additive. For example, $\rk(M) < d$ does not imply $[M]_\hbar=0$. We will see that $M \mapsto [M]_\hbar$ satisfies the more refined version of additivity called valuativity (recall \Cref{matroidpolytopedef}). 

Now we prove the difficult important properties of $[M]_\hbar$ from \Cref{Mhbarintro}.
\begin{Thm}\label{starunion}
\begin{enumerate}[(a)]
\item If $M_1$ and $M_2$ are matrices with $n_1$ and $n_2$ columns respectively, then $$[M_1 \oplus M_2]_\hbar = [M_1]_\hbar \star [M_2]_\hbar,$$ where the $[M_1]_\hbar$ applies to the first $n_1$ tensor factors and the $[M_2]_\hbar$ applies to the last $n_2$ tensor factors.
\item The association $M \mapsto [M]_\hbar$ from $d \times n$ matrices to operators is valuative, and for fixed $n$ and possibly varying $d$, $[M]_\hbar$ depends only on the matroid of $M$.
\end{enumerate}
\end{Thm}

\begin{proof}[Proof of \cref{starunion}]
Rather than work with $[M]_\hbar$, we work with the Kronecker dual class \[[M]_\hbar^\dagger \in A^\bullet_{GL_{r+1}}(\mb{P}^r)^{\otimes n}\otimes QH^\bullet_{GL_{r+1}}(\mb{P}^r).\] By \cref{serpar} and \cref{formulaexists}, we know that $[M]_\hbar^\dagger \star [N]_\hbar^\dagger-[M \oplus N]_\hbar^\dagger$ can be written as $Q(H_1,\ldots,H_n, z,F(H_1),\ldots,F(H_n), F(z))$, where $Q$ is a rational function with coefficients in $A^\bullet_{GL_{r+1}}(\pt)$. Similarly for any given expression $\sum a_i [M_i]_\hbar$. Hence, by the Zariski density argument at the end of the proof of \Cref{finkformulas}, it suffices to show the statements in the non-equivariant setting; that is, after substituting $F(H)=H^{r+1}$.

We claim that the non-equivariant $[M]_\hbar^\dagger$ is given by \begin{equation} \label{noneqmhbar} [M]_\hbar^\dagger=\sum z^{\sum e_i}\prod_{i=1}^n H_i^{r-e_i}\end{equation} where the sum is over all tuples $(e_1, \ldots, e_n)\in \mb{Z}^n$ such that
$\sum_{i \in A} e_i < (r+1)\rk(A)$ for all nonempty $A\subset [n]$. (Note that in light of \cref{NEgraphclosure}, this claim implies the non-equivariant analogue of \Cref{PFAintro}, which we will prove equivariantly in \Cref{graphclosuresection}.)

Note that the non-equivariant $[M]_\hbar^\dagger$ is homogeneous of degree $nr$ by definition, so it suffices to find the $(e_i)$ which appear in monomials $z^{\sum e_i}\prod_{i=1}^n H_i^{r-e_i}$. By \Cref{NE}, a monomial $H_{n+1}^{r-e}\prod_{i=1}^n H_i^{r-e_i}$ appears in $\opc{\tau^{\le (k+1)}(M\oplus \ast)}$ if and only if $\sum_{i \in A} e_i < (r+1)\rk_M(A)$ for all nonempty $A \subset [n]$ and $\sum_{i=1}^n e_i + e=(r+1)(k+1)-1$. As $e$ can range between $0$ and $r$, this latter equality implies $\sum_{i=1}^n e_i \in [r(k+1),(r+1)(k+1)-1]$.

Summing over $k$, we obtain all nonnegative sequences $(e_i)$ with $\sum_{i \in A} e_i < (r+1)\rk(A)$ for all nonempty $A \subset [n]$ as claimed.

Fixing $0<\epsilon_1,\ldots,\epsilon_n$ with $\sum_{i=1}^n \epsilon_i \le 1$, the sum in \eqref{noneqmhbar} is equivalently over all $(e_1, \ldots, e_n)\in \mb{Z}_{\ge 0}^n$ such that $(e_1+\epsilon_1, \ldots, e_n + \epsilon_n) \in (r+1)I_M$, so (b) follows as $[M]_{\hbar}^\dagger$ is manifestly valuative in light of \Cref{rankind}.
Also, (a) follows from the fact that $I_{M \oplus N} = I_M \times I_N$ (taking $\epsilon_i$ sequences for $M$, $N$ which combined add up to at most $1$).
\end{proof}

\subsection{Kronecker duals for general matrices and mod $F^k$}
\label{PDGM}
We will now investigate the Kronecker dual of $\opc{M}$ for a general $d \times n$ matrix $M$.

A general $d \times n$ matrix $M$ is realized as $\tau^{\le d}(\ast \oplus \ldots \oplus \ast)$, so by \Cref{mhbareasy}(a,b) and \Cref{starunion}(a) we have $$[M]_\hbar(f_1(H_1),\ldots,f_n(H_n))=\bar{f_1}(z)\ldots\bar{f_n}(z)\text{ mod }F(z)^d$$ whenever $f_1, \ldots, f_n$ are polynomials with coefficients in $A^\bullet_{GL_{r+1}}(\pt)$. Writing $\bar{f_1}(z)\ldots\bar{f_n}(z)=a_0+a_1F(z)+\ldots$, we have $$\opc{M}^\dagger(f_1(H_1),\ldots,f_n(H_n))=\int_{\mb{P}^r} a_{d-1}=[z^r][F(z)^{d-1}](\bar{f_1}(z)\ldots\bar{f_n}(z)\text{ mod }F(z)^d)$$ by \Cref{mhbareasy}(d).

\subsection{Kronecker duals for Schubert matrices and iterated mod $F^k$}
We will now investigate the Kronecker dual of $\opc{M}$ where $M$ is a Schubert matrix. Let $M$ be a matrix realizing $\operatorname{Sch}(r_1, \ldots, r_\ell, \{1, \ldots, a_1\}, \ldots, \{1, \ldots, a_\ell\})$ where $0 \leq r_1 \leq \ldots \leq r_\ell = d$ and $0 \leq a_1 \leq \ldots \leq a_\ell = n$. Recall from \cref{SchubertDef} that we may realize $M$ as $$M = \tau^{\le r_k}(\tau^{\le r_{k-1}}(\ldots\tau^{\le r_2}(\tau^{\le r_1}(\ast \oplus \ldots \oplus \ast)\oplus \ast \oplus \ldots \oplus \ast)\ldots)\oplus \ast \oplus \ldots \oplus \ast).$$
By \cref{mhbareasy}(a,b) and \Cref{starunion}(a) we have
$$[M]_\hbar(f_1(H_1),\ldots,f_n(H_n))=\left(\left(\bar{f}_1\ldots\bar{f}_{a_1}\text{ mod }F^{r_1}\right)\bar{f}_{a_1+1}\ldots \bar{f}_{a_2}\text{ mod }F^{r_2}\right)\ldots\bar{f}_n\text{ mod }F^{r_\ell}$$ (evaluated at $z$). By \Cref{mhbareasy}(d), applying the operator $[z^r][F(z)^{d-1}]$ to this expression yields $\opc{M}^\dagger(f_1(H_1),\ldots,f_n(H_n))$. Both $[M]_\hbar$ and $\opc{M}^\dagger$ are asymmetric and we may recover the expressions for all Schubert matroids by permuting the inputs.

\subsection{Kronecker duals for all $\opc{M}$} \label{poincaredualsforallom}
For an arbitrary matrix $M$, we apply \cref{SchubertFinkThm} and \Cref{starunion}(b) to get $$[M]_\hbar=\sum_{\ell=1}^n\sum_{\emptyset \subsetneq X_1  \subsetneq \ldots \subsetneq X_\ell = \{1,\ldots,n\}}(-1)^{n-\ell}[\operatorname{Sch}(\rk_M(X_1),\ldots,\rk_M(X_\ell),X_1,\ldots,X_\ell)]_\hbar,$$
This expression in particular implies the uniqueness in \Cref{Mhbarintro} in light of the previous subsection. We have thus expressed $[M]_\hbar(f_1(H_1), \ldots, f_n(H_n))$ in terms of iterated mod $F^k$ over various groupings of the $\bar{f}_i$. By \Cref{mhbareasy}(d), applying the operator $[z^r][F(z)^{d-1}]$ to this expression yields $\opc{M}^\dagger(f_1(H_1),\ldots,f_n(H_n))$.

\section{Equivariant Chow class of the graph closure $\ol{\Gamma}_M$ of $\mu_M$}\label{graphclosuresection}
In this section, we compute the equivariant Chow class of the closure $\ol{\Gamma}_M \subset \mb{P}^{(r+1)\times d-1}\times (\mb{P}^r)^n$ of the graph of $\mu_M$, which yields \Cref{PFAintro} from \Cref{qcirc} (the other part involving $\ell$ formally follows for degree reasons). \Cref{graphvaluative}, which completes the proof of \Cref{realvalintro}, then follows from the valuativity of $[M]_{\hbar}$. We will use the specific resolution of $\mu_M$ given by the wonderful compactification of \cite{LiLi} described in Appendix \ref{LL} to describe $(\mu_M)_*$ in terms of generalized matrix orbits, and then use this to reconstruct $[\ol{\Gamma}_M]$ in terms of $[M]_{\hbar}^\dagger$ ( \Cref{Mhbardaggerdef}).

\begin{Thm}
\label{PFA}
Let $M$ be a $d \times n$ matrix, and let $$q : QH^\bullet_{GL_{r+1}}(\mb{P}^r) \to A^\bullet_{GL_{r+1}}(\mb{P}^{(r+1)\times d-1})\cong A^\bullet_{GL_{r+1}}(\pt)[H]/(F(H)^d)$$ be the quotient map, with $q(z) = H$ and $q(\hbar) = F(H)$. Then the equivariant Chow class \[[\ol{\Gamma}_M] \in A^\bullet_{GL_{r+1}}(\mb{P}^{(r+1)\times d-1}\times(\mb{P}^r)^n) = A^\bullet_{GL_{r+1}}(\mb{P}^r)^{\otimes n}\otimes A^\bullet_{GL_{r+1}}(\mb{P}^{(r+1)\times d-1})\] is equal to $(1 \otimes q)([M]_\hbar^\dagger)$.

Equivalently, let $a_k \in  A^\bullet_{GL_{r+1}} ((\mb{P}^r)^n)\otimes A^\bullet_{GL_{r+1}}(\mb{P}^{(r+1)\times d-1})$ be the result of replacing $H_{n+1}$ with $H$ in any expression for $\opc{\tau^{\le k+1}(M \oplus \ast)}$ with degree at most $r$ in $H_{n+1}$. (For example, we can use the formula given in \Cref{finkformulas}.)
Then
$$[\ol{\Gamma_M}]=\sum_{k=0}^{d-1} a_k F(H)^k.$$

\end{Thm}
\begin{Cor}
\label{qcirc}
We have  $q\circ[M]_\hbar(f_1(H_1),\ldots,f_n(H_n))=\mu_M^*(f_1(H_1)\cdots f_n(H_n))$.
\end{Cor}
\begin{Cor}\label{graphvaluative}
The map $M \mapsto [\ol{\Gamma_M}]$ from $d \times n$ matrices to equivariant Chow classes is valuative.
\end{Cor}
\begin{proof}[Proof of \cref{PFA}]
Let $V = K^{r+1}$ and $W = K^d$ and let $x_1, \ldots, x_n \in W$ be the columns of $M$. As in \Cref{ParAndSer}, we can think of the generalized matrix orbit map as the map $\mu_{M} : \mb{P}(V \otimes W^\vee) \dashrightarrow (\mb{P}(V))^n$ defined by the formula $\mu_{M}(\wt A) = (\wt{Ax_1}, \ldots, \wt{Ax_n})$ whenever $A \in \Hom(W, V) = V \otimes W^*$ is a matrix such that $A x_i \neq 0$ for $1 \leq i \leq n$.

Choose a general $x_{n+1} \in W$ and let $M'$ be the $d \times (n+1)$ matrix whose columns are $x_1, \ldots, x_{n+1}$. For $0 \leq i \leq d$, let $\Lambda \subset W$ be a general subspace of dimension $i$, and let $\overline{x}_1, \ldots, \overline{x}_{n+1} \in W / \Lambda$ be the images of $x_1, \ldots, x_{n+1}$. After identifying $W/\Lambda$ with $K^{d-i}$ by a fixed isomorphism, let $\overline{M'}$ be the $(d - i)\times (n+1)$ matrix whose columns are $\overline{x}_1, \ldots, \overline{x}_{n+1}$. By \Cref{truncgen}, $\overline{M'}$ is a realization of $\tau^{\le d-i}(M \oplus \ast)$, so $\opc[r][n+1]{\overline{M'}} = \opc[r][n+1]{\tau^{\le d-i}(M \oplus \ast)}$. The following key lemma relates the maps $\mu_{M}$, $\mu_{M'}$, and $\mu_{\overline{M'}}$.

\begin{Lem} \label{pushforwardkeylemma}
There exist resolutions $B$, $B'$, and $\overline{B'}$ of the maps $\mu_{M}$, $\mu_{M'}$, and $\mu_{\overline{M'}}$, respectively, that fit into a commutative diagram
\begin{equation} \label{pushforwardkeydiagram}
\begin{tikzcd}
\overline{B'}\arrow[ddr, phantom, "\square", description, font=\normalsize]\arrow[r,hook,"\wt{\iota}"]\arrow[dd,"\pi_{\ol{M'}}"]&B'\arrow[dd,bend right=25, swap, "\pi_{M'}"]\arrow[d,"b"]\arrow[ddrr,"\wt{\mu}_{M'}"]&&\\
&B\arrow[d,"\pi_M"]\arrow[dr,"\wt{\mu}_M"]&&\\
\mb{P}(V \otimes \Lambda^{\perp})\arrow[rrr,dashed,bend right=20]\arrow[r,hook,"\iota"]&\mb{P}(V \otimes W^\vee)\arrow[r,dashed]\arrow[rr,dashed,bend right=15]&\mb{P}(V)^n&\mb{P}(V)^{n+1}\arrow[l,swap,"p"]
\end{tikzcd}
\end{equation}
where $p : \mb{P}(V)^{n+1} \to \mb{P}(V)$ is the projection onto the first $n$ factors and $\iota : \mb{P}(V \otimes \Lambda^\perp) \to \mb{P}(V \otimes W^\vee)$ is the inclusion, satisfying the following equalities of equivariant Chow classes:
\begin{enumerate}[(i)]
\item $\pi_M^* H^j = b_*\wt{\mu}_{M'}^* (H_{n+1}^j)$ for $0 \leq j \leq r$, and
\item $\pi_{M'}^{*} (F(H)^i)  = \wt{\iota}_*[\overline{B'}]$.
\end{enumerate}
\end{Lem}
To prove \Cref{pushforwardkeylemma}, we will need the following proposition, which states that blowup commutes with restriction to a subvariety transverse to the center of blowup.
\begin{Prop}
\label{BUR}
Let $X$ be a smooth variety containing smooth subvarieties $Y,Z$ such that $Y\pitchfork Z$. Let $\pi: {\rm Bl}_Z X\to X$ be the blowup map. Then, the canonical embedding ${\rm Bl}_{Y\cap Z} Y\hookrightarrow \pi^{-1}(Y)$ is an isomorphism. Furthermore, we have the equality of Chow classes $\pi^*[Y] = [{\rm Bl}_{Y\cap Z} Y]$ (which is an equality of $G$-equivariant Chow classes if $X$ carries an action of a group $G$ and $Y$ and $Z$ are $G$-invariant).
\end{Prop}
\begin{proof}
By the blowup closure lemma, the proper transform of $Y$ in ${\rm Bl}_Z X$ is ${\rm Bl}_{Y\cap Z} Y$ and ${\rm Bl}_{Y\cap Z} Y$ is irreducible. Therefore, the map ${\rm Bl}_{Y\cap Z} Y\hookrightarrow \pi^{-1}(Y)$ is a closed immersion.

The exceptional divisor $E\subset {\rm Bl}_Z X$ is isomorphic to the projectivization of the normal bundle $\mb{P}(N_{Z/X})$. Since $Y\pitchfork Z$, $Y\cap Z$ is smooth and of the expected dimension $\dim(X)-\codim(Y)-\codim(Z)$. Hence, $\pi^{-1}(Y\cap Z)$ is smooth of dimension $\dim(X)-1-\codim(Y) = \dim(Y) - 1$.

Therefore, ${\rm Bl}_{Y\cap Z} Y\hookrightarrow \pi^{-1}(Y)$ is a bijection as it is a bijection over $Z\cap Y$. Finally, $\pi^{-1}(Y\cap Z)=\pi^{-1}(Y)\cap E\subset \pi^{-1}(Y)$ is cut out by one equation and smooth of codimension 1, so $\pi^{-1}(Y)$ must be smooth. By Zariski's main theorem,  ${\rm Bl}_{Y\cap Z} Y\hookrightarrow \pi^{-1}(Y)$ is an isomorphism. 

The equality $\pi^*[Y] = [{\rm Bl}_{Y\cap Z} Y]$ follows from \cite[Corollary~6.7.2]{Fulton}.
\end{proof}
\begin{proof}[Proof of \Cref{pushforwardkeylemma}]
We will construct the resolutions $B$, $B'$, and $\overline{B'}$ using the wonderful compactification as follows. Let $\mc{F}$ be the finite set of linear subspaces of $W^{\vee}$ given by
\begin{align*}
\mc{F}=\left\{\bigcap_{i\in S}{x_i^{\perp}}\mid S\subset \{1,\ldots,n\}\right\}\end{align*}
and let
\begin{align*}
\mc{S}=\{\mb{P}(V \otimes U)\mid U\in \mc{F}, \{0\}\subsetneq U \subsetneq W^\vee\}. 
\end{align*}
Let $G_1, \ldots, G_N$ be the elements of $\mc{S}$, ordered in such a way that $G_i \subset G_j$ implies $i \leq j$. Define the ordered sets
\begin{align*}
\mc{G} &= \{G_1, \ldots, G_N\}, \qquad
\mc{G}' = \{G_1, \ldots, G_N, \mb{P}(V \otimes x_{n+1}^\perp)\}, \\
\overline{\mc{G}} &= \{G_1 \cap \mb{P}(V \otimes \Lambda^\perp), \ldots, G_N \cap \mb{P}(V \otimes \Lambda^\perp), \mb{P}(V \otimes x_{n+1}^\perp) \cap \mb{P}(V \otimes \Lambda^\perp)\}.
\end{align*}
It follows from the genericity of $x_{n+1}$ and $\Lambda$ that $\mb{P}(V \otimes x_{n+1}^\perp)$ and $\mb{P}(V \otimes \Lambda^\perp)$ are transverse to each $G_i$ and to each other. Hence $\mc{G}$ and $\mc{G}'$ are ordered building sets of subvarieties of $\mb{P}(V \otimes W^*)$ and $\overline{\mc{G}'}$ is an ordered building set of subvarieties of $\mb{P}(V \otimes \Lambda^\perp)$ (in the sense of \Cref{iterBU}).

Let 
\begin{align*}
B = \mb{P}(V \otimes W^\vee)_\mathcal{G}, \qquad 
B' = \mb{P}(V \otimes W^\vee)_\mathcal{G'} , \qquad
\overline{B'} = \mb{P}(V \otimes \Lambda^\perp)_{\overline{\mathcal{G'}}}
\end{align*} be the corresponding wonderful compactifications.
By \Cref{RMM}, the spaces $B$, $B'$, and $\overline{B'}$ smoothly and $GL(V)$-equivariantly resolve the maps $\mu_M$, $\mu_{M'}$, and $\mu_{\overline{M'}}$ respectively.

By \Cref{iterBU}, the space $B'$ is the blowup of $B$ along the strict transform of $\mb{P}(V \otimes x_{n+1}^\perp)$ in $B$. Also $\overline{B'} = \pi_{M'}^{-1}(\mb{P}(V \otimes \Lambda^\perp))$ by a repeated application of \Cref{BUR}. Hence, we can construct the diagram \eqref{pushforwardkeydiagram}. It remains to prove equations (i) and (ii).

First, we show (i). Let $E$ be the exceptional divisor of $b$. Letting $\wt{H} = \pi_{M'}^* H$, we have $\wt{H}=\wt{\mu}_{M'}^{*}H_{n+1}+E$. Hence by rewriting $\mu_{M'}^{*}(H_{n+1}^j)= (\wt{H}-E)^j$, we get
\begin{align}
b_{*}\wt{\mu}_{M'}^{*}(H_{n+1}^j)=\sum_{k=0}^{j}(-1)^{j-k}\binom{j}{k}b_*(\wt{H}^{k}E^{j-k}). \label{HEsum}
\end{align}
Let $ \iota_E : E \hookrightarrow B'$ be the inclusion of the exceptional divisor, which is $GL(V)$-equivariant. For $k<j$, the class $\wt{H}^{k}E^{j-k}=(\iota_E)_* \iota_E^* (\wt{H}^{k}E^{j-k-1})$ is the pushforward to $B'$ of a class in $A^{j-1}(E)$. Since the map $b$ restricted to $E$ has $r$-dimensional fibers and $r > j-1$, this implies that $b_{*}(\wt{H}^{k}E^{j-k}) = 0$. Hence each term of \eqref{HEsum} vanishes except the term with $k=j$, and
\[
b_{*}\wt{\mu}_{M'}^{*}(H_{n+1}^j) = b_* (\wt{H}^j) = b_* \pi_{M'}^* (H^j) = b_* b^* \pi_{M}^* (H^j) = \pi_{M}^* (H^j),
\]
which is exactly (i).

Equation (ii) follows from $\pi_{M'}^{*} (F(H)^i) = \pi_{M'}^{*} \iota_*[\mb{P}(V \otimes \Lambda^\perp)]$ and repeated applications of \Cref{BUR}.
\end{proof}

Continuing the proof of the theorem, for $\beta\in A^\bullet_{GL_{r+1}}((\mb{P}^r)^n)$, $0\leq i\leq d$, $0\leq j\leq r$, we claim
\begin{equation} \label{pushforwardintegral}
\int_{(\mb{P}^r)^n}{(\mu_M)_*(F(H)^i H^j)\cap \beta}=\int_{(\mb{P}^r)^{n+1}}{\opc[r][n+1]{\tau^{\le d-i}(M \oplus \ast)}\cap H_{n+1}^j \cap p^* \beta},
\end{equation}
where $p : (\mb{P}^r)^{n+1}\to (\mb{P}^r)^{n}$ is the projection to the first $n$ factors. Indeed, by \Cref{pushforwardkeylemma},
\begin{align*}
\int_{(\mb{P}^r)^n} (\mu_{M})_{*}(F(H)^i H^j)\cap \beta &= \int_{B}\pi_M^*(F(H)^i) \cap \pi_M^*(H^j)\cap \wt{\mu}_{M}^*\beta \\
&= \int_{B}\pi_M^*(F(H)^i) \cap b_* \wt{\mu}_{M'}^*(H_{n+1}^j)\cap \wt{\mu}_{M}^*\beta \\
&= \int_{B'} \pi_{M'}^*(F(H)^i) \cap \wt{\mu}_{M'}^*(H_{n+1}^j) \cap \wt{\mu}_{M'}^* p^* \beta \\
&= \int_{B'} \wt{\iota}_*[\overline{B'}] \cap \wt{\mu}_{M'}^*(H_{n+1}^j \cap p^* \beta) \\
&= \int_{\overline{B'}} \wt{\mu}_{\overline{M'}}^*(H_{n+1}^j \cap p^* \beta)= \int_{(\mb{P}^r)^{n+1}}{\opc[r][n+1]{\tau^{\le d-i}(M \oplus \ast)}\cap H_{n+1}^j \cap p^* \beta}.
\end{align*}

Now, we complete the proof of the theorem. The two statements of the theorem are equivalent by the definition of $[M]_\hbar$, so we wish to show that $[\ol{\Gamma}_M] = (1 \otimes q)([M]_\hbar^\dagger)$. It suffices to show \[\int_{\mb{P}^{(r+1)\times d - 1} \times (\mb{P}^r)^n}[\ol{\Gamma}_M] \cap (F(H)^i H^j) \cap \beta = \int_{\mb{P}^{(r+1)\times d - 1} \times (\mb{P}^r)^n} (1 \otimes q)([M]_\hbar^\dagger) \cap (F(H)^i H^j) \cap \beta\] for $\beta\in A^\bullet_{GL_{r+1}}((\mb{P}^r)^n)$, $0\leq i\leq d$ and $0\leq j\leq r$.

The pushforward $(\mu_M)_*$ does not depend on the choice of the resolution of $\mu_M$, so we can take the resolution $\ol{\Gamma_M} \subset \mb{P}^{(r+1)\times d - 1} \times (\mb{P}^r)^n$. Then
\begin{align*}
\int_{\mb{P}^{(r+1)\times d - 1} \times (\mb{P}^r)^n}[\ol{\Gamma_M}] \cap (F(H)^i H^j) \cap \beta 
&=\int_{(\mb{P}^r)^n} (\mu_{M})_{*}(F(H)^i H^j)\cap \beta \\
&=\int_{(\mb{P}^r)^{n+1}}{\opc{\tau^{\le d-i}(M \oplus \ast)}\cap H_{n+1}^j \cap p^* \beta}\\
&=\int_{\mb{P}^{(r+1)\times d - 1} \times (\mb{P}^r)^n} (1 \otimes q)([M]_\hbar^\dagger) \cap (F(H)^i H^j) \cap \beta.
\end{align*}
\end{proof}
\section{Application: line sections of a hypersurface}
\label{application}
We consider the following problem.
\begin{Pro}
\label{Patel}
For a general hypersurface $X\subset\mb{P}^r$ of degree $d$ with $d=2r+1$, what is the degree of the map
\begin{align*}
\phi_X: \mb{G}(1,r)\dashrightarrow {\rm Sym}^d\mb{P}^1//SL_2\cong \ol{\mc{M}_{0,d}}/S_d
\end{align*}
that sends each line $\ell\subset \mb{P}^r$ to the moduli of the $d$ unordered points $\ell \cap X$ on $\ell\cong \mb{P}^1$?
\end{Pro}

When $r=2$, the answer to Problem \ref{Patel} was shown to be 420 in \cite{Laza}. We give a separate approach that yields the answer for all $r$. More generally, in the threshold $d\geq 2r-2$ where a general such $X$ contains no lines, we find the class of a general fiber of $\phi_X$ in $\mb{G}(1,r)$. 

We can also ask a related question.
\begin{Pro}
\label{Patel2}
With the setup of \Cref{Patel}, when $r=2$ Cadman and Laza \cite{Laza} show 
\begin{align*}
\deg(\phi_X)= 2\cdot (\text{\# bitangents of $X$})+4\cdot (\text{\# flexes of $X$}) = 2\cdot 120 + 4\cdot 45 = 420.
\end{align*}
How does this generalize to arbitrary $r$? 
\end{Pro}

We solve Problem \ref{Patel} in \Cref{AP} and answer Problem \ref{Patel2} in \Cref{APFS}. 

\subsection{Connection with $\op{M}$}
For the rest of \Cref{application}, all of our cycles $\op{M}$ are in a product of $\mb{P}^1$'s. For any matrix $M$, denote by $(\op[1]{M})_{\mb{P}(\mc{S}^{\vee})^d}$ the pullback of $\op{M}$ from $\mb{P}(\zeta_2)^d$, where $\zeta_2$ is the universal rank $2$ vector bundle. By an abuse of notation, we will drop the subscript $\mb{P}(\mc{S}^{\vee})^d$ when it is clear from context. 
\begin{Lem}
\label{section}
Let $f$ be a homogeneous form of degree $d$ on $K^{r+1}$ such that $X = V(f) \subset \mb{P}^r$ contains no lines. Denoting by $\mc{S}$ the tautological bundle on $\mb{G}(1,r)$, there is an everywhere nonzero section $\mb{G}(1,r)\to {\rm Sym}^d\mc{S}^\vee$ induced by $f$, and the class of the image of the induced section $s: \mb{G}(1,r)\to \mb{P}({\rm Sym}^d\mc{S}^\vee)$ is given by
\begin{align*}
[\im(s)]=H^{d}+c_1({\rm Sym}^d\mc{S}^\vee)H^{d-1}+\cdots+c_{d}({\rm Sym}^d\mc{S}^\vee).
\end{align*}
\end{Lem}

\begin{proof}
Let $V=K^{r+1}$. Then, the canonical inclusion of $\mc{S}\subset V\times \mb{G}(1,V)$ of vector bundles over $\mb{G}(1,V)$ induces a surjection ${\rm Sym}^d(V^\vee)\times \mb{G}(1,V)\to{\rm Sym}^d(\mc{S}^\vee)$ of vector bundles. The section $\mb{G}(1,V)\to {\rm Sym}^d\mc{S}^\vee$ is the composition 
\begin{align*}
\mb{G}(1,V)\to {\rm Sym}^d(V^\vee)\times \mb{G}(1,V)\to {\rm Sym}^d\mc{S}^\vee.
\end{align*}
If $X=V(f)$ contains no lines, then $f$ restricts to something nonzero on each line, so $\mb{G}(1,V)\to {\rm Sym}^d\mc{S}^\vee$ avoids zero. The formula for $[\im(s)]$ follows from \cite[Proposition 9.13]{3264}. 
\end{proof}

\Cref{IN} describes the class of a general fiber of $\phi_X$, and contains a rigorous definition of $\phi_X$ in its proof.

\begin{Lem}
\label{IN}
With the setup of Lemma \ref{section}, assume every line in $\mb{P}^r$ meets $X$ in at least three points. The general fiber of
\begin{align*}
\phi_X : \mb{G}(1,r)\dashrightarrow \mb{P}({\rm Sym}^d (K^{\vee})^{2})//SL_{2}
\end{align*}
is closed in $\mb{G}(1,r)$ and has class $s^{*}\Phi_*\opc[1]{M}_{\mb{P}(\mc{S}^{\vee})^d}$, where $M$ is a general $2 \times d$ matrix and $\Phi$ is the multiplication map
\begin{align*}
\Phi:\mb{P}(\mc{S}^{\vee})^d\to \mb{P}({\rm Sym}^d \mc{S}^{\vee}), 
\end{align*}
where $\mb{P}(\mc{S}^{\vee})^d$ denotes the $d$-fold fiber product of $\mb{P}(\mc{S}^{\vee})$ with itself over $\mb{G}(1,r)$.
\end{Lem}

\begin{proof}
Let $U\subset \mb{G}(1,r)$ be an open subset over which $\mc{S}\to \mb{G}(1,r)$ is trivial. Taking a trivialization, we have an $SL_{2}$-equivariant map $\mb{P}({\rm Sym^d} \mc{S}^{\vee})|_{U}\to U$. We can take the GIT quotient, which gives us
\begin{center}
\begin{tikzcd}
\mb{P}({\rm Sym^d} \mc{S}^{\vee})|_{U} \arrow[dr] \arrow[rr, dashed] & & \mb{P}({\rm Sym^d}\mc{S}^{\vee})|_{U}//SL_{2} \arrow[dl]\\
& U&
\end{tikzcd}
\end{center}
Given another open subset $U'\subset \mb{G}(1,r)$ over which $\mc{S}\to \mb{G}(1,r)$ is trivial, we get identifications of the GIT quotients over $U\cap U'$ that are independent of the trivializations chosen over $U$ and $U'$.
\begin{center}
\begin{tikzcd}
(\mb{P}({\rm Sym^d} \mc{S}^{\vee})|_{U})|_{U'} \arrow[d,dashed] \arrow[r, "\sim"] & (\mb{P}({\rm Sym^d} \mc{S}^{\vee})|_{U'})|_{U} \arrow[d,dashed]  \\
(\mb{P}({\rm Sym^d} \mc{S}^{\vee})|_{U}//SL_{2})|_{U'} \arrow[r,equal] &(\mb{P}({\rm Sym^d} \mc{S}^{\vee})|_{U'}//SL_{2})|_{U} 
\end{tikzcd}
\end{center}
This means the GIT quotients over each $U$ glue together to a global object, and this global object is precisely $\mb{G}(1,r)\times (\mb{P}({\rm Sym}^d (K^{\vee})^2)//SL_{2})$. We have an induced map
\begin{align}\label{symmapexists}
\mb{P}({\rm Sym^d} \mc{S}^{\vee})\dashrightarrow \mb{G}(1,r)\times (\mb{P}({\rm Sym}^d (K^{\vee})^2)//SL_{2})\to \mb{P}({\rm Sym}^d (K^{\vee})^2)//SL_{2}
\end{align}
and the map $\phi_X$ can be defined as the composite
\begin{align*}
\phi_X:\mb{G}(1,r)\xrightarrow{s} \mb{P}({\rm Sym^d} \mc{S}^{\vee})\dashrightarrow \mb{G}(1,r)\times (\mb{P}({\rm Sym}^d (K^{\vee})^2)//SL_{2}) \to \mb{P}({\rm Sym}^d (K^{\vee})^2)//SL_{2}.
\end{align*}
Let $U_s\subset \mb{P}({\rm Sym^d} \mc{S}^{\vee})$ denote the dense open subset of points that are stable with respect to the local GIT quotients above. The set $U_s$ consists of configurations of $d$ points on a line such that no point with multiplicity at least $\frac{d}{2}$ appears. The map \eqref{symmapexists} is defined on $U_{s}$ and maps a stable configuration of $d$ points to the class of its $SL_2$ orbit in $\mb{P}({\rm Sym}^d (K^{\vee})^2)//SL_{2}$. A general fiber of the map $U_{s}\to \mb{P}({\rm Sym}^d (K^{\vee})^2)//SL_{2}$ given by \eqref{symmapexists} is $\Phi_*\op[1]{M}\cap U_{s}$. 

The set $s^{-1}(U_s)$ is in the domain of definition of $\phi_X$, and the scheme $s^{-1}(\Phi_*\op[1]{M}\cap U_{s})$ is a general fiber of $\phi_X$ restricted to $s^{-1}(U_s)$. We claim that 
\begin{align*}
s(\mb{G}(1,r))\cap \Phi_*\op[1]{M}\subset U_s.
\end{align*}
Indeed, every configuration in the boundary of $\Phi_*\op[1]{M}$ (supported on fewer than $d$ points) must be supported on at most 2 points for dimension reasons, and thus $\im(s)$ does not intersect the boundary of $\Phi_*\op[1]{M}$ by assumption. Since the complement of the boundary of $\Phi_*\op[1]{M}$ lies in $U_s$, we have $s(\mb{G}(1,r))\cap \Phi_*\op[1]{M}\subset U_s$. 

Therefore, $s^{-1}(\Phi_*\op[1]{M}\cap U_{s})= s^{-1}(\Phi_*\op[1]{M})$, which completes the proof.
\end{proof}

\subsection{Computation of degree}
\label{APS}
\begin{Thm}
\label{AP}
Let $X\subset\mb{P}^r$ be a hypersurface of degree $d$ such that every line meets $X$ in at least three points and is not contained in $X$. After restricting the rational map
\begin{align*}
\phi_X: \mb{G}(1,r)\dashrightarrow \mb{P}({\rm Sym}^d (K^{\vee})^{2})//SL_{2}
\end{align*}
to its domain of definition, the general fiber is closed in $\mb{G}(1,r)$ and has class 
\begin{align*}
d(d-1)(d-2)\prod_{k=2}^{d-2} (ku+(d-k)v),
\end{align*}
where $u,v$ are the Chern roots of the dual $\mc{S}^{\vee}\to \mb{G}(1,r)$ of the tautological bundle. In particular, if $d=2r+1$, then the degree of $\phi_X$ is the difference between the coefficients of $u^{r-1}v^{r-1}$ and $u^{r-2}v^{r}$ in the expansion of $d(d-1)(d-2)\prod_{k=2}^{d-2} (ku+(d-k)v)$. 
\end{Thm}

\begin{Exam}
For $(d,r)=(7,3)$, we are considering line sections of a degree $7$ surface in $\mb{P}^3$. \Cref{AP} says the degree in this case is 77070.  
\end{Exam}

\begin{proof}
Let $M$ be a general $2\times d$ matrix. By Lemma \ref{IN}, the class of a general fiber of $\phi_X$ is $s^{*}\Phi_*\opc[1]{M}_{\mb{P}(\mc{S}^{\vee})^d}$.
\begin{center}
\begin{tikzcd}
\mb{P}(\mc{S}^\vee)^d \arrow[dr, "\pi_1"] \arrow[rr, "\Phi"] & &  \mb{P}({\rm Sym^d} \mc{S}^{\vee})  \arrow[dl, "\pi_2"]\\
& \mb{G}(1,r) \arrow[ur, swap, bend right=30, "s"]&
\end{tikzcd}
\end{center}
By the projection formula, we have
\begin{align*}
s^{*}\Phi_*\opc[1]{M}_{\mb{P}(\mc{S}^{\vee})^d} =(\pi_2)_{*}s_{*}(s^{*}\Phi_*\opc[1]{M}_{\mb{P}(\mc{S}^{\vee})^d})
&=(\pi_2)_{*}(s_{*}[\mb{G}(1,r)]\cap\Phi_*\opc[1]{M}_{\mb{P}(\mc{S}^{\vee})^d})\\
&=(\pi_2)_{*}\Phi_*(\Phi^*s_{*}[\mb{G}(1,r)]\cap\opc[1]{M}_{\mb{P}(\mc{S}^{\vee})^d})\\
&=(\pi_1)_*(\Phi^*[\im(s)] \cap \opc[1]{M}_{\mb{P}(\mc{S}^{\vee})^d}).
\end{align*}

To compute $(\pi_1)_*(\Phi^*[\im(s)] \cap \opc[1]{M}_{\mb{P}(\mc{S}^{\vee})^d})$, we will use \Cref{section} and \Cref{PDGM}.

Let $F(z) = (z+u)(z+v)$ be the Leray relation of $\mb{P}(\mc{S}^\vee)$, let
$$G(z)=\prod_{k=0}^d (z+ku+(d-k)v) = z^{d+1} + c_{1}({\rm Sym}^d\mc{S}^\vee)z^d + \cdots + c_{d+1}({\rm Sym}^d\mc{S}^\vee)$$ be the Leray relation of $\mb{P}({\rm Sym}^d\mc{S}^\vee)$ and let $H=\ms{O}_{\mb{P}({\rm Sym}^d\mc{S}^\vee)}(1)$. By \Cref{section}, we have $[\im(s)] = L(H)$, where $L$ is the polynomial
$$L(z)=\frac{G(z)-G(0)}{z}.$$ Hence $\Phi^*[\im(s)] = L(\sum_{i=1}^d H_i)$, where $H_i$ denotes $\ms{O}_{\mb{P}(\mc{S}^{\vee})}(1)$ pulled back from the $i$th factor. 

Let $\ol{L}(H_1, \ldots, H_d)$ be the result of reducing $L(\sum_{i=1}^d H_i)$ modulo $F(H_i)=(H_i+u)(H_i+v)$ for each $i$. Then by \Cref{PDGM} we have  \[\pi_*(\Phi^*[\im(s)] \cap \opc[1]{M}_{\mb{P}(\mc{S}^{\vee})^d})=[z^1][F(z)^1] \ol{L}(z, \ldots, z).\] (As usual, $[F(z)^1] \ol{L}(z, \ldots, z)$ denotes the coefficient $a_1$ when $\ol{L}(z, \ldots, z)$ is expressed as $a_0(z) + a_1(z) F(z) + a_2 F(z)^2 + \cdots$, where each $a_i$ is a polynomial of degree at most $1$ in $z$.)

We can thus carry out the computation of the class of the general fiber of $\phi_X$ in three steps: first we compute $\ol{L}(H_1, \ldots, H_n)$, second we compute $\ol{L}(z, \ldots, z)$, and third we compute the answer $[z^1][F(z)^1] \ol{L}(z, \ldots, z)$. We will perform these steps completely formally, treating all variables as indeterminates with no relations between them, i.e. in the ring $\mb{Q}(u,v,H_1,\ldots,H_n, z)$. In particular, we will not use any relations between the $H_i$ and $u,v$ in the relevant Chow rings.

\textbf{Step 1} Note that $\ol{L}(H_1, \ldots, H_n)$ has degree at most $1$ in each $H_i$ and agrees with $L(\sum_{i=1}^d H_i)$ on $\{-u, -v\}^d$. Hence we may compute $\ol{L}(H_1, \ldots, H_n)$ by evaluating $L(\sum_{i=1}^d H_i)$ on $\{-u, -v\}^d$ and using the Lagrange interpolation formula.  Observe that the term $G(H_1 + \ldots + H_d)$ of \[L(H_1 + \ldots + H_d) = \frac{G(H_1 + \ldots + H_d) - G(0)}{H_1 + \ldots + H_d}\] vanishes on $\{-u, -v\}^d$, so we get
$$\ol{L}(H_1, \ldots, H_n) = \sum_{T \subset \{1,\ldots,d\}} \frac{-G(0)}{-|T| v - (d - |T|) u}\left(\prod_{i \in T} \frac{H_i + u}{-v + u}\right) \left(\prod_{i \in \{1, \ldots, d\} \setminus T} \frac{H_i + v}{-u + v}\right).$$

\textbf{Step 2}
Substituting $z$ for each $H_i$ yields
\begin{align}
\ol{L}(z, \ldots, z) &= G(0)\sum_{k=0}^d \frac{1}{kv+(d-k)u} \binom{d}{k} \frac{(z+u)^k(z+v)^{d-k}}{(u-v)^k(v-u)^{d-k}} \nonumber \\
&= \frac{G(0)}{(u - v)^d}\sum_{k=0}^d \frac{\binom{d}{k}(z+u)^k(-z-v)^{d-k}}{kv+(d-k)u}  \label{lbarsum}
\end{align}

\textbf{Step 3} All terms of \eqref{lbarsum} except for the $k=0,1,d-1,d$ terms are divisible by $F(z)^2$. Thus 

\begin{alignat}{2}
[z^1][F(z)^1] \ol{L}(z, \ldots, z)&=\frac{G(0)}{(u - v)^d}[z^1][F(z)^1]&&\left(\frac{(-z-v)^d}{du} + \frac{d(z+u)(-z-v)^{d-1}}{v+(d-1)u} \right. \nonumber \\
& && \left. + \frac{d(z+u)^{d-1}(-z-v)}{(d-1)v+u}+ \frac{(z+u)^d}{dv} \right)\nonumber \\
&=\frac{G(0)}{(u - v)^d}[z^1][F(z)^1] && \left(\frac{(-1)^d(z+v)^d}{du} + \frac{(-1)^{d-1}dF(z)(z+v)^{d-2}}{v+(d-1)u} \right. \label{lbarfourterms} \\
& &&\left. - \frac{dF(z)(z+u)^{d-2}}{(d-1)v+u}+ \frac{(z+u)^d}{dv} \right). \nonumber
\end{alignat}

We note now that 
\begin{align*}
(z+u)^k\equiv &(z+u)(u-v)^{k-1}\\
& +F(z)(-(z+v)(u-v)^{k-3}+(z+u)(k-1)(u-v)^{k-3}) \pmod{F(z)^2}.
\end{align*}
because both sides of the equation and their first derivatives agree on $\{-u, -v\}$. Hence 
\begin{alignat*}{2}
[z^1][F(z)^1] F(z)(z+u)^k &= (u-v)^{k-1},\qquad
[z^1][F(z)^1] (z+u)^k &&= (k-2)(u-v)^{k-3}, \\
[z^1][F(z)^1] F(z)(z+ v)^k &= (v-u)^{k-1},\qquad
[z^1][F(z)^1](z + v)^k &&= (k-2)(v-u)^{k-3}.
\end{alignat*} Applying this term by term to \eqref{lbarfourterms}, we obtain
\begin{align*}
&\phantom{=} \frac{G(0)}{(u - v)^d} \left(-\frac{(d-2)(u-v)^{d-3}}{du} + \frac{d(u-v)^{d-3}}{v+(d-1)u} - \frac{d(u-v)^{d-3}}{(d-1)v+u}+ \frac{(d-2)(u-v)^{d-3}}{dv} \right) \\
&=G(0)\frac{(d-1)(d-2)}{du(u+(d-1)v)((d-1)v+u)v}\\ 
&=d(d-1)(d-2)\prod_{k=2}^{d-2} (ku+(d-k)v),
\end{align*} 
which finishes the proof in general. In the case $d = 2r + 1$, the degree of $\phi_X$ is the $u^rv^r$ term when we decompose our expression into Schubert cycles, which are Schur polynomials in the Chern roots $u,v$ of $\mc{S}^{\vee}$ \cite[Chapter 10, Proposition 8]{FultonYoung}. This is the difference between the $u^{r-1}v^{r-1}$ coefficient and the $u^{r-2}v^{r}$ coefficient of $d(d-1)(d-2)\prod_{k=2}^{d-2} (ku+(d-k)v)$.
\end{proof}

\subsection{Degeneration of the general orbit}
\label{APFS}
We now answer \Cref{Patel2}, generalizing the relation in the case $r=2$ to arbitrary $r$, via a relation between generalized matrix orbit classes. The following definition will be used in \Cref{APFS} only. 

\begin{Def}
Given a partition $A_1\sqcup\ldots\sqcup A_k$ of $\{1,\ldots,n\}$, let $M_{A_1\sqcup\cdots\sqcup A_k}$ denote a choice of a $2\times n$ matrix such that two columns are linearly dependent if and only if their indices lie in the same $A_i$. Equivalently, we may define $$M_{A_1 \sqcup \ldots \sqcup A_k}= \tau^{\le 2}(\tau^{\le 1}(\oplus_{j \in A_1} \ast)\oplus\ldots\oplus\tau^{\le 1}(\oplus_{j \in A_k} \ast)).$$
\end{Def}

By \Cref{realval}, the class $\opc[1]{M_{A_1\sqcup\cdots \sqcup A_k}}$ depends only on the partition $A_1\sqcup\cdots \sqcup A_k$ and not on the specific choice of $M_{A_1\sqcup\cdots\sqcup A_k}$.

\begin{Prop}
\label{P1S}
We have 
\begin{align*}
\opc[1]{M_{\{1\}\sqcup\cdots\sqcup \{n\}}} = \sum_{i=1}^{n-2}{\opc[1]{M_{\{1,\ldots,i\}\sqcup \{i+1\}\sqcup \{i+2, \ldots,n\}}}}
\end{align*}
\end{Prop}

\begin{proof}
This follows by \cref{realval} and the subdivision of Schubert rank polytopes in the proof of \cref{formulaexists}. This also follows from an application of \cite[Theorem 1.3.6]{Kapranov} in the case where interior vertices of our tree form a chain. Not coincidentially, in characteristic zero, this also follows from \Cref{DTDC}.  
\end{proof}

As a corollary of \Cref{P1S}, we have
\begin{Thm}
\label{APD}
With the setup of \Cref{section}, if $d=2r+1$, then the degree of $\phi_X$ is
\begin{align*}
 \sum_{a\ge b>1,a+b+1=d}{2n_{a,b,1}}+4n_{d-2,1,1},
\end{align*}
where \[n_{a,b,c} = \int_{\mb{G}(1, r)}s^*[\ol{Z_{a,b,c}}]\] and $Z_{a,b,c} \subset {\rm Sym}^d(\mb{P}(\mc{S}^\vee)) = \mb{P}({\rm Sym}^d\mc{S}^\vee)$ is the closure of the locus of $n$-tuples of points in each fiber supported at $3$ points with multiplicities $a,b,c$.
\end{Thm}

\begin{Rem}
The inverse image $s^{-1}(Z_{a, b, c})$ consists exactly of lines that intersect $X$ at three points with multiplicities exactly $a,b,c$. Hence, we may think of $n_{a, b, c}$ as the number of such lines, counted with multiplicity. In characteristic zero, it is possible to give a transversality argument to show that, for a general $X$, we can take $n_{a,b,c}$ to be the honest number of lines, as all lines appear with multiplicity one. 
\end{Rem}

\begin{proof}
By \Cref{IN}, the degree of $\phi_X$ is equal to the degree of $s^{*}\Phi_*\opc[1]{M_{\{1\}\sqcup\cdots\sqcup \{d\}}}$, where
\begin{align*}
\Phi:\mb{P}(\mc{S}^{\vee})^d\to \mb{P}({\rm Sym}^d \mc{S}^{\vee})
\end{align*} is the multiplication map.
By \Cref{P1S}, we have
\begin{align*}
s^{*}\Phi_*\opc[1]{M_{\{1\}\sqcup\cdots\sqcup \{d\}}} = \sum_{i=1}^{d-2} s^{*}\Phi_*\opc[1]{M_{\{1,\ldots,i\}\sqcup \{i+1\}\sqcup \{i+2, \ldots,d\}}}.
\end{align*}
Now, observe that $\Phi$ maps $\op[1]{M_{\{1,\ldots,i\}\sqcup \{i+1\}\sqcup \{i+2,\ldots,d\}}}$ generically two-to-one onto $\ol{Z_{i,d-i-1, 1}}$ if $i\in\{1,\frac{d-1}{2},d-2\}$ and birationally onto it otherwise. Therefore $ s^{*}\Phi_*\opc[1]{M_{\{1,\ldots,i\}\sqcup \{i+1\}\sqcup \{i+2, \ldots,d\}}}$ is the class of $2n_{i,d-i-1,1}$ points if $i\in\{1,\frac{d-1}{2},d-2\}$ and $n_{i,d-i-1,1}$ points otherwise. Taking the sum from $i=1$ to $d-2$ yields the desired answer.
\end{proof}
\begin{Rem}
Using the properties of $[M]_\hbar$ in \Cref{poincaresection} and the description of $M_{A_1 \sqcup \ldots \sqcup A_k}$ in terms of truncations and direct sums with $\ast$, one can compute the $n_{a,b,1}$ directly.
\end{Rem}
\section{Lifting to $\mb{A}^{(r+1)\times n}$}\label{liftingsection}
In this section, we show that our formulas for $\opc{M}$ and $[\ol{\Gamma_M}]$ allow us to compute the $GL_{r+1}\times (K^\times)^n$-equivariant classes of the analogues of $\op{M}$ and $\ol{\Gamma_M}$ in $\mb{A}^{(r+1)\times n}$, thus completing and vastly extending the computation of Berget and Fink in \cite{Fink}. The precise formulation of the problem for $\op{M}$ is as follows.
\begin{Pro}
\label{AV}
Let $T = (K^\times)^n$ be the torus. We have a $GL_{r+1} \times T$-equivariant analogue
\begin{align*}
{\rm mult}_M:\mb{A}^{(r+1)\times d} \times_{K^\times} T \rightarrow \mb{A}^{(r+1)\times n}
\end{align*}
of $\mu_M$ that sends $(A,D) \in \mb{A}^{(r+1)\times d}\times_{K^{\times}} T$ to the product $AMD$. What is the class of the image closure of ${\rm mult}_M$ in $A^{\bullet}_{GL_{r+1} \times T}(\mb{A}^{(r+1)\times n})$ as an $((r+1)d+n-1)$-dimensional cycle?
\end{Pro}

By \Cref{SX} below, the formula in \cref{finkformulas} for $\opc{M}$ is precisely the same as the formula for the lifted class described in \cref{AV}, as it satisfies a certain degree bound of Feh\'er and Rim\'anyi \cite{FR07}. This formula works verbatim even if $M$ has zero columns because the locus of matrices with vanishing $i$th column has class $F(H_i)$. We can also lift $[\ol{\Gamma_M}]$ in an analogous fashion because the formula in \Cref{PFA} satisfies the degree bounds.


\subsection{Degree bound of Feh\'er and Rim\'anyi}
To pass from 
\begin{align*}
A^{\bullet}_{GL_{r+1}\times T}(\mb{A}^{(r+1)\times n})\cong \mb{Z}[t_0,\ldots,t_r]^{S_{r+1}}[H_1,\ldots,H_n]
\end{align*}
to $A^{\bullet}_{GL_{r+1}}((\mb{P}^r)^n)$, we remove the locus in $\mb{A}^{(r+1)\times n}$ of matrices containing a zero column and quotient out by the resulting free $T$-action. This amounts to taking the quotient of $A^{\bullet}_{GL_{r+1}\times T}(\mb{A}^{(r+1)\times n})$ by the ideal generated by $F(H_i)$ for $1\leq i\leq n$.

\begin{Lem}
[{\cite[Theorem 2.1]{FR07}}]
\label{FRL}
Let $V_1,V_2$ be vector spaces and let $G$ be an algebraic group that acts on $V_1\oplus V_2$. Let $K^{\times}$ act on $V_1$ trivially and on $V_2$ by scaling. Let $X\subset V_1\oplus V_2$ be an irreducible closed $G\times K^{\times}$-invariant subvariety not contained in $V_1$. 

Then, under the isomorphism $A_{G\times K^{\times}}^{\bullet}(V_1\oplus V_2)\cong A_{G}^{\bullet}(\pt)\otimes \mb{Z}[H]$, the class $[X]$ maps to an expression $\sum x_j\otimes H^j$ where $x_j=0$ if $j\geq \dim(V_2)$. 
\end{Lem}

\begin{proof}
This is a special case of the argument in \cite[Theorem 2.1]{FR07} applied to the case $V=V_1$, $V^{\#}=V_1\oplus V_2$, $G=G$, $H=K^{\times}$, $\eta^{\#}=X$, and $\eta=X\cap V_1$. Since the argument is short and our setting is with equivariant Chow groups instead of equivariant cohomology, we recall the proof for completeness. 

Consider the diagram
\begin{center}
\begin{tikzcd}
A_{G}^{\bullet}(\pt)\otimes \mb{Z}[H] \arrow[r,"\sim"]& A_{G\times K^{\times}}^{\bullet}(V_1\oplus V_2) \arrow[r] \arrow[d, "\sim"] & A_{G\times K^{\times}}^{\bullet}((V_1\oplus V_2)\backslash X) \arrow[d]\\
& A_{G\times K^{\times}}^{\bullet}(V_1) \arrow[r] & A_{G\times K^{\times}}^{\bullet}(V_1\backslash (X\cap V_1))
\end{tikzcd}
\end{center}
The class $\sum x_j\otimes H^j$ maps to zero in $A_{G\times K^{\times}}^{\bullet}((V_1\oplus V_2)\backslash X)$, so the image of $\sum x_i\otimes H^j$ in $A_{G\times K^{\times}}^{\bullet}(V_1) $ is in the kernel of $A_{G\times K^{\times}}^{\bullet}(V_1) \to A_{G\times K^{\times}}^{\bullet}(V_1\backslash (X\cap V_1))$. In particular, this means that each $x_j$ is in the kernel of $A_{G\times K^{\times}}^{\bullet}(V_1) \to A_{G\times K^{\times}}^{\bullet}(V_1\backslash (X\cap V_1))$. 

Since $X$ is irreducible and not contained in $V_1$, we have $\dim(X\cap V_1)\leq \dim(X)-1$, or equivalently, the codimension of $X\cap V_1$ in $V_1$ is at most the codimension of $X$ in $V_1\oplus V_2$ minus $(\dim(V_2)-1)$. If $j\geq \dim(V_2)$, then the codimension of $x_j$ is at most the codimension of $X$ in $V_1\oplus V_2$ minus $\dim(V_2)$. Since it is in the kernel of $A_{G\times K^{\times}}^{\bullet}(V_1) \to A_{G\times K^{\times}}^{\bullet}(V_1\backslash (X\cap V_1))$, it must be zero.
\end{proof}

\begin{Cor}
\label{SX}Let $G$ be an algebraic group and let $V_1,\ldots,V_n$ be representations of $G$. Let $T=(K^{\times})^n$ act on $V_1\oplus\cdots\oplus V_n$ by scaling each summand separately.

Let $X\subset V_1\times\cdots\times V_n$ be an irreducible closed $G \times T$-invariant subvariety intersecting $(V_1\backslash \{0\})\times\cdots\times (V_n\backslash \{0\})$ and let $f$ be a polynomial with coefficients in $A^\bullet_{G}(\pt)$ such that $f(H_1, \ldots, H_n)$ is the class
\begin{align*}
[X\cap (V_1\backslash \{0\})\times\cdots\times (V_n\backslash \{0\})] &\in A_{G\times T}^{\bullet}((V_1\backslash \{0\})\times\cdots\times (V_n\backslash \{0\})) \\
&= A_{G}^{\bullet}(\mb{P}(V_1)\times\cdots\times\mb{P}(V_n)) \\
&= A_{G}^{\bullet}(\pt)[H_1,\ldots,H_n]/(p_{V_1}(H_1),\ldots,p_{V_n}(H_n)),
\end{align*}
where $p_V(z) = z^{\dim V} + c_1^G(V) z^{\dim V - 1} + \ldots + c_{\dim V}^G(V)$ denotes the Leray relation of $V$.

Assume that $f$ has degree less than $\dim(V_i)$ in each $H_i$. Then the class \[[X] \in A^\bullet_{G \times T}(V_1 \times \cdots \times V_n) = A^\bullet_{G}(\pt)[H_1, \ldots, H_n]\] is equal to $f(H_1, \ldots, H_n)$.

\end{Cor}

\begin{proof}

The result follows from a repeated application of \Cref{FRL}. 
\end{proof}
Hence, the formulas in \Cref{finkformulas,PFA} solve \Cref{AV}.

\appendix
\section{Polytopes}
\label{polyapp}
We collect here facts about (convex) polytope geometry and matroid polytopes that we need in the body of the paper. Not all polytopes we will work with are bounded.

\begin{Def}
\label{vertdef}
Say $P$ is a \emph{cone} if $|\Vt(P)| \le 1$. Equivalently, $P$ is the Minkowski sum of a point $v$ and rays $\mb{R}_{\ge 0}v_i$. If such a cone does not contain a line, then $v$ is the unique vertex of the cone, called its \emph{apex}. 
\end{Def}
If a cone $v + \sum_i \mb{R}_{\ge 0}v_i$ has an apex, then its edges are a subset of the rays $v+\mb{R}_{\ge 0}v_i$.
\begin{Def}
The \emph{tangent cone} of a polytope $P$ at a point $p \in P$ is $$T_{p,P}=p+\bigcup_{p' \in P} \mb{R}_{\ge 0}(p'-p).$$ If the polytope $P$ is understood, we will denote the tangent cone simply by $T_p$. There exists a finite set $S \subset P$ such that $T_p=p+\sum_{p' \in S}\mb{R}_{\ge 0} (p'-p)$, so $T_p$ a polytope and a cone. 
\end{Def}
An equivalent description of $x \in P$ to be a vertex is that $T_x$ contains no line.

Given polytopes $P$, $Q$, and points $p\in P$, $q \in Q$, we have $T_{p+q, P+Q}=T_{p, P}+T_{q,Q}$. As the sum of cones which contain lines is a cone which contains lines, we have $$\Vt(P+Q) = \{x+y \mid x \in \Vt(P), y \in \Vt(Q), T_x+T_y\text{ does not contain a line}\}.$$
More generally, every face of $P+Q$ is the Minkowski sum of a face of $P$ with a face of $Q$.

The following is an equivalent characterization of rank polytopes arising from matroids (not necessarily arising from matrices).
 
\begin{Lem} \cite{GGM}
A polytope $P \subset \Delta$ is the rank polytope of a matroid if and only if $\Vt(P) \subset \Vt(\Delta)$ and all edges of $P$ are parallel to vectors of the form $e_i - e_j$. In particular, every nonempty face of a rank polytope is also a rank polytope.
\end{Lem}

This is the only property of rank polytopes we will be using, and to emphasize this, in this section we work with matroids $\MM$ and their rank polytopes $P_\MM$ rather than with matrices.

\begin{Def}
\label{matroidalconedef}
We define a \emph{matroidal} cone $C$ to be a cone with apex at the origin which is generated by rays of the form $\mb{R}_{\ge 0}(e_i-e_j)$. If $C$ is a matroidal cone generated by lines, we define the equivalence relation $\sim_C$ on $\{1, \ldots, n\}$ by $i \sim_C j$ if $e_i - e_j \in C$ and we call $C$ the \emph{associated cone} to $\sim_C$.
\end{Def}
Note that if $C$ is a matroidal cone generated by lines, then \[C = \left\{(x_1, \ldots, x_n) \in \mb{R}^n \middle| \text{$\sum_{i \in S_j} x_i = 0$ for all $j$}\right\},\] where $S_1, \ldots, S_k$ are the equivalence classes of $\sim_C$.

\begin{Cor}
\label{equivrel}
Given a matroid $\MM$, there exists a unique partition $S_1 \sqcup \ldots \sqcup S_k$ of $\{1, \ldots, n\}$ such that $P_\MM$ is a full-dimensional subset of
$$T_z=\left\{(x_1, \ldots, x_n) \in \mb{R}^n \middle| \sum_{i \in S_j} x_i=\rk_{\MM}(S_j)\text{ for all $j$}\right\}$$ for any relative interior point $z \in P_\MM$. In particular, for any point $(x_1, \ldots, x_n) \in P_\MM$ we have $\sum_{i \in S_j} x_j \in \mb{Z}$ for all $j$.

\end{Cor}
\begin{proof}
The tangent cone at $z$ is the Minkowski sum of the tangent cones of the vertices of $P_\MM$, so it is generated by some collection of vectors of the form $e_i-e_j$ and their negations. Now take $S_1, \ldots, S_k$ to be the equivalence classes of $\sim_{T_z - z}$.
\end{proof}

We will denote by $\sim_{\MM}$ the equivalence relation $\sim_{T_{z}-z}$ for any relative interior point $z$ of $P_\MM$.

We are interested in polytopes of the form $P_\MM+C$ for $C$ a matroidal cone and the intersections $(P_\MM+C) \cap \Delta$. We prove the key facts about these polytopes below. Recall that by abuse of notation we identify the vertices of $\Delta$ with $d$-element subsets of $\{1, \ldots, n\}$.
\begin{Lem} \label{horribleequivalence}
Let $\MM$ be a matroid and let $C$ be a matroidal cone. Then all edges of $P_\MM+C$ are parallel to vectors of the form $e_i-e_j$ and $(P_\MM+C)\cap \Delta$ is a matroid rank polytope with
\begin{align}
\Vt((P_\MM+C) \cap \Delta)=\bigcup_{w \in \Vt(P_\MM)} (w+C) \cap \Vt(\Delta) \label{vertpmplusc1} \\
=\{z \in \Vt(\Delta) \mid \exists w \in \Vt(P_\MM),\text{bijection } \phi:z\setminus w \to w \setminus z\text{ with }e_i-e_{\phi(i)}\in C\ \forall i\} \label{vertpmplusc2}
\end{align}
\end{Lem}
\begin{proof}
The edges of $P_M+C$ are parallel to vectors of the form $e_i-e_j$ by the results on Minkowski sums. We now prove \eqref{vertpmplusc1} and \eqref{vertpmplusc2}, postponing the proof that $(P_\MM+C)\cap \Delta$ is a matroid rank polytope until the end.

Suppose $z \in \Vt((P_\MM+C) \cap \Delta)$. Then there exists $w \in P_\MM$ such that $w+\sum a_{\alpha}(e_{i_\alpha}-e_{j_\alpha})=z$ where each $a_\alpha>0$ and each $e_{i_\alpha}-e_{j_\alpha} \in C$. Let $C'=\sum_{\alpha}\mb{R}(e_{i_\alpha}-e_{j_\alpha})$. Let $\MM'$ be the matroid such that $P_{\MM'}$ is the face of $P_\MM$ that contains $w$ in its relative interior. 
Let $\sim_{\MM'+C'}$ be the equivalence relation generated by $\sim_{\MM'}$ and $\sim_{C'}$. Then by the above $T_{z, P_\MM + C} - z$ contains the associated cone of $\sim_{\MM' + C'}$.

By \Cref{equivrel}, $w$ has integral sum of coordinates on any $\sim_{\MM'}$ equivalence class, and $z$ has the same sum of coordinates on any $\sim_{C'}$ equivalence class as $w$, so $z$ has integral sum of coordinates on any $\sim_{\MM'+C'}$-equivalence class.

We first show \eqref{vertpmplusc1}. Suppose $z_i \in (0,1)$ for some $i$. Then there must exist an $i'\sim_{\MM'+C'}i$ with $i' \neq i$ for which $z_{i'} \in (0,1)$. But then for sufficiently small $\epsilon>0$, we have $z\pm \epsilon(e_i-e_{i'}) \in P_\MM+C$ as $T_{z, P_\MM + C} - z$ contains the associated cone of $\sim_{\MM' + C'}$. Also $z \pm \epsilon(e_i-e_{i'}) \in \Delta$, so $z\pm \epsilon(e_i-e_{i'}) \in (P_\MM+C) \cap \Delta,$
contradicting $z \in \Vt((P_\MM+C) \cap \Delta)$. Hence, $z \in \Vt(\Delta)$.

Assume $w \in \Vt((z - C) \cap P_\MM)$. We will now show that $w \in \Vt(P_\MM)$ (or equivalently $w \in \Vt(\Delta)$), which will complete the proof of \eqref{vertpmplusc1}.

It suffices to show that if $w \not \in \Vt(\Delta)$, then $(T_{w,\MM'} -w)\cap C'$ contains a nonzero vector $v$. Indeed, such a $v$ yields a small linear perturbation $w\pm \epsilon v\in P_\MM$ such that $w \pm \epsilon v \in z-C$ as well, contradicting that $w$ is a vertex of $(z - C) \cap P_\MM$.

If $w \not \in \Vt(\Delta)$, then there exists $w_{i} \in (0,1)$. Then since the sum of the coordinates of $w$ along any $\sim_{\MM'}$-equivalence class is integral, there exists $i' \sim_{\MM'} i$ with $i \ne i'$ for which $w_{i'} \in (0,1)$. Similarly, since the sum of the coordinates of $z$ agrees with the sum of the coordinates of $w$ on any $\sim_{C'}$ equivalence class, the sum of the coordinates of $w$ on any $\sim_{C'}$ equivalence class is integral. Hence, there is an $i''\sim_{C'}i'$ with $i'' \neq i'$ for which $w_{i'} \in (0,1)$. Continuing in this fashion, we create a sequence of indices $i=i^1,i^2,\ldots$ for which $i^{2\beta-1}\sim_{C'}i^{2\beta}$ are distinct and $i^{2\beta}\sim_{\MM'}i^{2\beta+1}$ are distinct and $w_{i^\beta} \in (0,1)$ for all $\beta$. As the list of indices is finite, we have $i^{\alpha} = i^{\beta}$ for some minimal $\alpha < \beta$. We may assume that $\alpha$ and $\beta$ have the same parity; otherwise, we replace $\alpha$ by $\alpha + 1$ and $i^\beta$ by $i^{\alpha + 1}$. Now let \[v = \sum_{j = \alpha}^{\beta -1} (-1)^j e_{i^j}.\] By grouping the terms in two different ways, we have $v \in T_{w,\MM'}-w$ and $v \in C'$ as desired.

Hence $w \in \Vt(P_\MM)$, completing the proof of \eqref{vertpmplusc1}.

By combining terms, we may assume that we never have $i_\alpha=j_{\alpha'}$ in the equation. As $w,z \in \Vt(\Delta)$, this then forces $j_\alpha \in w \setminus z$ and $i_\alpha \in z \setminus w$. By Hall's marriage theorem, some subset of  $\{(j_\alpha, i_\alpha)\}$ yields a matching between $w\setminus z$ and $z \setminus w$, demonstrating \eqref{vertpmplusc2}.

Finally, we show that $(P_\MM+C)\cap \Delta$ is a matroid rank polytope. By \eqref{vertpmplusc2}, we can reach all of the vertices of $(P_\MM+C)\cap \Delta$ by starting with $P_\MM$ and applying the operation $P \mapsto (P+\mb{R}_{\ge 0}v)\cap \Delta$ once for each $v$ of the form $e_i-e_j \in C$. As it is clear that the result of applying this operation will keep $P$ contained inside $(P_\MM+C)\cap \Delta$, when we reach all of the vertices of $(P_\MM+C)\cap \Delta$ we will have exactly the polytope $(P_\MM+C)\cap \Delta$. Hence, it suffices to show that the property of being a matroid rank polytope is preserved under $P \mapsto (P+\mb{R}_{\ge 0}(e_i-e_j))\cap \Delta$. By \eqref{vertpmplusc1}, all vertices of $(P+\mb{R}_{\ge 0}(e_i-e_j))\cap \Delta$ are of the form $v+\lambda(e_i-e_j) \in \Vt(\Delta)$ for $v \in \Vt(P)$ and $\lambda \geq 0$. This forces either $\lambda=0$, in which case $v+\lambda(e_i-e_j) \in \Vt(P)$, or $\lambda=1$, in which case we get $(v \setminus j)\cup i$ whenever $j \in v$ and $i \not \in v$. This description yields a matroid by the basis exchange property.
\end{proof}

\begin{Lem} \label{lambdalambdaprime}
Let $\MM$ be a matroid and let $C$ be a matroidal cone such that $P_\MM = (P_\MM + C) \cap \Delta$. Let $v=e_i-e_j$ for some $i \ne j$. Let $z$ be a vertex of $(P_\MM+C+\mb{R}_{\ge 0}v)\cap \Delta$. Let $\lambda$ be the minimum nonnegative value such that $z - \lambda v \in P_\MM + C$, and let $\lambda'$ be the minimum value of $|z \setminus w|$ over all pairs $(w, \phi)$ such that $w \in \Vt(P_\MM)$ and $\phi: z \setminus w \to w \setminus z$ is a bijection with $e_k-e_{\phi(k)} \in C+\mb{R}_{\ge 0}v$ for all $k \in z \setminus w$. Then $\lambda = \lambda'$.
\end{Lem}

\begin{proof}
\Cref{horribleequivalence} guarantees that $\lambda'$ is well-defined. Observe that $z - w = \sum_{k \in z \setminus w} (e_k - e_{\phi(k)})$. For each $k \in z \setminus w$, we have either $e_k - e_{\phi(k)} \in C$ or $e_k - e_{\phi(k)} = v + (e_k - e_i) + (e_j - e_{\phi(k)}) \in v + C$. Summing over all $k$ yields $z - w \in \mu v + C$ for some $\mu \leq |z \setminus w| = \lambda'$. Then $z - \mu v \in w + C \subset P_\MM + C$, so $\lambda \leq \mu \leq \lambda'$.

It remains to show that $\lambda' \leq \lambda$. We have an equation
\begin{equation}\label{pluslambdav}
w+\sum_{\alpha} a_{\alpha}(e_{i_\alpha}-e_{j_\alpha})+\lambda v=z,
\end{equation} 
where $w \in P_\MM$ and $a_\alpha>0$ and $e_{i_\alpha}-e_{j_\alpha} \in C$ for all $\alpha$. Note that if $\lambda=0$, then $z\in P_\MM$ and we know that all coordinates of $z$ are integral so in fact $z \in \Vt(P_\MM)$ and the result is trivial. Hence, we may assume $\lambda>0$.

We may assume our equation is reduced in the sense that $j_{\alpha}\ne i_\beta$ for any $\alpha \ne \beta$. Let $C'=\sum_\alpha \mb{R}(e_{i_\alpha}-e_{j_\alpha})$. Clearly $e_i-e_j \not \in C'$ (nor its negation) as then we could use this to decrease $\lambda$ by perturbing the equation. Hence, by reordering the indices, we may assume without loss of generality that $j_\alpha < i_\alpha$ and $j < i$.

We may assume that $w \in P_\MM$ is lexicographically maximal satisfying an equation of the form \eqref{pluslambdav} with $j_\alpha < i_\alpha$ and $j < i$. 

We show first that $w\in \Vt(P_\MM)$. As in the proof of \Cref{horribleequivalence}, if not all coordinates of $w$ are integral, then $(T_{w,\MM}-w)\cap C''$ contains a line, where $C''$ is the cone generated by $\mb{R}(e_i-e_j)$ and $\mb{R}(e_{i_\alpha}-e_{j_\alpha})$ for all $\alpha$. If a nonzero multiple of $e_i-e_j$ is used in the expression for this line, then a small perturbation of $w$ along one of the two directions of the line decreases $\lambda$, contradicting the minimality of $\lambda$. Otherwise, $w$ may be perturbed in either direction, contradicting the lexicographic maximality of $w$. Hence, $w \in \Vt(P_\MM)$.

Suppose $\lambda$ is not integral. We have $i\sim_{C'}j$ as otherwise the sum of the coordinates of the left-hand side of \eqref{pluslambdav} in the $C'$-equivalence class containing $i$ is not an integer, contradicting the integrality of $z$. Thus we may create a perturbation of the expression which decreases $\lambda$, again a contradiction.
Hence $\lambda$ is integral.

Suppose $a(e_\ell-e_k)$ appears as one of the terms of the sum in \eqref{pluslambdav} with $k \ne i,j$ and $\ell \ne i,j$. Then for $0<\epsilon<a$, we have $w+\epsilon(e_l-e_k) \in P_M+C$ and $w+\epsilon(e_l-e_k) \not \in P_M=(P_M+C)\cap \Delta$ by the lexicographic maximality of $w$. Hence $w+\epsilon(e_l-e_k) \not \in \Delta$ for $\epsilon$ arbitrarily small. This implies either $w_l=1$ or $w_k=0$. No expression of the form $b(e_k-e_{k'})$ appears in the sum, so if $w_k=0$, then we have $w+\sum a_\alpha (e_{i_\alpha}-e_{j_\alpha})+\lambda(e_i-e_j)$ has negative $k$-coordinate, contradicting that $z \in \Delta$. Similarly, no expression of the form $b(e_{\ell'}-e_\ell)$ appears in the sum, so if $w_l=1$ then we have $w+\sum_\alpha a_\alpha (e_{i_\alpha}-e_{j_\alpha})+\lambda(e_i-e_j)$ has $\ell$-coordinate strictly greater than $1$, contradicting the assumption $z \in \Delta$. Hence no such term $a(e_\ell-e_k)$ appears in the sum.

Suppose now that $a(e_\ell-e_j)$ appears as one of the terms of the sum in \eqref{pluslambdav} with $\ell \ne i$. Then by the same argument as above, either $w_\ell=1$ or $w_j=0$ lest $w$ not be lexicographically critical. If $w_\ell=1$ we argue as above, and if $w_j=0$ then since there is no other term in the sum of the form $b(e_k-e_j)$ and $\lambda \ge 0$, we get that $z_j<0$, contradicting that $z \in \Delta$. A similar argument shows that $a(e_i-e_k)$ does not appear as one of the terms of the sum for $k \ne j$.

Hence, each $e_{i_\alpha}-e_{j_\alpha}$ is of the form $e_j-e_k$ or $e_\ell-e_i$. Since $w,z \in \Vt(\Delta)$, we conclude that each $a_\alpha=1$. Furthermore $w_k=1$ if a term of the form $e_j-e_k$ appears in the sum, and $w_\ell=0$ if a term of the form $e_\ell-e_i$ appears in the sum. Let $a$ be the number of terms of the form $e_j-e_k$ and let $b$ be the number of terms of the form $e_l-e_i$. Then since $z_i,z_j \in \{0,1\}$, we conclude that $\lambda-1 \le a,b \le \lambda+1$. But $a=\lambda+1$ implies $a>0$, so there is a term of the form $e_j-e_k$, and it also implies $w_j=0$, so $w+(e_j-e_k)\in (P_\MM+C)\cap \Delta=P_\MM$, contradicting the lexicographic maximality of $w$. Thus $a \leq \lambda$. Similarly, we have $b \le \lambda$. Combining the terms in the sum with $\lambda(e_i-e_j)$, we get $\lambda$ terms of the form $e_{i'}-e_{j'} \in C \cup (C+v)$ which yield a bijection $\phi:z \setminus w \to w \setminus z$ with $\phi(i')=j'$. Hence $\lambda' \le \lambda$ as desired.
\end{proof}
\section{Resolution via wonderful compactification}
\label{LL}
We recall the setup to blow up inductively an arrangement of subvarieties from \cite{LiLi}. In this section, a \emph{variety} is a reduced irreducible scheme of finite type over $K$.

\subsection{Arrangement of subvarieties}
\begin{Def}
[{\cite[Definition~2.1, Lemma~5.1]{LiLi}}] Let $X$ be a smooth variety and let $\mc{S}=\{S_i\}$ be a finite set of smooth nonempty proper subvarieties of $X$. We say that $\mc{S}$ is a \emph{simple arrangment of subvarieties} of $X$ if $\mc{S} \cup \{\emptyset\}$ is closed under scheme-theoretic intersection.
\end{Def}

\begin{Def}
[{\cite[Definition 2.2]{LiLi}}]
\label{BSD}
Let $\mc{S}$ be a simple arrangement of subvarieties of $X$. A subset $\mc{G}\subset \mc{S}$ is called a \emph{building set} of $\mc{S}$ if for all $S\in \mc{S}$, the minimal elements of $\mc{G}$ containing $S$ intersect transversely and their intersection is $S$. 

We say that a set $\mc{G}$ of subvarieties of $X$ is a \emph{building set} if it is a building set of some simple arrangement $\mc{S}$ of subvarieties of $X$. In this case $\mc{S}$ is the set of all proper nonempty intersections of elements of $\mc{G}$.
\end{Def}

\subsection{Wonderful compactification}

\begin{Def} [{\cite[Definition 2.12]{LiLi}}]
\label{iterBU}
Let $X$ be a smooth variety and let $\mc{G}$ be a building set of subvarieties of $X$. Order $\mc{G}=\{G_1,\ldots,G_n\}$ so that $G_i\subset G_j$ implies $i < j$, i.e. $G_i$ is a minimal element of $\{G_i,\ldots,G_n\}$ for all $i$. Let $X_0 = X$ and for each $k$ with $0 < k \leq n$, let $X_{k}$ be the blowup of $X_{k-1}$ along $\wt{G_k}$ (the iterated strict transform of $G_k$ in $X_{k-1}$). Define the \emph{wonderful compactification} $X_{\mc{G}}$ to be $X_n$, which by \Cref{wonderfulcompactificationclosure} below is independent of the ordering of $\mc{G}$.

By \cite[Proposition 2.8]{LiLi}, all the blowups carried out in \Cref{iterBU} have smooth centers, so $X_{\mc{G}}$ is smooth.

\end{Def}

\begin{Rem}\label{equivariantblowup}
If $X$ carries the action of an algebraic group $G$ and $\mc{G}$ is a building set containing only $G$-invariant subvarieties of $X$, then the action of $G$ on $X$ induces an action of $G$ on $X_\mc{G}$ and the blowdown map $X_{\mc{G}} \to X$ is $G$-equivariant.
\end{Rem}

\begin{Prop}
[{\cite[Proposition 2.13]{LiLi}}] \label{wonderfulcompactificationclosure}
\label{WCE}
The construction of $X_{\mc{G}}$ in \Cref{iterBU} agrees with the closure of the image
\begin{align*}
X\backslash\bigcup_{G\in \mc{G}}G\hookrightarrow \prod_{G\in\mc{G}}{{\rm Bl}_GX}.
\end{align*}
\end{Prop}

\subsection{Resolution of maps}
Proposition \ref{WCE} allows one to resolve a collection of rational maps with smooth base loci via iterated blowups.
\begin{Prop}
\label{RMBU}
Let $X\dashrightarrow X_i$ be a finite set of rational maps with simple base loci $B_i$. If $\mc{G}$ is a building set of subvarieties of $X$ that contains each $B_i$, then the wonderful compactification $X_{\mc{G}}\to X$ resolves the map $X\dashrightarrow \prod X_i$.
\end{Prop}

\begin{proof}
Consider the diagram
\begin{center}
\begin{tikzcd}
X_{\mc{G}} \arrow[r,hook] \arrow[d]& \prod_{G\in \mc{G}}{{\rm Bl}_{G}X} \arrow[d]\\
X \arrow[r,dashed] & \prod_{i}{X_i},
\end{tikzcd}
\end{center}
where the map $\prod_{G\in \mc{G}}{{\rm Bl}_{G}X}\to \prod_{i}{X_i}$ is given on each factor by the composition $\prod_{G\in \mc{G}}{{\rm Bl}_{G}X}\to {\rm Bl}_{B_i}X\to X_i$.
\end{proof}

\begin{Cor}
\label{RMM}
Let $M$ be a $d \times n$ matrix with no zero columns. Let $V = K^r$ and let $W = K^d$ and let $x_1, \ldots, x_n \in W$ be the columns of $M$. Let $\mc{G}$ be a building set of subvarieties of $\mb{P}^{(r+1) \times d-1} = \mb{P}(V \otimes W^*)$ that contains each subspace $\mb{P}(V \otimes x_i^{\perp})$. Then the wonderful compactification $B = (\mb{P}^{(r+1) \times d-1})_\mc{G}$ is a $GL_{r+1}$-equivariant resolution of $\mu_M$:
\begin{center}
\begin{tikzcd}
B \arrow[d,"\pi_M"] \arrow[dr, "\wt{\mu}_M"] & \\
\mb{P}^{(r+1)\times d-1} \arrow[r, dashed, "\mu_M"] & (\mb{P}^r)^n
\end{tikzcd}
\end{center}
\end{Cor}
\begin{proof}
Apply \Cref{RMBU} with $X = \mb{P}^{(r+1) \times d-1}$ and $X_i = \mb{P}^r$, where the map $X \dashrightarrow X_i$ is the multiplication map by $x_i$. The equivariance follows from \Cref{equivariantblowup}.
\end{proof}
\section{Tables of symbols and definitions}
\label{tableappendix}
For the benefit of the reader, we collect certain definitions in \Cref{wordtable} and symbols used throughout the paper in \Cref{symboltable}.

\begin{figure}[h]
    \centering
\begin{longtable}{l | p{8cm} | r}
Word & Short Description & Reference \\\hline
valuative & Function on matroids respecting rank polytope subdivision & \Cref{matroidpolytopedef} \\
additive & Function on matroids respecting rank polytope subdivision (ignores lower dimensional faces) & \Cref{matroidpolytopedef} \\
\begin{tabular}{l}Series parallel\\ matroid\end{tabular} &
 Matroid built from $\ast$ by series and parallel connection with general $1\times 2$ matrices & \Cref{serpardef}\\
Schubert matroid & Matroid of general matrix in Schubert cell & \Cref{SchubertDef}\\
Regular Subdivision &Subdivision of polytope induced by a ``lifting function'' on the vertices & \Cref{regularsubdef}
\\
Orb-limit & A matrix $N$ such that $\op{M(t)}$ contains $\op{N}$ when $t=0$& \Cref{orblimitdef}
\\
Cone-like matroid polytope &Matroid polytope arising from intersecting a matroidal cone with $\Delta_{d,n}$& \Cref{conelike}\\
(Positive) matroidal cone &Cone generated by vectors $e_i-e_j$ (resp. $i>j$)& \Cref{matroidalconedef}
\\
\end{longtable}
\caption{Table of certain definitions} \label{wordtable}
\end{figure}

\begin{figure}[h]
    \centering
\begin{longtable}{l | p{10cm} | r}
Symbol & Short Description & Reference \\\hline
 $\pt$ & variety consisting of a point & \\
  $\wt{v}$ & $\wt{v}\in \mb{P}(V)$ is the point corresponding to the vector $v\in V$ & \Cref{NotandConv} \\
  $\ast$ & General $1\times 1$ matrix & \Cref{asttau}\\
$\tau_{\leq k}$ & General matrix representing projection onto $K^k$& \Cref{asttau}\\
$P_M$ & Rank polytope of the matroid of $M$ & \Cref{matroidpolytopedef}\\
$I_M$ & Independence polytope of the matroid of $M$ & \Cref{matroidpolytopedef}\\
$P(M,N)$ & parallel connection of matrices $M$ and $N$ & \Cref{serpardefn}\\
$S(M,N)$ & series connection of matrices $M$ and $N$ & \Cref{serpardefn}\\
 $\operatorname{Vert}(P)$ & vertices of the polytope $P$ & \Cref{matroidpolytopedef} \\
$t_0,\ldots,t_r$ &Torus characters of diagonal matrices in $GL_{r+1}$  &\Cref{Chowring} \\
$A^\bullet_{GL_{r+1}}(\pt)$ & \begin{tabular}{l}$GL_{r+1}$-equivariant Chow ring of a point \\ Ring presentation: $\mathbb{Z}[t_0,\ldots,t_r]^{S_{r+1}}$ \end{tabular} &\Cref{Chowring} \\
$F(z)$&  Universal Leray Relation, Formula: $\prod_{i=0}^{r}(z+t_i)$&\Cref{Chowring} \\
$\overline{f}$ & Reduction of $f$ modulo $F$ & \Cref{reducedef}\\
$\alpha^\dagger$ & Kronecker dual class of $\alpha$ & \Cref{dualdef}
\\
$A^\bullet_{GL_{r+1}}(\mathbb{P}^n)$ &  \begin{tabular}{c} $GL_{r+1}$-equivariant Chow ring of $\mb{P}^r$ \\ $A^\bullet_{GL_{r+1}}(\pt)[H]/F(H)$ \end{tabular}&\Cref{Chowring} \\
$QH^{\bullet}_{GL_{r+1}}(\mb{P}^r)$ & \begin{tabular}{l} Small equivariant quantum cohomology ring of $\mathbb{P}^r$ \\ Ring Presentation: $\mb{Z}[t_0,\ldots,t_r][H][\hbar]/(F(H)=\hbar)$ \end{tabular} & \Cref{QCRD}\\
$[z^k]f$ & Coefficient of $z^k$ of $f$ & \Cref{QCRD}\\
$[F(z)^k]f$& $a_k$ from $f=a_0+a_1F+\cdots$ with each $a_i$ reduced mod $F$ & \Cref{QCRD}\\
$\mu_M$ & Matrix multiplication map $\mb{P}^{(r+1)\times d-1} \dashrightarrow (\mb{P}^r)^n$ & \Cref{Setup}
\\
$\op{M}$ & Generalized matrix orbit & \Cref{Setup}
\\
$\Gamma_M$ & Graph closure of $\mu_M$ & \Cref{Setup}\\
$[M]_{\hbar}$ & Operator $A^\bullet_{GL_{r+1}}(\mb{P}^r)^{\otimes n} \to QH^\bullet_{GL_{r+1}}(\mb{P}^r)$ & \Cref{Mhbardef}
\\
$[M]_{\hbar}^\dagger$ & Kronecker dual to $[M]_{\hbar}$ in $A^\bullet_{GL_{r+1}}(\mb{P}^r)^{\otimes n}\otimes QH^\bullet_{GL_{r+1}}(\mb{P}^r)$ & \Cref{Mhbardaggerdef}
\\
\end{longtable}
\caption{Table of symbols} \label{symboltable}
\end{figure}

\clearpage
\bibliographystyle{alpha}
\bibliography{references.bib}
\end{document}